\documentclass[draft]{amsart}
\usepackage{amsmath,amssymb,amsthm}
\usepackage{enumitem}
\usepackage{fixltx2e}
\usepackage{mathtools}
\usepackage{tikz}
\usepackage{tikz-cd}
\usepackage{xparse}
\usepackage{xspace}
\usepackage{mathtools}
\usepackage{dsfont}
\usepackage[all]{xy}
\usepackage{color}
\usepackage{verbatim}
\usepackage{hyperref}
\usepackage{mathrsfs}
\usepackage{microtype}
\mathtoolsset{showonlyrefs}
\usepackage{bm}

\usepackage{scalerel}
\usepackage{stackengine,wasysym}

\numberwithin{equation}{section}

\let\cal\mathcal
\def\Ascr{{\cal A}}

\def\Cscr{{\cal C}}
\def\Dscr{{\cal D}}
\def\Escr{{\cal E}}
\def\Fscr{{\cal F}}
\def\Gscr{{\cal G}}

\def\Kscr{{\cal K}}
\def\Lscr{{\cal L}}
\def\Mscr{{\cal M}}
\def\Nscr{{\cal N}}
\def\Oscr{{\cal O}}

\def\Qscr{{\cal Q}}
\def\Rscr{{\cal R}}
\def\Sscr{{\cal S}}
\def\Tscr{{\cal T}}

\def\Zscr{{\cal Z}}

\let\blb\mathbb

\def\GG{{\blb G}}

\def \ZZ{{\blb Z}}

\def \NN{{\blb N}}
\def \RR{{\blb R}}

\def\id{\text{id}}

\def\Lotimes{\overset{L}{\otimes}}

\def\Mod{\operatorname{Mod}}
\def\mod{\operatorname{mod}}
\def\Gr{\operatorname{Gr}}

\def\gr{\operatorname{gr}}
\def\Lie{\mathop{\text{Lie}}}

\def\rad{\operatorname {rad}}

\def\gr{\operatorname {gr}}
\def\Spec{\operatorname {Spec}}
\def\Rep{\operatorname {Rep}}
\def\GL{\operatorname {GL}}

\def\Ext{\operatorname {Ext}}
\def\uHom{\operatorname {\mathcal{H}\mathit{om}}}
\def\uEnd{\operatorname {\mathcal{E}\mathit{nd}}}
\def\End{\operatorname {End}}
\def\RHom{\operatorname {RHom}}

\def\Sl{\operatorname {SL}}
\def\Gl{\operatorname {GL}}

\def\im{\operatorname {im}}

\def\ker{\operatorname {ker}}

\def\Tor{\operatorname {Tor}}
\def\End{\operatorname {End}}

\def\id{{\operatorname {id}}}

\def\gldim{\operatorname {gl\,dim}}

\def\r{\rightarrow}
\def\d{\downarrow}

\def\GL{\operatorname {GL}}
\DeclareMathOperator{\Proj}{Proj}

\DeclareMathOperator{\Ind}{Ind}
\DeclareMathOperator{\RInd}{RInd}

\DeclareMathOperator{\coh}{coh}

\DeclareMathOperator{\Sym}{Sym}
\DeclareMathOperator{\Fr}{Fr}
\DeclareMathOperator{\codim}{codim}
\DeclareMathOperator{\St}{St}
\DeclareMathOperator{\uotimes}{\underline{\otimes}}
\DeclareMathOperator{\SL}{SL}
\DeclareMathOperator{\tH}{\widehat{H}}
\DeclareMathOperator{\HHH}{H}
\def\sscr{{\mathfrak{s}}}

\let\dirlim\injlim

\theoremstyle{definition}

\newtheorem{lemma}{Lemma}[section]
\newtheorem{proposition}[lemma]{Proposition}
\newtheorem{theorem}[lemma]{Theorem}
\newtheorem{corollary}[lemma]{Corollary}

\newtheorem{example}[lemma]{Example}
\newtheorem{definition}[lemma]{Definition}
\newtheorem{conjecture}[lemma]{Conjecture}

\newtheorem{remark}[lemma]{Remark}

\DeclareMathOperator\Hom{Hom}

\DeclareMathOperator{\Dist}{Dist}

\def\Xs{X^{\mathbf{s}}}
\def\tH{\hat{H}}
\def\refl{\operatorname{ref}}
\def\dur{*}
\makeatletter
\newcommand*\bigcdot{\mathpalette\bigcdot@{.5}}
\newcommand*\bigcdot@[2]{\mathbin{\vcenter{\hbox{\scalebox{#2}{$\m@th#1\bullet$}}}}}
\makeatother

\definecolor{ruta2}{rgb}{0.409, 0.459, 0.208}
\usepackage{xcolor}
\hypersetup{
    colorlinks,
    linkcolor={red!50!black},
    citecolor={blue!50!black},
    urlcolor={blue!80!black}
}

\mathchardef\mhyphen="2D

\newcounter{todocounter}
\DeclareDocumentCommand\addreference{g}{\stepcounter{todocounter}\todo[color = blue!30, fancyline]{\thetodocounter. Add reference\IfNoValueF{#1}{: #1}}\xspace}
\DeclareDocumentCommand\checkthis{g}{\stepcounter{todocounter}\todo[color = red!50, fancyline]{\thetodocounter. Check this\IfNoValueF{#1}{: #1}}\xspace}
\DeclareDocumentCommand\fixthis{g}{\stepcounter{todocounter}\todo[color = orange!50, fancyline]{\thetodocounter. Fix this\IfNoValueF{#1}{: #1}}\xspace}
\DeclareDocumentCommand\expand{g}{\stepcounter{todocounter}\todo[color = green!50, fancyline]{\thetodocounter. Expand\IfNoValueF{#1}{: #1}}\xspace}

\makeatletter
\newcommand{\mylabel}[2]{#2\def\@currentlabel{#2}\label{#1}}
\makeatother

\title[]{The Frobenius morphism in invariant theory}
\author{Theo Raedschelders}
\email[Theo Raedschelders]{Theo.Raedschelders@vub.be}
\thanks{The first author is supported by PhD fellowship 11K0216N with the Research Foundation Flanders (FWO) }
\author{\v{S}pela \v{S}penko}
\thanks{The second author is a FWO $[$PEGASUS$]^2$ Marie Sk\l odowska-Curie fellow at the Free University of Brussels
(funded by the European Union Horizon 2020 research and innovation
programme under the Marie Sk\l odowska-Curie grant agreement
No 665501 with the Research Foundation Flanders (FWO)). During part of this work she was also a postdoc with Sue Sierra at the University of Edinburgh}
\email[\v{S}pela \v{S}penko]{Spela.Spenko@vub.ac.be}
\address{Departement Wiskunde, Vrije Universiteit Brussel, 
Pleinlaan $2$, B-1050 Elsene}
\author{Michel Van den Bergh}
\email[Michel Van den Bergh]{michel.vandenbergh@uhasselt.be}
\address{Departement WNI, Universiteit Hasselt, Universitaire Campus \\
B-3590 Diepenbeek}
\thanks{The third author is a senior researcher at the Research Foundation Flanders (FWO).  While working on this project he was supported by
the FWO grant G0D8616N: ``Hochschild cohomology and deformation theory of triangulated categories''.}

\let\oldmarginpar\marginpar
\def\marginpar#1{\oldmarginpar{\tiny #1}}

\def\pdim{\operatorname{pdim}}
\begin{document}

\begin{abstract}
Let $R$ be the homogeneous coordinate ring of the Grassmannian $\GG=\Gr(2,n)$  defined over an algebraically closed field of characteristic $p>0$.
In this paper we  give a completely characteristic free  description of the decomposition  of
$R$, considered as a graded $R^p$-module, into indecomposables (``Frobenius summands''). 
As a corollary we obtain a similar decomposition 
 for the  Frobenius pushforward of the structure sheaf of $\GG$ and we obtain in particular that this pushforward is almost never a tilting bundle.
On the other hand we show that $R$ provides a ``noncommutative resolution'' for $R^p$ when $p\ge n-2$, generalizing a result known to be true for toric varieties.

In both the invariant theory and the geometric setting  we observe that if the characteristic is not too small the 
Frobenius summands
do not depend on the characteristic in a suitable sense. In the geometric setting this is an explicit version of a general result by Bezrukavnikov and Mirkovi\'c on Frobenius decompositions for partial flag varieities. 
 We are hopeful that it is an instance of a more
general ``$p$-uniformity'' principle.
\end{abstract}

\maketitle
\tableofcontents

\section{Introduction}
\subsection{The Frobenius morphism}
The Frobenius morphism is a  familiar tool in 
algebraic
geometry~\cite{deligneweil2}
and commutative algebra \cite{HH2}.
Anecdotal evidence suggests that it
also has a 
role to play in noncommutative geometry, but in this
case a unified theory remains to be developed.
As an example consider the following empirical observations:
\begin{enumerate}
\item[(i)] 
In many cases the pushforward $\Tscr:=\Fr_\ast(\Oscr_X)$ of the structure sheaf under the Frobenius morphism\footnote{Throughout we only consider the naive Frobenius morphism $\Fr:f\mapsto f^p$. In particular it is not linear over the ground field.} $\Fr:X\r X$ provides a canonical \emph{tilting bundle} on a smooth proper variety $X$, i.e.\ $\Tscr$ generates $D^b(\coh(X))$ and $\Ext^i_X(\Tscr,\Tscr)=0$ for $i>0$.
 In appropriate characteristics this is true for projective spaces, quadrics  \cite{langer2008d}, some partial flag varieties \cite{samokhin2010,samokhin2014frobenius,samokhin2017},
and also for many toric Fano threefolds \cite{BondalOW}. 
\item[(ii)] 
Similarly in many cases $\Mscr=\Fr_\ast(\Oscr_Y)$ provides a canonical \emph{noncommutative resolution}\footnote{In general we would not expect the resolution to be crepant. See e.g.\ \cite[Corollary 4.9]{Dao}.}  (NCR) (e.g.\ \cite[\S5]{BOO}) for a singular variety $Y$. In other words $\uEnd_Y(\Mscr)$ is a sheaf of algebras on $Y$ which is locally of finite global dimension. This is for example true for quadrics \cite{MR2531377} and for toric varieties
\cite[Proposition 1.3.6]{vspenko2015non}.
\end{enumerate}
One of the simplest possible situations with regard to (i,ii), which is covered only in very few cases by prior results, is the following:
let $\GG$ be the Grassmannian $\Gr(2,n)$ for $n\ge 4$ over an algebraically closed base field $k$ of characteristic $p>0$ 
and let $R=\Gamma_\ast(\GG)$ be its 
homogeneous coordinate ring. Let $Y=\Spec R$ be the (singular) affine cone over $\GG$. The following theorem, which will be a corollary of our results below, says that unfortunately we should not be overly optimistic with regard to the generality of (i).
\begin{theorem}
\label{th:mainth1} 
$\Fr_\ast \Oscr_\GG$ is \emph{not} a tilting bundle on $\GG$ unless $n=4$ and $p>3$ (see Corollary \ref{cor:langer}). 
\end{theorem}
When $n=4$ then $\GG$ is a four-dimensional projective quadric
hypersurface and then this theorem follows from the results in
\cite{langer2008d}. For $n=5$ a version has also been proved by David
Yang \cite{DY}
(private communication). 
In \cite{KanedaG2}
Kaneda constructs a counterexample for (i) given by a $G_2$-Grassmannian.

On the positive side, Kaneda constructs  in \cite{kanedaML} a full strong exceptional collection
on $\GG$ (different from the Kapranov exceptional collection) whose
elements are subquotients of $\Fr_*\Oscr_\GG$ for $p\gg 0$. We show 
in Corollary \ref{cor:kaneda}  that this
exceptional collection is a direct summand of
$\Fr_*\Oscr_\GG$ for $p\ge n$. 
So in this case, even though $\Fr_\ast\Oscr_{\GG}$ is not a tilting bundle, it still contains a direct summand which is 
a tilting bundle. 

Surprisingly the situation for  (ii) is much more clear cut than for (i) as we have the following result.
\begin{theorem}
\label{thm:ncr}
Assume $p\ge n-2$. Then $\Fr_\ast \Oscr_Y$ provides a noncommutative resolution for~$Y$. 
\end{theorem}
If $n=4$ then $Y$ is an affine quadric hypersurface and then this conjecture follows from \cite{MR2531377}. 
We will obtain Theorem \ref{thm:ncr} as a corollary to a result which does not require any characteristic assumptions
(Theorem \ref{thm:ncr1}).
\subsection{Main results}
It is classical that the homogeneous coordinate ring of $\Gr(2,n)$ can be described using invariant theory. Let $V$ be a two dimensional vector space,
 $W=V^{\oplus n}$ for $n\ge 4$, $G=\Sl(V)$, $S=\Sym(W)$, 
$R=S^G$. Then $\GG=\Proj R=\Gr(2,n)$. 

\medskip

Below we will give a complete description of the decomposition of $R$ as graded $R^p$-module and $\Fr_\ast\Oscr_{\GG}$ as coherent $\Oscr_\GG$-module, \emph{in a completely characteristic free manner}. 
We will refer to the indecomposable summands as ``Frobenius summands''.
This will lead in particular to a proof of Theorem \ref{th:mainth1}.

\medskip

It will also follow that the list of Frobenius summands is independent of  $p$ in a suitable sense. In the geometric setting this is a special case of a general (but less explicit) result by 
Bezrukavnikov and Mirkovi\'c \cite{BM} 
for partial flag varieties (see \S\ref{sec:partial}).

\medskip

We are hopeful that such ``$p$-uniformity'' (which can be considered as a sibling of the FFRT-property from \cite{MR1444312})  occurs much more generally.
In \S\ref{ref:punif}  we will discuss it in the context of invariant theory.
\begin{remark} \label{rem:cov}
In case $G$ is linearly reductive and $R=S^G$, $S=\Sym(W)$  it is shown in \cite{MR1444312} that the Frobenius summands of $R$ are given by so-called ``modules of covariants'' (i.e. $R$-modules
of the form $M(U):=(U\otimes_k S)^G$ for a $G$-representation $U$ \cite{Brion,van1991cohen}). One of the consequences of the results in the current paper is that this is no longer true
when $G$ is only reductive,  making it harder to identify candidate  Frobenius summands. 
\end{remark}

\subsubsection{Frobenius summands for  some invariant rings}
For $j\in \NN$ let $T(j)$ be the indecomposable tilting $G$-module with
highest weight $j$ (see \S\ref{sec:tiltingsl}) and define the $R$-modules of covariants $T\{j\}=M(T(j))=(T(j)\otimes_k S)^G$.  Then $T\{j\}$ is a 
Cohen-Macaulay $R$-module if and only if $0\le j\le n-3$ (see Proposition
\ref{prop:cm}). Note that this property does not hold for 
the simpler modules of covariants $M(S^j V)$
(see Example \ref{ex:nonsplit}). However for $p>j$ we have $T\{j\}=M(S^j V)$.

\medskip

As indicated in Remark  \ref{rem:cov} we will need more than just modules of covariants to describe 
all Frobenius summands. We now describe some other $R$-modules. It is easy to see that $T\{1\}_1=(V\otimes_k S)_1^G=k^n$.
Let $K^u\{1\}$ be the kernel of the corresponding morphism 
$R(-1)^n\r T\{1\}$ of graded $R$-modules and put $K\{1\}=K^u\{1\}(3)$, $K\{j\}=(\wedge^j K\{1\})^{\vee\vee}(-2j+2)$ (where $(-)^\vee$ denotes the dual $R$-module).  
In the body of the paper the $K\{j\}$ appear as $K^G_j$ and their properties are described
in \S\ref{sec:aux} combined with \eqref{eq:refl}. In particular they are non-zero indecomposable Cohen-Macaulay $R$-modules of rank ${n-2\choose j}$, 
which are non-free for $1\le j\le n-3$. 

\medskip
 
At this point we have the following collection of non-isomorphic indecomposable Cohen-Macaulay $R$-modules:
\[
\{T\{0\},T\{1\},\ldots,T\{n-3\},K\{1\},K\{2\}\ldots,K\{n-3\}\}.
\]
The following theorem states that these make up all the Frobenius summands, up to shift. 
If $M$ is a graded $R$-module then we write $M^{\Fr}$ for the corresponding graded~$R^p$-module via the isomorphism $\Fr:R\r R^p$. Note that the grading on $M^{\Fr}$ is concentrated in degrees $p\ZZ$.
\begin{theorem}[Corollaries \ref{cor:frob} and \ref{cor:indec}] \label{th:mainth2}
Up to multiplicities, the indecomposable summands of $R$ as $R^p$-module are
\begin{enumerate}
\item 
$T\{0\}^{\Fr}(-d)$, $d\in[0,2n(p-1)]$ even,
\item
for $1\leq j\leq n-3$, if $p \geq 1+\lceil j/(n-2-j)\rceil$
\[
T\{j\}^{\Fr}(-d),\, d\in[p(j+2)-2,p(2n-2-j)-2n+2],\, jp\equiv d \,(2),
\]
\item
for odd $1\leq j\leq n-3$ 
\[\quad\;\,
\begin{cases}
K\{j\}^{\Fr}(-d),\, d\in [p(j+3)-2,p(j+2+n-1)-2(n-1)]\, \text{even}& \text{if $n$ even},\\
K\{j\}^{\Fr}(-d),\, d\in [p(j+3)-2,p(j+2+n)-2n]\ \text{even}& \text{if $n$ odd},
\end{cases}
\]
\item
for even $1\leq j\leq n-3$
\[
\begin{cases}
K\{j\}^{\Fr}(-d),\, d\in [p(j+2),p(j+2+n)-2n]\, \text{even}& \text{if $n$ even},\\
K\{j\}^{\Fr}(-d),\, d\in [p(j+2),p(j+2+n-1)-2(n-1)]\, \text{even}& \text{if $n$ odd}.
\end{cases}
\]
\end{enumerate}
\end{theorem}
Our main result, Theorem \ref{thm:mainR} below, in fact also gives  the multiplicities of these summands
 but this is too cumbersome to state in the introduction. 

If $n=4$ then $R$ is a quadric and in that case our results are in agreement with Achinger's \cite[Theorem 2,3]{achinger2012frobenius}. 
In this case $T\{1\}(-1)$ and $K\{1\}(-2)$ correspond to the so-called ``spinor bundles" in loc. cit.

The following
corollary to Theorem \ref{th:mainth2} yields the ``$p$-uniformity'' property mentioned above
\begin{corollary} \label{cor:summandlist}
Forgetting the grading the Frobenius summands
are contained in 
\[
\{T\{0\}^{\Fr},T\{1\}^{\Fr},\ldots,T\{n-3\}^{\Fr},K\{1\}^{\Fr},K\{2\}^{\Fr},\ldots,K\{n-3\}^{\Fr}\}.
\]
 Moreover all these summands effectively appear if and only if $p\ge n-2$. 
\end{corollary}
\subsubsection{Frobenius summands for some Grassmannians}
From Theorem \ref{th:mainth2} one obtains in a straightforward way a corresponding result 
for $\Fr_\ast(\Oscr_{\GG})$. 

We write $\GG=\Gl_n/P$ where $P$ is the parabolic subgroup with Levi subgroup $\Gl_2\times \GL_{n-2}$. 
For $j\in \NN$ we upgrade $T(j)$ to a $\Gl_2$-representation with highest weight $(j,0)$ and
we let $\Tscr_j$ be the $\Gl_n$-equivariant vector bundle on $\GG$ whose fiber in $[P]$ is $T(j)$.
In particular $\Tscr_1=\Qscr$ is the universal quotient bundle $q:\Oscr_{\GG}^n\r \Qscr$. Moreover if $p>j$
then we also have $\Tscr_j=S^j\Qscr$.
Below put $\Oscr(1)=\wedge^2\Qscr$ and we let
$\Rscr:=\ker q$ be the universal  subbundle. 

We have $\GG=\Proj R$ but the grading on $R$ is doubled.
The correspondence between the graded $R$-modules
we considered in the previous section and the sheaves on $\GG$ defined above
is as follows (see \S\ref{sec:sheaf}) 
\begin{equation}
\label{eq:correspondence}
(T\{i\}(i))^{[2]}\leftrightarrow \Tscr_i\qquad
K\{j\}^{[2]}\leftrightarrow (\wedge^j\Rscr)(1),
\end{equation}
where $(-)^{[2]}$ denotes the 2-Veronese. Taking the $2p$-Veronese of the decomposition in 
Theorem \ref{th:mainth2} and using the correspondence \eqref{eq:correspondence} then yields
the following result.
\begin{theorem}[Corollary \ref{cor:indecgrass}] 
\label{mainth:grass} 
Up to multiplicity, the indecomposable summands of $\Fr_*\Oscr_\GG$ are 
\begin{enumerate}
\item
$\Oscr(-d)$, $d\in[0, n-\lceil n/p\rceil]$, 
\item
for $1\leq j\leq n-3$, if $p\geq 1+\lceil (j+1)/(n-2-j)\rceil$
\[
\Tscr_j(-d),\, d\in[j+1, (n-1)-\lceil (n-1)/p\rceil],\, 
\]
\item  
for odd $1\leq j\leq n-3$, if $p>2$
\[
\begin{cases}
\wedge^j\Rscr(-d+1),\, d\in [(j+3)/2,(j+1+n)/2-\lceil (n-1)/p\rceil]& \text{if $n$ even},\\
\wedge^j\Rscr(-d+1),\, d\in [(j+3)/2,(j+2+n)/2-\lceil n/p\rceil]& \text{if $n$ odd},
\end{cases}
\]
\item
for even $1\leq j\leq n-3$
\[
\begin{cases}
\wedge^j\Rscr(-d+1),\, d\in [(j+2)/2,(j+2+n)/2-\lceil n/p\rceil]& \text{if $n$ even},\\
\wedge^j\Rscr(-d+1),\, d\in [(j+2)/2,(j+1+n)/2-\lceil (n-1)/p\rceil]& \text{if $n$ odd}.
\end{cases}
\]
\end{enumerate}
\end{theorem}
Again Theorem \ref{thm:grass} gives us precise information on the multiplicities of the summands.
From Theorem \ref{mainth:grass} one now quickly obtains Theorem \ref{th:mainth1}. See \S\ref{sec:decompg}.
\begin{remark} Note that the summands listed in Theorem \ref{mainth:grass} no longer depend on~$p$ if $p>n$. This makes
explicit in the case of the Grassmannians $\Gr(2,n)$ the general  ``$p$-uniformity'' property for partial flag varieties
\cite{BM} (see \S\ref{sec:partial}).
\end{remark}

\subsection{Proof outline} 
\label{sec:outline}
Our approach to finding Frobenius summands is different from the methods used in some earlier papers (e.g. \cite{KanedaG2,kaneda2009some}  and  \cite{samokhin2014frobenius})
and 
also from~\cite{DY} which uses the results from \cite{BM,BMR} (see \S\ref{sec:partial}).

We first note that decomposing $R$ as $R^p$-module
 is equivalent to decomposing $S^{G_1}$ as
$(G^{(1)}, S^p)$-module, where $G_1$ denotes the first Frobenius
kernel of $G$, and $G^{(1)}=G/G_1$ (see \S\ref{sec:kernels}).

The next observation is that if we forget about the $G^{(1)}$-structure, then it suffices to construct a complex $M^{\bullet}$ of $S^p$-modules satisfying the following conditions:
\begin{enumerate}
\item $M^{\bullet}$ consists of free $S^p$-modules,
\item $Z^0(M^{\bullet})=S^{G_1}$,
\item the canonical truncation $\tau_{\geq 1}(M^{\bullet})$ is formal
\end{enumerate}
(see \S\ref{sec:syzygies}). It then follows that $S^{G_1}$ is equal to $\bigoplus_{j \geq 1} \Omega^{j+1}\HHH^j(M^{\bullet})$ as a graded $S^p$-module, ignoring free summands. 
 For the complex $M^\bullet$ we take
a complex representing $\RHom_{G_1}(k,S)$ (see \S\ref{calculations}) so that in particular 
$H^i(M^\bullet)=H^i(G_1,S)$.

Using the results in \S\ref{sec:syzygies} we show that a more
technical version of this procedure can be made to work if we also take
the $G^{(1)}$-structure into account.  To this end we have to develop some
homological algebra for equivariant modules over non-linearly
reductive groups (see \S\ref{sec:categoryGS}). We believe this
material may be interesting in its own right.

To prove an equivariant version of (3) we make use of the fusion category for $\SL_2$ as introduced by Gelfand and Kazhdan (see \S\ref{fusiongeneral},\S\ref{tensor}) and an ensuing formality result for the Tate cohomology of $G_1$ (see Theorem \ref{formalityy} and its proof).
\begin{remark}
The results in this paper, in particular, the summand lists given in Theorems \ref{th:mainth2} and \ref{mainth:grass} and the more precise
multiplicity formulas given in Theorems \ref{mainth} and \ref{thm:grass} have been subjected to a number of verifications using \cite{pyfrobenius}.
In particular we have verified that the Hilbert functions of the decompositions are correct in many examples. With regard to the NCR propery in Theorem \ref{thm:ncr} we have used Mathematica
to verify the numerical corollary that the inverse Hilbert series matrix is polynomial (see e.g.\ \cite[\S10.1]{vspenko2015non}) for $n=5,\ldots,12$.
\end{remark}

\subsection{The FFRT property and {\mathversion{bold}$p$}-uniformity for invariant rings}
\label{ref:punif}
In this section  we formulate some  conjectures with regard to Frobenius summands for invariant rings.
There is currently very little evidence for these conjectures but we hope that they will stimulate further research.

\medskip

Recall that the ``Finite F-representation type property'' was introduced in  \cite{MR1444312} and
 is concerned with the structure of $R$ as $R^{p^r}$-module for $r\ge 1$. It is most convenient 
to state it as a property of the $R$-modules $R^{1/p^r}$.
\begin{definition}
\label{FFRT}
Let $R=k+R_1+\cdots$ be a reduced finitely generated $\NN$-graded $k$-algebra. 
A \emph{higher Frobenius summand} is an indecomposable summand of $R^{1/p^r}$ 
for some $r\ge 1$. 
 We say that $R$ has  Finite F-representation type (FFRT) if
the number of isomorphism classes of higher Frobenius summands is finite.
\end{definition}
The FFRT property has been verified essentially in only the following two cases although it is expected to
hold much more generally.
\begin{enumerate}
\item If $R$ is Cohen-Macaulay and of finite representation type (e.g. if $R$ is a quadric)
then it obviously satisfies FFRT. 
\item One of the main results in \cite{MR1444312} states
that the FFRT property holds for invariant rings for \emph{linearly} reductive groups.
\end{enumerate}
In finite characteristic the class of linearly reductive groups is very small.\footnote{It consists
of extensions of finite groups by tori.}
Motivated by potential applications in the theory of rings of differential operators
(see \S\ref{sec:diff} 
for a short introduction) one is led to conjecture the following generalization in spirit of
(2).
\begin{conjecture}[FFRT for invariant rings]
\label{con:maincon}
Let $k$ be an algebraically closed field of characteristic zero and let $W$ be a finite dimensional $k$-representation for a reductive
group $G$. Put $R=\Sym(W)^G$. Then 
$R$ satisfies FFRT on a Zariski open dense set of fibers of any finitely
generated $\ZZ$-algebra  $A$
over which $G,W$ are defined.
\end{conjecture} 
This conjecture
is of great interest to us,
but this being said, in the current paper we will not consider higher Frobenius summands,
but only summands for the first Frobenius power and rather than considering the FFRT property we will consider
a related property stating that
these summands 
do not depend on the characteristic in an appropriate sense (``$p$-uniformity''). 
We will consider the  FFRT-property in \cite{FFRT2} which is in preparation.

\begin{conjecture}[$p$-uniformity for invariant rings]\label{punif}
Let $k$ be an algebraically closed field of characteristic zero and let $W$ be a finite dimensional $k$-representation for a reductive
group $G$. Put $R=\Sym(W)^G$. Then 
 there exists a finite set $(M_i)_{i=1}^m$ of indecomposable graded $R$-modules such that for any finitely generated $\ZZ$-algebra~$A$ over which $G,W$ are defined, there is a Zariski open dense set of $\Spec A$ on which the Frobenius summands
of the fibers of $R$ are obtained by reduction of shifts of $(M_i^{\Fr})_{i=1}^m$.
\end{conjecture}
\subsubsection{Relation with rings of differential operators for invariant rings}
\label{sec:diff}
For the benefit of the reader we give some more background from \cite{MR1444312}.
The authors in loc.\ cit.\ were motivated by the following folklore conjecture 
which was originally stated  as a question by Levasseur and Stafford in \cite{LSt}
\begin{conjecture}\label{con1}
Let $k$ be an algebraically closed field of characteristic zero and let $W$ be a finite dimensional $k$-representation for a reductive
group $G$. Put $R=\Sym(W)^G$. Then the ring of differential operators $D_k(R)$ is a simple ring.
\end{conjecture}
This conjecture is known to be true for finite groups \cite{Kantor}, tori \cite{MVdB}, odd $\Sl_2$-representations with
at least two summands \cite{VdB13} and various classical representations of classical groups \cite{LSt}. Otherwise it is wide open.
It implies the  Hochster-Roberts theorem that $R$ is Cohen-Macaulay via \cite[Theorem 6.2.5]{VdB5}.

In loc.\ cit.\ the authors initiated an attempt 
 to prove Conjecture \ref{con1} by reduction to finite characteristic where there is an intimate connection
between Frobenius summands and rings of differential operators via Yekutieli's formula \cite{Yek}
\[
D_k(R)=\dirlim_n \End_{R^{p^n}}(R).
\]
One of the main results of \cite{MR1444312} is that if $R$ is as in Definition \ref{FFRT}
and satisfies FFRT as well as another natural condition (strong F-regularity) then 
its ring of differential operators is simple.
Using this 
 they were able to prove a finite characteristic version of Conjecture \ref{con1} in case $G$ is linearly
reductive. Sadly when $\operatorname{char} k>0$ the class of linearly reductive groups 
is too small. Therefore one needs a version of Conjecture \ref{con:maincon} to have any
hope of proving Conjecture \ref{con1} via reduction to finite characteristic and the FFRT
property.
\subsection{\boldmath $p$-uniformity for partial flag varieties}
\label{sec:partial}
For the benefit of the reader we mention
a remarkable instance of $p$-uni\-formity which occurs in \cite{BM,BMR} and which this paper makes explicit for the Grassmannians $\Gr(2,n)$.

 If $G$ is a split reductive group over $\ZZ$ and $B\subset G$ is a Borel subgroup 
then it follows from \cite[\S1.6.8 (Example)] {BM} that the collection of indecomposable summands  of $\Fr_\ast(\Oscr_{G/B})$  (after reducing $G/B$ mod $p$) does not depend on $p$ for $p\gg 0$.
A similar result for the partial flag varieties $G/P$ follows by considering the pushforward for the map $G/B\r G/P$. 

This is a side result to the construction in \cite{BM,BMR} of a tilting bundle $\Escr$ on $T^\ast(G/B)$,\footnote{Since $T^\ast(G/B)$ is a crepant resolution
of singularities of the nilpotent cone $\Nscr$ in $\mathfrak{g}=\operatorname{Lie}(G)$, the endomorphism algebra of $\Escr$ yields in fact a noncommutative crepant resolution of $\Nscr$ in the sense of \cite{VdB32}.}  
defined over $\ZZ[1/h!]$,
with~$h$ being the Coxeter number,
corresponding to an ``exotic $t$-structure'' on $D^b(\coh(T^\ast(G/B)))$ such that $\Escr\vert_{G/B}$ has the same indecomposable summands as $\Fr_\ast(\Oscr_{G/B})$ after reduction mod $p$ for $p>h$.

To obtain the actual decomposition  of $\Fr_\ast(\Oscr_{G/B})$
one then needs to explicitly decompose~$\Escr|_{G/B}\mod p$ into indecomposables and this appears to be quite subtle.\footnote{Roman Bezrukavnikov explains to us that there is a close 
 connection (in spirit at least) between the regularity of such decompositions and the validity of Luzstig's conjectures. The latter are now known to be false for quite large $p$ \cite{GW}.} However for $p\gg 0$  standard base change arguments yield that  
the decomposition is independent of $p$ which implies at least the claimed $p$-uniformity.

Decompositions of $\Fr_\ast(\Oscr_{G/P})$ based on the methods in \cite{BM,BMR} occur in the appendix to \cite{BMR} for $\SL_3/B$ and in \cite{DY} for the Grassmannian $\Gr(2,5)$. It would be interesting to compare this approach to the very different one which is used in the current paper (see \S\ref{sec:outline}).

\section{Acknowledgement}
The authors thank Henning Haahr Andersen for useful comments on the invariant theory for Frobenius kernels.
They are  very grateful to Roman Bezrukavnikov and Alexander Samokhin for communicating their insights on Frobenius
decompositions for flag varieties. 
They thank  David Jordan for useful discussions on fusion categories, and heartily 
 acknowledge Martina Lanini for patient and joyful explanations. 

\section{Notation and conventions}
\label{sec:notation}
Below a $G$-module for an algebraic group $G$ is a comodule for $\Oscr(G)$. A $G$-representation
is a finite dimensional $G$-module. Some ``modules'' like tilting modules are finite dimensional
by definition and so they are in fact representations.

\medskip

If $S$ is a commutative $k$-algebra and  $M$ is an $S$-module then we
  write $M^{\Fr}$ for the pullback
of $M$ under the ring isomorphism $S^p\r S:f\mapsto f^{1/p}$, so
$S^{\Fr}=S^p$. If $S=\Sym(W)$ for a $G$-representation $W$ then 
$S^{\Fr}=\Sym(W^{\Fr})$ where $W^{\Fr}=W^{(1)}$ is the usual Frobenius twist
of $W$.

\medskip

We will often assume that we are in the setting of the introduction, i.e., we assume 
$G=\SL_2=\Sl(V)$, $\dim V=2$ and $S=\Sym(W)$, $R=S^G$,
where $W=V^{\oplus n}=F \otimes_k V$, $\dim F=n$ with $n \geq 4$. For brevity we sometimes call this the ``standard $\Sl_2$-setting''.

\medskip

In order to avoid notation collisions a torus is denoted by the letter $H$ instead of the more standard $T$, which is reserved  for tilting modules.

\medskip

If $\chi$ is a character of a torus $H$
 and $n\in \ZZ$ then we write $(n\chi)$ for the 1-dimensional $H$-representation
with character $\chi(-)^n$.

\medskip

We use $(-)$ both for the grading shift and as an indexing device for various standard types
of representations. This never leads to confusion. For example, $T(j)$ is the indecomposable
$\Sl_2$-tilting module with highest weight $j$ and $T(j)(l)$ is the graded $\Sl_2$-representation 
given by $T(j)$ located in degree $-l$.

\medskip

We use $(-)^\vee$ for the dual of a module, and $(-)^\dur$ for the dual of a  vector space.

\medskip

Sometimes auxiliary results are stated in  greater generality than strictly needed  in order to accommodate \cite{FFRT2}.

\section{Preliminaries}

\subsection{Preliminaries on modular representation theory}
\label{sec:prelim}
In this section we review some basic notions from modular
representation theory of reductive algebraic groups. If $G$ is an
algebraic group then $\Rep(G)$ denotes the category of  $G$-representations
and we write $G_r=\ker \Fr^{r}$ for the $r$'th Frobenius kernel. We have $G/G_r\cong G^{(r)}$
where $G^{(r)}$ is the $r$-Frobenius twist of $G$.

\medskip

Assume now that $G$ is a reductive algebraic group, with Borel subgroup $B$ and
maximal torus $H$.  By convention the roots of $B$ are the negative
roots. 
Write $X(H)$ for the weights of $H$ and $X(H)^+$ for the set of
dominant weights.  The set $X(H)$ is partially ordered by putting $\mu<\lambda$
whenever $\lambda-\mu$ is a sum of positive roots.

For $\lambda \in X(H)^+$, the induced, Weyl, and simple representations with highest weight $\lambda$ are defined as\footnote{In \cite{jantzen2007representations}  $\nabla(\lambda)$ and 
$\Delta(\lambda)$ are denoted by $H^0(\lambda)$ and $V(\lambda)$ respectively.}
\begin{align}
\nabla(\lambda) &= \Ind^G_{B}(\lambda), \\
\Delta(\lambda) &= \nabla(-w_0\lambda)^\dur,\\
L(\lambda)&=\operatorname{socle}(\nabla(\lambda))=\operatorname{top}(\Delta(\lambda)),
\end{align}
where $w_0$ is the longest Weyl group element. 
Moreover we have
\begin{proposition}
For $M=\nabla(\lambda)$ or $\Delta(\lambda)$ we have 

\begin{enumerate}
\item $(M:L(\lambda))=1$;
\item If $(M:L(\mu)) \neq 0$, then $\lambda \geq \mu$.
\end{enumerate}
\end{proposition}
\begin{proposition} \label{basicdeltanabla}
For $\lambda,\mu$ in $X(H)^+$:
\begin{enumerate}
\item \label{basicdeltanabla1} $\Hom_G(\Delta(\lambda),\nabla(\lambda))$
is $1$-dimensional;
\item \label{basicdeltanabla2} $\Ext_G^i(\Delta(\lambda),\nabla(\mu))=0$ for $i>0$ or $\lambda\neq \mu$;
\item \label{basicdeltanabla3} $\Delta(\lambda)$ and $\nabla(\lambda)$ are exceptional objects in $\Rep(G)$;
\item \label{basicdeltanabla4} $\Ext_G^\ast(\nabla(\lambda),\nabla(\mu))\neq 0\Rightarrow \mu\le \lambda$;
\item \label{basicdeltanabla5} $\Ext_G^\ast(\Delta(\lambda),\Delta(\mu))\neq 0\Rightarrow \mu\ge \lambda$.
\end{enumerate}
\end{proposition}
Let $\Fscr(\Delta)$, $\Fscr(\nabla)$ be the finite dimensional representations having respectively a $\Delta$ or a $\nabla$-filtration. 
Equivalently, $V\in {\Fscr}(\Delta)$ (resp. $V\in \Fscr(\nabla)$) if and only if $\Ext^{>0}_G(V,\nabla(\lambda))=0$ (resp. $\Ext^{>0}_G(\Delta(\lambda),V)=0$) for all $\lambda$ by \cite[Proposition II.4.16, 4.19]{jantzen2007representations}. 
The (possibly infinite-dimensional) $G$-modules $V$ which satisfy  
$\Ext^{>0}_G(\Delta(\lambda),V)=0$ (resp. 
$\Ext^{>0}_G(V,\nabla(\lambda))=0$) for all $\lambda$
are said to possess a ``good filtration'' (resp. Weyl filtration).  The multiplicity of $\nabla(\lambda)$ as a section in a good filtration of $X$ is independent of the choice of filtration and is denoted $(X:\nabla(\lambda))$. Similar notation is used for Weyl-filtrations. We will often use the following
\begin{proposition} \label{prop:exact:good}  $\Fscr(\nabla)$, $\Fscr(\Delta)$ are respectively closed under cokernels of injective maps and kernels of surjective maps. Moreover
$(-)^G$ is exact on the exact category $\Fscr(\nabla)$.
\end{proposition}

The following is a deep theorem. 

\begin{proposition}\cite[II.4.21]{MR1054234,jantzen2007representations}
\label{tensorgoodfiltration}
Modules with good filtration are closed under tensor product.
\end{proposition}
The objects
in $T\in \Fscr(\Delta)\cap \Fscr(\nabla)$ are called
tilting modules.
In particular, $\Ext^{>0}_{G}(T,T)=0$ by Proposition \ref{basicdeltanabla}.
Proposition \ref{tensorgoodfiltration} yields:
\begin{corollary}
\label{tensortilting}
If $T, T'$ are tilting modules, then so is $T \otimes_k T'$.
\end{corollary}

For every $\lambda\in X(H)^+$ there is a unique indecomposable tilting module $T(\lambda)$ with highest weight $\lambda$
and every tilting module is a sum of $T(\lambda)$. By~\cite[II.E.6]{jantzen2007representations}, for all $\lambda \in X(H)^+$ we have
\begin{equation}
\label{tiltingduality}
T(\lambda)^\dur \cong T(-w_0 \lambda).
\end{equation}

As usual let $\rho$ be half the sum of the positive roots. If $\rho\in X(H)^+$ (this can always be achieved by replacing $G$ by a finite cover) then the $r$-th Steinberg representation of $G$ is defined as $\St_r=L((p^r-1)\rho)$. According to Donkin in ~\cite{MR2384612} it ``exerts an enormous moderating influence over the representation theory of $G$''.

\begin{proposition}~\cite[II.3.18]{jantzen2007representations}
\label{steinberg}
The $r$-th Steinberg representation satisfies
$$
\St_r = T((p^r-1)\rho)=\nabla((p^r-1)\rho)=\Delta((p^r-1)\rho),
$$
it is self-dual: $\St_r^\dur = \St_r$, and $\St_r$ is projective and injective as a $G_r$-representation.
\end{proposition}

\subsection{The fusion category}
\label{fusiongeneral}
Remember~\cite{MR2183279} that a fusion category is a $k$-linear semisimple rigid monoidal category with finite dimensional Hom-spaces, finitely many isomorphism classes of simple objects, and simple unit object. 
We set $C=\{\lambda \in X(H) \otimes_{\ZZ} \RR | 0<\langle \lambda+\rho,\beta^{\vee}\rangle<p \text{ for all } \beta \in R^+ \}$, the fundamental alcove (see 
\cite[II.6.2(6)]{jantzen2007representations} for unexplained notation in this definition). 
Denote by $\Cscr$ the full subcategory of $\Rep(G)$ with objects the finite direct sums of $G$-representations of the form $\Delta(\mu)$, with $\mu \in C \cap X(H)^+$. For $V_1, V_2 \in \Cscr$, one shows~\cite{MR1176207} that
\begin{equation}
\label{eq:fusionproduct}
V_1 \otimes_k V_2\cong M \oplus T,
\end{equation}
where $M \in \Cscr$ and $T=\oplus_i T(\lambda_i), \lambda_i \notin C$. 

\begin{proposition}~\cite{MR1176207} 
The category $\Cscr$ can be given the structure of a fusion category $(\Cscr, \uotimes, k)$ with the property that
$
V_1 \uotimes V_2 \cong M$, where $M$ is as in \eqref{eq:fusionproduct}.
\end{proposition}

\subsection{Specific results for $\SL_2$}
Here we assume that $G=\SL_2=\Sl(V)$, $\dim V=2$. We identify $X(H)$ with $\ZZ$ and for clarity we 
often use the standard notations of divided powers $D^uV$ (resp. symmetric powers $S^uV$) instead of $\Delta(u)$ (resp. $\nabla(u)$). 
\subsubsection{Tilting modules}
\label{sec:tiltingsl}
By \eqref{tiltingduality}  tilting modules over $G=\SL_2$ are self-dual. For small $u$, the modules $T(u)$ are easy to understand.
\begin{proposition}
\label{dotyhenke}\cite[Lemma 1.1]{MR2143497} 
\begin{enumerate}
\item For $0\leq u \leq p-1$ we have $T(u)=L(u)=S^uV=D^u V$.
\item For $p \leq u \leq 2p-2$ the module $T(u)$ is uniserial and its unique composition series has the form $[L(2p-2-u),L(u),L(2p-2-u)]$. Moreover, $T(u)$ is a non-split extension of $D^{2p-2-u}V$ by $D^uV$ (or, dually, of $S^uV$ by $S^{2p-2-u}V$).
\end{enumerate}
\end{proposition}
Higher weight tilting modules can be understood via the following decomposition result.
\begin{proposition} \label{cor:simple} \cite[Lemma II.E.9]{jantzen2007representations}
Assume $u=u_0+pu_1$ where either $u_1=0$ and $0\le u_0\le p-1$ or $p-1\le u_0\le 2p-2$ and $u_1\ge 0$. Then
\[
T(u)\cong T(u_0)\otimes_k T(u_1)^{\Fr}.
\]
\end{proposition}
Applying this result recursively one obtains a precise expansion of the higher weight tilting modules as tensor products of Frobenius twists of $T(i)$, $0\leq i\leq 2p-2$. 
\begin{corollary} \label{prop:tiltclass}
Let $u \geq 0$ be written in the form $u=\sum_{i=0}^{k} u_ip^i$,  $p-1\leq u_i\leq 2p-2$ for $i < k$, and $0 \leq u_k\leq p-1$. With this notation 
\begin{equation}
\label{eq:decomptilting}
T(u)\cong\bigotimes_{i=0}^kT(u_i)^{\Fr^i}.
\end{equation} 
\end{corollary}
\begin{remark} The expansion of $u$ in the previous corollary becomes unique if we require $u_k<p-1$. 
\end{remark}
Corollary \ref{prop:tiltclass} may be used to obtain the following extension of Proposition \ref{dotyhenke} (e.g.\ by checking characters).
\begin{proposition}\cite[Lemma 6]{MR1866327}
\label{lem:erdmannhenke}
Let $T(u)\cong\bigotimes_{i=0}^kT(u_i)^{\Fr^i}$ be as in \eqref{eq:decomptilting}. Then $(T(u):S^vV) \neq 0$ if and only if $v=\sum_{j=0}^k v_jp^j$ where $v_j=u_j$ or $v_j=2p-2-u_j$ for $j<k$ and $v_k=u_k$. Moreover $(T(u):S^vV)\leq 1$.
\end{proposition}
\subsubsection{$G_1$-invariants and representations}
The following result of van der Kallen will be important in the sequel.
\begin{theorem}\cite[Corollary 2.2]{MR1202803}\label{thm:vandenkallen}
Let $M$ be a module with good filtration. Then $H^*(G_1,M)$ (as a $G^{(1)}$ module) has a good filtration. 
\end{theorem}

The theorem is specific to $\SL_2$. In loc. cit. van der Kallen gives a counterexample in case $G=\SL_3$ and $p=3$.

We will also need to understand the $G_1$-invariants for tilting modules.
\begin{proposition}\label{prop:G1inv}
If $T$ is a  tilting $G$-module then $T^{G_1}$ is a tilting $G^{(1)}$-module. More precisely: if we have as in Proposition \ref{cor:simple}:
$u=u_0+pu_1$ where either $u_1=0$ and $0\le u_0\le p-1$ or $p-1\le u_0\le 2p-2$ and $u_1\ge 0$ then
\begin{equation}
\label{lem:G1inv:3}
T(u)^{G_1}\cong 
\begin{cases}
k&\text{if $u=0$},\\
T(u_1)^{\Fr}&\text{if $u_0=2p-2$},\\
0&\text{otherwise}.
\end{cases}
\end{equation}
\end{proposition}
\begin{proof}
By Proposition \ref{cor:simple} we find
$$
T(u)^{G_1} \cong (T(u_0) \otimes_k T(u_1)^{\Fr})^{G_1}=T(u_0)^{G_1} \otimes_k T(u_1)^{\Fr},
$$
so we are reduced to the computation of $T(u_0)^{G_1}$ for $0 \leq u_0 \leq 2p-2$. For $0\le u_0\le p-1$ we have $T(u_0)=S^{u_0}V$ by Proposition \ref{dotyhenke}
and hence we may use  \eqref{lem:G1inv:1} below.

So assume $u_0\ge p$.
By Proposition \ref{dotyhenke} there is an exact sequence
$$
0 \to S^{2p-2-u_0}V \to T(u_0) \to S^{u_0}V \to 0.
$$
If $u_0 \neq 2p-2, p$ then we again use \eqref{lem:G1inv:1} below, after taking $G_1$-invariants. If $u_0=2p-2$ and $p=2$, then we compute directly that $T(2)^{G_1} = (V \otimes_k V)^{G_1} \cong k$. If $u_0=2p-2$ and $p>2$ then we use \eqref{lem:G1inv:1} once more. Finally, if $u_0=p >2$, then we use the exact sequence
$$
0 \to D^pV \to T(p) \to D^{p-2}V \to 0
$$
in combination with \eqref{lem:G1inv:2} below. 
\end{proof}
We have used the following technical results.
\begin{lemma}\label{lem:G1inv}
We have 
\begin{equation}\label{lem:G1inv:1}
(S^{u}V)^{G_1}\cong 
\begin{cases}
(S^kV)^{\Fr} & \text{if $u=kp$, $k\geq 0$,}\\
0&\text{otherwise},
\end{cases}
\end{equation}
and also 
\begin{equation}\label{lem:G1inv:2}
(D^iV)^{G_1} \cong 
\begin{cases}
k & \text{if $i=p=2$,}\\
0 & \text{if $i=p, p-2$ and $p>2$.}
\end{cases}
\end{equation}
\end{lemma}
\begin{proof}
The isomorphisms in \eqref{lem:G1inv:1} follow from \cite[Proposition 1.2]{MR1933877}. The ones in \eqref{lem:G1inv:2} from \cite[Lemma 3.6]{MR2294227}. 
\end{proof}
For use below we also record the following result. 
\begin{proposition}\label{cor:tiltingkernel}
If $T$ is an indecomposable tilting module over $G$ then
$T$ is either projective as $G_1$-representation or else $T$ is among the
$L(i)_{i=0,\ldots,p-2}$.
\end{proposition}
\begin{proof} This follows from \cite[II.E.8]{jantzen2007representations}. For simplicity we give a direct proof for $\Sl_2$.
By Proposition \ref{dotyhenke} we have $T(u)=L(u)$ for $u\le p-1$. Assume $u\ge p-1$. By Proposition \ref{steinberg} $T(p-1)=L(p-1)=\St_1$ is projective over $G_1$. By looking at highest
weights we see that $T(u)$ is a summand of $T(p-1) \otimes_k T(u-p+1)$. This yields what we want. 
\end{proof}
\subsubsection{Tensor products of simple representations}
\label{tensor}
Let us denote for $0\le b\le a \le p-1$
\[\small
\overline{L(a){\otimes} L(b)}=
\begin{cases}
L(a-b)\oplus L(a-b+2)\oplus\cdots \oplus L(a+b)&\text{if $a+b< p-1$,}\\
L(a-b)\oplus L(a-b+2)\oplus \cdots \oplus L(2p-4-a-b)&\text{otherwise,}
\end{cases}
\]
where the last sum is taken to be $0$ if $a-b > 2p-4-a-b$.

\begin{lemma}\cite[Lemma 1.3]{MR2143497}\label{rem:tnz}
Let $0\leq b\le a\leq p-1$. We have 
\begin{equation}\small\label{eq:L(a)tnzL(b)}
L(a){\otimes} L(b)=
\begin{cases}
\overline{L(a){\otimes} L(b)}&\text{if $a+b< p-1$,}\\
\overline{L(a){\otimes} L(b)} \oplus T(p)\oplus \cdots \oplus T(a+b)&
\text{if $a+b\ge p-1, a+b \equiv p\,(2)$,}\\
\overline{L(a){\otimes} L(b)}\oplus L(p-1)
\oplus 
\cdots \oplus T(a+b)
&\text{if $a+b\ge p-1, a+b \not\equiv p\,(2)$.}
\end{cases}
\end{equation}
\end{lemma}

\begin{remark}
Note that $L(p-1)=T(p-1)$ by Proposition \ref{steinberg}.
\end{remark}

Now one can describe the fusion tensor product 
(see \S\ref{fusiongeneral}) explicitly for $\Sl_2$. One may check that in this case the ``fundamental alcove'' is $[0,p-2]$.
Hence by  Lemma \ref{rem:tnz} and the definition of the fusion category:
\begin{proposition}\label{prop:fusiontnz}
For $G=\SL_2$, the indecomposable objects of $\Cscr$ are the simple representations $L(0), \ldots, L(p-2)$ and for $a \geq b$ we have $L(a)\underline{\otimes} L(b)=\overline{L(a)\otimes_k L(b)}$.
\end{proposition}
\begin{corollary}
\label{cor:invertible}
$L(p-2)\,\underline{\otimes}\, L(p-2)\cong L(0)\cong k$ in $\Cscr$. Thus in particular $L(p-2)$ is an invertible object. Moreover we have $L(p-2)\,\underline{\otimes}\, L(a)=L(p-2-a)$.
\end{corollary}
Recall that a Young tableau is called semi-standard if the entries
weakly increase along each row and strictly increase down each
column. The shape $(\lambda_1,\ldots ,\lambda_n)$ of a Young tableau
is the tuple of its row lengths, and the weight $(i_1,\dots,i_k)$
gives the number of times $i_c$ that each number $1\leq c\leq k$
appears in the tableau.  In characteristic zero, using Pieri's formula repeatedly
yields a decomposition
\begin{equation}
L(i_1)\otimes_k L(i_2)\otimes\cdots\otimes_k L(i_k) \cong \bigoplus_{n \geq 0} L(n)^{\oplus c(n)},
\end{equation}
where $c(n)$ is the number of semi-standard  tableaux of shape $(\lambda_1,\lambda_2)$ and weight $(i_1,\dots,i_k)$, satisfying $\lambda_1-\lambda_2=n$.

In the fusion category we have a similar decomposition but with fewer semi-standard  tableaux. 
Recall that a ``skew diagram'' $\lambda/\mu$ is a pair of partitions such that $\mu\subset \lambda$. Assume $\lambda$, $\mu$ have at most two non-zero rows. Then we define the width of $\lambda/\mu$ as $\lambda_1-\mu_2$.
\begin{proposition}
\label{prop:admissible}
For $0 \leq i_1, \ldots, i_k \leq p-2$, there is a decomposition
\begin{equation}
\label{prop:tensorfusion}
L(i_1)\underline{\otimes} L(i_2)\underline{\otimes}\cdots\underline{\otimes} L(i_k) \cong \bigoplus_{n \geq 0} L(n)^{\oplus c_p(n)},
\end{equation} 
where $c_p(n)$ is the number of semi-standard  tableaux of shape $(\lambda_1,\lambda_2)$ and weight $(i_1,\dots,i_k)$, satisfying $\lambda_1-\lambda_2=n$ such that for every $1 \leq c \leq k$, 
the boxes labeled by $c$ form a skew diagram of width $\leq p-2$. 
\end{proposition}
\begin{proof}
  Consider $k=2$ with $(a,b)=(i_1,i_2)$. Then the restriction is imposed by
  \eqref{eq:L(a)tnzL(b)}. The admissible tableaux are those in which
  the difference of the length of the rows lies in $[a-b,2p-4-a-b]$,
  which is equivalent to requiring that the first row is of length
  $\leq p-2$. By induction we obtain \eqref{prop:tensorfusion}.
\end{proof}
Below we will call semi-standard tableaux as in Proposition \ref{prop:admissible} ``$p$-admissible''.
\section{Frobenius summands and Frobenius kernels}
\label{sec:kernels}
In this paper we are concerned with computing Frobenius summands for certain invariant
rings for reductive groups. In Proposition \ref{SG1} below we observe
that this problem is intimately connected with the invariant theory of 
Frobenius kernels.

We first recall the following convenient definition from \cite{vspenko2015non} which identifies representations
which are not ``too small''.
\begin{definition}
\label{def:generic}
Let $G$  be a reductive group and let $W$ be a $G$-representation. 
Denote $X=W^\vee$. 
We say that $W$ is \emph{generic} if  
\begin{enumerate}
\item  $X$ contains a point with closed orbit and trivial stabilizer.
\item If $X^{\mathbf{s}}\subset X$ is the locus of points that satisfy (1) then $\codim
  (X-\Xs,X) \ge 2$.
\end{enumerate}
\end{definition}
In this section we assume throughout that $W$ is generic, $S=\Sym(W)$, $R=S^G$.
In particular $W^{\Fr}$ is a generic $G^{(1)}$-representation 
and $(S^p)^{G^{(1)}}=R^p$. We obtain as in \cite[Lemma 4.1.3]{vspenko2015non}  inverse
symmetric monoidal equivalences
\begin{equation}
\label{eq:refl}
\xymatrix{
\refl(R^p)\ar@/^1em/[rr]^{(S^p\otimes_{R^p}-)^{\vee\vee}}&&\refl(G^{(1)},S^p)\ar@/^1em/[ll]^{(-)^{G^{(1)}}}
}
\end{equation}
between the category of reflexive $R^p$-modules
and the category of $G^{(1)}$-equivariant reflexive $S^p$-modules. These are rigid symmetric monoidal categories (see \cite[\S4]{vspenko2015non}) when equipped with the modified tensor product $(-\otimes-)^{\vee\vee}$. 
\begin{proposition}\label{SG1} If $M$ is a reflexive $S$-module then
 $M^{G_1}$ is a reflexive $S^p$-module and moreover
$M^G$ and $M^{G_1}$ correspond to each other under the equivalences \eqref{eq:refl}.
\end{proposition}
\begin{proof}
We first verify that $M^{G_1}$ is reflexive.
Assume that $A$ is a normal noetherian domain and denote by $X_1(A)$ the set of height $1$ prime ideals in $A$.  
Then for a finitely generated torsion free $A$-module $M$ we have $M^{\vee\vee}=\bigcap_{q\in X_1(A)}M_q$.
Hence in our case:
\[
(M^{G_1})^{\vee\vee}=\bigcap_{q\in X_1(S^p)}(M^{G_1})_q\subset \left(\bigcap_{q\in X_1(S^p)}M_q\right)^{G_1}=M^{G_1},
\]
where the inclusion follows since $(M^{G_1})_q$ is $G_1$-invariant, and the last equality is a consequence of the fact $M$ is a reflexive $S^p$-module. 

The last claim follows from the obvious fact that $(M^{G_1})^{G^{(1)}}=M^G$.
\end{proof}
\begin{corollary} \label{cor:frob} The equivalences \eqref{eq:refl}
provide a 1-1 correspondence between the Frobenius summands of $R$ and the indecomposable summands of $S^{G_1}$ as $(G^{(1)},S^p)$-module.
\end{corollary}

\begin{definition}
\label{def:unimodular}
Let $G$ be an algebraic group and let $W$ be a $G$-representation of dimension $d$. We say that $W$ is \emph{unimodular} if $\wedge^dW \cong k$.
\end{definition}
We record the following autoduality result which will be used in \S\ref{sec:mainproof} 
(for $M=S$) and in \cite{FFRT2}.
\begin{proposition}
\label{canonical}
Assume that $G$ is connected reductive and that $W$ is a generic unimodular $G$-representation of dimension 
$d$ such that $S=\Sym(W)$ has a good filtration. Put $R=S^G$. Then $R$ is Gorenstein and moreover for a reflexive $S$-module $M$
there is an
isomorphism
\begin{equation}\label{Rdual}
\Hom_{R^p}(M^G,R^p) \cong (M^\vee)^G(d(p-1))
\end{equation}
of graded $R$-modules, and an isomorphism 
\begin{equation}\label{Sdual}
\Hom_{S^p}(M^{G_1},S^p) \cong (M^\vee)^{G_1}(d(p-1))
\end{equation}
of graded $(G^{(1)},S^{G_1})$-modules.
\end{proposition}
\begin{proof}
By \cite[Corollary 8.7]{MR1715590} (following \cite{MR860731}) $R$ is strongly $F$-regular (hence Cohen-Macaulay \cite[Theorem 4.2]{MR1273534}) and $\omega_R \cong \omega_S^G=R(-d)$, so $R$ is Gorenstein.  
Thus, $\omega_{R^p}\cong R^p(-dp)$. 
The adjunction formula $\omega_R \cong \Hom_{R^p}(R,\omega_{R^p})$ (as $R$-modules) now proves \eqref{Rdual} for $M=S$. In general we have
\begin{align*}
\Hom_{R^p}(M^G,R^p)&=\Hom_R(M^G,\Hom_{R^p}(R,R^p))\\
&\cong\Hom_R(M^G,R)(d(p-1))\\ 
&\cong (M^\vee)^G(d(p-1))
\end{align*}
where the last isomorphism follows from the $(G,S)$-variant of \eqref{eq:refl}. By applying $(-)^{G^{(1)}}$ to both sides of \eqref{Sdual}
we obtain \eqref{Sdual} from \eqref{Rdual} using \eqref{eq:refl}.
\end{proof}

\section{Statement of the main decomposition result}
\label{sec:maintechnical}
We assume that we are in the standard $\Sl_2$-setting (i.e. 
$G=\SL_2=\Sl(V)$, $\dim V=2$ and $S=\Sym(W)$, $R=S^G$,
for $W=V^{\oplus n}=F\otimes_k V$, $\dim F=n$ with $n \geq 4$). In this section we state our main decomposition result (Theorem \ref{finaldecomp} below) from which all our other results
are derived. The proof of Theorem \ref{finaldecomp} will be spread out over a number of sections and will be finished in \S\ref{sec:mainproof}.

We want to decompose $R$ into indecomposable $R^p$-modules. According
to Corollary \ref{cor:frob} above this is equivalent to decomposing
$S^{G_1}$ as $(G^{(1)},S^p)$-module. In Theorem \ref{finaldecomp} we
will describe this decomposition exactly. First we have to define some objects,
indexing sets and some associated notation.

\begin{lemma}\label{lem:mapKj}
Let $1\leq j\leq n$. 
There is a unique (up to scalar) nonzero $\SL_2\times \GL(F)$-equivariant map of graded $S$-modules
\[\wedge^jF \otimes_k S(-j) \xrightarrow{\alpha} V \otimes_k \wedge^{j-1}F\otimes S(-j+1).
\]
\end{lemma}
\begin{proof}
Note that such a map is determined in degree $j$ by a map 
\[\wedge^jF   \to V \otimes_k \wedge^{j-1}F\otimes V\otimes F.
\]
It is easy to see that $\dim \Hom_{\SL_2}(k,V\otimes_k V)=1$. 
By Pieri's formulas (which give filtrations in characteristic $p$ \cite{boffi1988universal}), there is an exact sequence 
\begin{equation}
0\r \wedge^jF\r\wedge^{j-1}F\otimes_k F\r \nabla([2,1^{j-2}])\r0,
\end{equation}  
and thus also $\dim \Hom_{\GL(F)}(\wedge^jF,\wedge^{j-1}F\otimes_k F)=1$ by Proposition \ref{basicdeltanabla}. 
\end{proof}
For $1\leq j \leq n-2$, let $K^u_j$ be the graded $S$-module which is the kernel of the map $\alpha$ above.
It will be convenient to put $K_j=K_j^u(j+2)$. It will follow from \eqref{biresolution0} and Corollary \ref{cor:kernel} below that $K_j$ is the ``normalization'' of $K_j^u$. Thus~$K_j$ is characterized by  $K_{j0}\neq 0$ and $K_{j,<0}=0$. 

\medskip

The $K_j$ will yield one type of summands of $S^{G_1}$. To be able to describe the other summands we have to introduce some more notation.
Below for $h\in \NN$ we
let $[h]$ be the set $\{0,\ldots,h\}$.

\medskip

Let $\Mscr$ be the set of $n+1$-tuples $(q,i_1,\ldots,i_n)\in \NN\times [p-1]^n$ for which $T(q)$ is a direct summand of $L(i_1) \otimes_k L(i_2)  
\otimes_k\allowbreak \cdots \otimes_k L(i_n)$.
If $t=(q,i_1,\ldots,i_n)\in \Mscr$ then we put $(q_t,i_{t1},\ldots,i_{tn})=(q,i_1,\ldots,i_n)$,
and we let $n_t$ be the multiplicity of the corresponding summand. In addition we put $d_t=\sum_j i_{tj}$.

Put
$
\Nscr=\{t\in \Mscr\mid q_t\le p-2\}
$. 
From Proposition \ref{cor:tiltingkernel} 
and the construction of the fusion category (see Proposition \ref{prop:fusiontnz})
 it follows 
that $\Nscr$ is the set of $n+1$-tuples
$(q,i_1,\ldots,i_n)\in [p-2]^{n+1}$ for which $L(q)$ is a direct summand
of $L(i_1) \uotimes L(i_2) \uotimes \cdots \uotimes L(i_n)$. Moreover the resulting multiplicity is also given by $n_t$.

If $t\in \Mscr$ then we denote by $\mathfrak{S}_t$ the set of sequences of signs $\sscr\in \{+,0,-\}^n$  satisfying  
$s_j=0$ if and only if $i_{tj}=p-1$. 
We put $d_t^{\sscr}=\sum_{j=1}^n i^{\sscr}_{tj}$, where 
$i^{\sscr}_{tj}=i_{tj}$ if $s_j\in\{-,0\}$ and $i^{\sscr}_{tj}=2p-2-i_{tj}$ if $s_j=+$.

Using these definitions we define the following graded tilting $G$-module:
\begin{equation}
\label{eq:Tdef}
T=\bigoplus_
{\begin{smallmatrix}t\in \Mscr,q_t=0,\\ \sscr=(-)_{\ell=1}^n\, \text{or}\,\sscr=(+)_{\ell=1}^n\end{smallmatrix}}
k(-d_t^{\sscr})^{\oplus n_t}\oplus \bigoplus_{\begin{smallmatrix}t\in \Mscr,q_t-2p+2\equiv 0\,(p),\\q_t\neq p-2,\sscr\in \mathfrak{S}_t\end{smallmatrix}} T((q_t-2p+2)/p)(-d_t^{\sscr})^{\oplus n_t}.
\end{equation}
For a graded representation  $U=\bigoplus_n U_n$  we put  $U^{\overline{\Fr}}=\bigoplus_n U_n^{\Fr}$ (in contrast to the more natural $U^{\Fr}=\bigoplus_n U^{\Fr}_{np}$).

\begin{theorem} 
\label{mainth} 
As graded $(G^{(1)},S^p)$-modules there is an isomorphism   
\begin{equation}
\label{finaldecomp}
S^{G_1} \cong  \bigoplus_{j=1}^{n-3} \bigoplus_{\begin{smallmatrix}t\in \Nscr, q_t=0,j\equiv 0\,(2)\text{ or }\\ q_t=p-2,j\equiv 1\,(2)\end{smallmatrix}} 
K^{\Fr}_j(-p(j+2)-d_t)^{\oplus n_t}\oplus T^{\overline{\Fr}} \otimes_k S^p,  
\end{equation}
where $T$ is the graded tilting $G$-module described in \eqref{eq:Tdef}.
\end{theorem}

\begin{remark} For $t\in \Nscr$ it follows from Proposition \ref{prop:admissible} that $n_t$ is the number of $p$-admissible semi-standard tableaux of weight $(i_{t1},\ldots,i_{tn})$
and shape $(\lambda_1,\lambda_2)$ such that $\lambda_1-\lambda_2=q_t$. It would be interesting to  have a  combinatorial model for~$n_t$ which works for arbitrary $t\in \Mscr$.
In any case the $(n_t)_{t\in\Mscr}$ are easy to determine 
by computer,
e.g.\ via the character formula in Proposition \ref{lem:erdmannhenke}.
 We refer to \cite{pyfrobenius} for a Python program
that explicitly computes the summands in \eqref{finaldecomp} and their multiplicities.
\end{remark}

\section{Equivariant modules for reductive groups}
\label{sec:categoryGS}
\subsection{Introduction}
Assume that $G$ is an algebraic group acting on a commutative $k$-algebra $S$. If $G$ is not
linearly reductive then the category of $(G,S)$-modules does not have enough
projectives in general. This makes it difficult to represent maps in the derived category
of $(G,S)$-modules as explicit maps between ``nice'' complexes.
In this section we will develop some tools to overcome this in certain
specific situations. Notably we exhibit the usefulness of certain types
of resolutions (see \S\ref{sec:representing}). However to construct such resolutions
we depend on a conjecture that we have been able to prove in our standard $\Sl_2$-setting (\S\ref{sec:notation}) but not in general (see \S\ref{sec:basic}).
\subsection{Some classes of projectives}
\label{sec:projectives}
Assume that $G$ is a a reductive group acting on a finitely generated (commutative) $k$-algebra $S$ with a good filtration. 
A finitely generated $(G,S)$-module $F$ which has a finite
ascending filtration $F^i$ such that $F^{i+1}/F^i$ is of the form
$L\otimes_k S$ with $L$ in $\Fscr(\nabla)$ will be called
$\nabla$-free.  Then $\Delta$-freeness is defined dually. A module which is
both $\nabla$-free and $\Delta$-free will be called tilt-free. A
direct summand of a $\nabla$-, $\Delta$- or tilt-free module will be
called $\nabla$-, $\Delta$- or tilt-projective respectively.

We will also consider the graded context, i.e. $S=k+S_1+S_2+\cdots$ is connected graded and the $G$-action respects the grading. In that case all modules and representations will be implicitly assumed to be graded.

\begin{proposition}
\label{prop:tiltfree}
Assume we are in the graded context. Then a tilt-free module is of the form $T \otimes_k S$, with $T$ a graded tilting module.
\end{proposition}

\begin{proof}
For a tilt-free module $M$, by definition there are two filtrations $F_1^nM$ and $F_2^nM$ as graded $(G,S)$-modules such that 
\begin{align}
\gr_{F_1}(M)^n&=\oplus_{i=1}^{s_n} X_{in} \otimes_k S(t_{in}), \\
\gr_{F_2}(M)^n&=\oplus_{i=1}^{u_n} Y_{in} \otimes_k S(v_{in})
\end{align}
for $X_{in} \in \Fscr(\nabla), Y_{in} \in \Fscr(\Delta)$, and suitable shifts $t_{in}$ and $v_{in}$. 
  We deduce that 
\begin{equation}
M/S_{>0}M=\bigoplus_n \bigoplus_{i=1}^{s_n} X_{in}(t_{in})^{a_i}=\bigoplus_n \bigoplus_{i=1}^{t_n} Y_{in}(v_{in})^{b_i} \in \Fscr(\nabla)\cap\Fscr(\Delta).
\end{equation}

We are thus reduced to showing that for a module $M$ with a filtration $F^i$ such that $F^i/F^{i+1}$ is of the form $T \otimes_k S$, for $T$ a graded tilting module, the filtration splits as graded $(G,S)$-modules. This follows from the more general fact that any short exact sequence of $(G,S)$-modules 
\[
0\r M\r M'\r T\otimes_k S\r 0
\]
with $T\in \Fscr(\Delta)$ and $M$ having a good filtration is split since by change of rings
\[
\Ext^1_{G,S}(T\otimes_kS, M)=\Ext^1_G(T,M)=0.\qedhere
\]
\end{proof}

\begin{proposition}
\label{prop:nablafree}
A graded $(G,S)$-module $F$ is 
\begin{enumerate}
\item $\nabla$-free iff $F \in \Fscr(\nabla)$ and $F$ is projective (hence free) as graded $S$-module,
\item $\Delta$-free iff $F^{\vee} \in \Fscr(\nabla)$ and $F$ is projective (hence free) as graded $S$-module,
\item tilt-free iff $F, F^{\vee} \in \Fscr(\nabla)$ and $F$ is projective (hence free) as graded $S$-module.
\end{enumerate}
\end{proposition}

\begin{proof}
Note that it is enough to prove (1). 
One direction follows by definition since~$S$ has a good filtration and the tensor product of modules with good filtration has a good filtration. 

For the other direction let us assume that $F_m=0$ for $m<-N$. 
Since $F\in \Fscr(\nabla)$ it follows that $F_{-N}\in \Fscr(\nabla)$. So $F_{-N}\otimes_k S(N)\in \Fscr(\nabla)$ is thus a graded $(G,S)$-submodule of $F$ which is a summand forgetting
the $G$-structure. It follows that both $F_{-N}\otimes_k S(N)\in \Fscr(\nabla)$ and $F'=F/(F_{-N}\otimes_k S(N))\in \Fscr(\nabla)$ are graded free $S$-modules.
We may now repeat the procedure with $F$ replaced by $F'$. 
\end{proof}

\subsection{Representing maps}
\label{sec:representing}
Let $\tilde{\Fscr}(\nabla)$ denote the closure of $\Fscr(\nabla)$ under
direct limits. It is well known that the injective $G$-modules live in $\tilde{\Fscr}(\nabla)$~\cite{MR611465}. Note that modules in $\tilde{\Fscr}(\nabla)$ have good filtration.

\begin{lemma}  \label{lem:ht}
Let $P^\bullet$ be a right bounded complex of
  $(G,S)$-modules consisting of $\Delta$-projectives. Let $M^\bullet$ be
  a left bounded acyclic complex in $\tilde{\Fscr}(\nabla)$.  Then any map $\phi:P^\bullet\r M^\bullet$  as complexes of $(G,S)$-modules is homotopic to zero.
\end{lemma}
\begin{proof} 
Write $Z^i$ for $Z^i(M^\bullet)$. Since $\Fscr(\nabla)$ is closed under cokernels of injective morphisms, $M^{\bullet}$ being left bounded and acyclic ensures that $Z^i$ has a good filtration. The requested homotopy is constructed in the usual manner. We will illustrate it on the following diagram.
\[
\xymatrix{
\dots\ar[r]&M^{i-3}\ar[r]^{d^{i-3}}&M^{i-2}\ar[r]^{d^{i-2}}&M^{i-1}\ar[r]^{d^{i-1}}&M^{i}\ar[r]^{d^i}&M^{i+1}\ar[r]^{d^{i+1}}&\dots\\
\dots\ar[r]&P^{i-3}\ar[r]_{d^{i-3}}\ar[u]|{\phi^{i-3}}&P^{i-2}\ar[r]_{d^{i-2}}\ar[u]|{\phi^{i-2}}\ar@{.>}[ul]|{h^{i-2}}&P^{i-1}\ar[r]_{d^{i-1}}\ar[u]|{\phi^{i-1}}\ar@{.>}[ul]|{h^{i-1}}&P^{i}\ar[r]_{d^{i}}\ar[u]|{\phi^{i}}\ar@{.>}[ul]|{h^i}&0\ar[u]
}
\]
 Note first that the image of $\phi^i$ lies in $Z^i$.
Using the exact sequence
\begin{equation}
\label{eq:exact}
0\r Z^{i-1}\r M^{i-1}\r Z^i\r 0
\end{equation}
and the fact that $\Ext^1_{G,S}(P^i,Z^{i-1})=0$ (as $Z^{i-1}$ has a good 
filtration and $P^i$ is $\Delta$-projective) we see it is possible to construct $h^i$ such that $\phi^i=d^{i-1}h^i$. Then $d^{i-1}(\phi^{i-1}-h^i d^{i-1})=0$ and hence $\im (\phi^{i-1}-h^i d^{i-1})\subset Z^{i-1}$. Hence using  \eqref{eq:exact} with $i$ replaced by $i-1$
we may construct $h^{i-1}$ such that $\phi^{i-1}-h^i d^{i-1}=d^{i-2}h^{i-1}$.
We now consider $\phi^{i-2}-h^{i-1}d^{i-2}$, etc\dots.
\end{proof}

\begin{proposition}
\label{prop:actual} 
Let $P^\bullet$ be a right bounded complex of $\Delta$-projectives. Let $M^\bullet$ be a left bounded complex in $\tilde{\Fscr}(\nabla)$. Then any map $\phi:P^\bullet\r M^\bullet$ in the derived category of $(G,S)$-modules can be represented by an actual map of complexes.
\end{proposition}

\begin{proof}
  Let $\theta:M^\bullet\r I^\bullet$ be a quasi-isomorphism\footnote{Note that $\mod(G,S)$ has enough injectives (see e.g. \cite[Corollary 1.1.9]{hashimoto}).} where $I^\bullet$
  is a left bounded complex of injectives. Then $\phi$ is represented
  by a map of complexes $\phi:P^\bullet\r I^\bullet$. Let $C^\bullet=\operatorname{cone}\theta$. Then $C^\bullet$ is acyclic.   Hence since for all $i$ we have
that 
$$
C^n=I^n\oplus M^{n+1}\in \tilde{\Fscr}(\nabla),
$$ 
the hypotheses of Lemma 
\ref{lem:ht} hold with $M^\bullet$ replaced by $C^\bullet$. Consider the following diagram in the
homotopy-category of $(G,S)$-modules (which is a triangulated category)
\[
\xymatrix{
M^\bullet\ar[r]^\theta& I^\bullet\ar[r]& C^\bullet\ar[r]&\\
&P^\bullet\ar[u]_{\phi}\ar@{.>}[ur]
}
\]
Since by Lemma \ref{lem:ht} the dotted map is zero in the homotopy category
we see that $\phi$ factors through $\theta$, up to homotopy.
\end{proof}

\subsection{The resolution conjecture}
\label{sec:basic}
To make progress we need a means to construct ``nice'' resolutions such as in \S\ref{sec:representing}. For this we need to assume
the following basic conjecture.
\begin{conjecture}[Resolution conjecture] 
\label{conj:stable}
 Let $M$ be a finitely generated $(G,S)$-module with a good
  filtration. Then there exists a tilting module $T$ for $G$
  and a surjective $(G,S)$-morphism $T\otimes_k S\r M$ whose
  kernel has a good filtration.
\end{conjecture}
We have been able to verify (the graded version of) this conjecture (see \S\ref{sec:veri} below)  in  our standard $\Sl_2$-setting (\S\ref{sec:notation}).  
\subsection{Stable syzygies}
\label{sec:stables}
In the non-equivariant setting it is well-known that syzygies are unique up to the addition of projective summands. In this subsection we develop a substitute for this in the equivariant setting. 

\begin{proposition} 
\label{prop:eq} 
Let $M$ be a $(G,S)$-module with  good
  filtration. If we have two surjections $T_i\otimes_k S\r M$, $i=1,2$
as in Conjecture \ref{conj:stable} with kernels $K_1, K_2$ then we have isomorphisms
\begin{equation}
\label{eq:stable}
K_1\oplus (T_2\otimes_k S)\cong K_2\oplus (T_1\otimes_k S).
\end{equation}
\end{proposition}
\begin{proof}
To prove \eqref{eq:stable} we use the standard pullback diagram
\[
\xymatrix{
&&0&0\\
0\ar[r]& K_1\ar[r]& T_1\otimes_k S\ar[u]\ar[r]& M\ar[r]\ar[u]&0\\
0\ar[r]& K_1\ar[r]\ar@{=}[u]& X\ar[r]\ar[u]& T_2\otimes_k S\ar[u]\ar[r]&0\\
&&K_2\ar[u]\ar@{=}[r]&K_2\ar[u]\\
&&0\ar[u]&0\ar[u]
}
\]
Since $K_1$, $K_2$ have good filtrations and $T_1$, $T_2$ have
$\Delta$-filtrations, the middle exact sequences must split (as we have $\Ext^1_{G,S}(T_i\otimes_k S,K_{j})=\Ext^1_G(T_i,K_j)$ by change of rings). This
yields the desired result.
\end{proof}

Let $T\otimes_k S\r M$ be a surjection as in Conjecture
\ref{conj:stable} and let $K$ be its kernel.  It makes sense to call $K$ a \emph{stable syzygy} of $M$ and write $K=\Omega_s M$. It follows from Proposition
\ref{prop:eq} that $\Omega_s M$ is well defined up to adding $T'\otimes_k S$ with $T'$ tilting.

Conjecture \ref{conj:stable} has an obvious graded version in which we assume that $T$ is a graded $G$-representation. 
In the graded setting we can be more specific as we have Krull-Schmidt.
Let
$K\cong K_1\oplus\cdots\oplus K_l$ be a decomposition into indecomposable graded-$(G,S)$-modules.
Then we define the stable syzygy $\Omega_s M$ of $M$ as the sum of the $K_i$ which are not of the
form $L\otimes_k S$ with $L$ graded tilting. It follows from \eqref{eq:stable} that in this case
$\Omega_s M$ is well defined up to isomorphism. So we obtain a presentation
\[
T'\otimes_K S\oplus \Omega_s M\r T\otimes_k S\r M\r 0
\]
with $T$, $T'$ tilting.

\subsection{Tilt-free resolutions}
\label{sec:tiltfree}
For use below it will be useful to define the concept of a tilt-free resolution
of a finitely generated $(G,S)$-module $M$ with good filtration as a complex $P^\bullet$ concentrated
in degree $\le 0$ such that 
\begin{enumerate}
\item $\HHH^0(P^\bullet)=M$.
\item
The $P^i$ are tilt-free for $i\le 0$.
\item For all $i\le 0$, $Z^i(P^\bullet)$ has a good filtration.
\end{enumerate}
The existence of a tilt-free resolution follows from Conjecture \ref{conj:stable}.
We say that a tilt-free resolution $P^\bullet$ is bounded if $P^i=0$
for $i\ll 0$. 

If $P^\bullet$ is a tilt-free resolution of $M$ then obviously we have
\[
Z^i(P^\bullet)=\Omega^{i+1}_s M
\]
 up to adding $T\otimes_k S$ with $T$ tilting. 
\begin{remark}
For bounded complexes the good filtration condition on the $Z^i(P^\bullet)$ is superfluous. Indeed, if $Z^m(P^\bullet)=P^m$ and $Z^l(P^\bullet)=0$ for $l<m$ then $Z^k(P^\bullet)$ has a good filtration by induction on $k\geq m$ as quotients of modules with a good filtration have a good filtration.
\end{remark}

\begin{lemma} \label{lem:bounded} Assume Conjecture \ref{conj:stable}.
If we are in the graded context, $S$ is a polynomial ring and $M$ is a finitely generated graded $(G,S)$-module with good filtration then $M$ has a bounded tilt-free resolution.
\end{lemma}
\begin{proof}
Let $P^\bullet$ be a tilt-free resolution of $M$. Since $S$ has finite global dimension, $Z=Z^i(P^\bullet)$ will be free for $i\ll 0$. Hence by Proposition \ref{prop:nablafree} $Z$ will be $\nabla$-free. Using the fact that objects in $\Fscr(\nabla)$  have a finite resolution by tilting modules (see \cite[Proposition A4.4]{MR1707336}), it is now easy to construct a finite tilt-free resolution of $Z$.
\end{proof}

\section{Decomposing modules using syzygies}
\label{sec:syzygies}
We keep the notations and assumptions of \S\ref{sec:categoryGS}.  
As a motivation for this section let us consider a (non-equivariant) complex $M^\bullet$ of free $S$-modules
\[
0\r M^0\r M^1\r\cdots\r M^r\r 0.
\]
If this complex only has cohomology in degrees $0$ and $r$ then one obviously has $H^0(M^\bullet)\cong\Omega^{r+1} H^r(M^\bullet)$ (up to free summands).
With a little work this can be generalized as in the following lemma:
\begin{lemma}
Assume we have complex $M^{\bullet}$ of free $S$-modules such that $\tau_{\geq 1} M^{\bullet}$ is formal. 
Then up to free summands, $Z^0(M^\bullet)\cong\bigoplus_{j \geq 1} \Omega^{j+1}\HHH^j(M^{\bullet})$, as $S$-modules.
\end{lemma}
Theorem \ref{prop:main} below shows that this result can be made to work in the equivariant case replacing syzygies by ``stable syzygies'' (see \S\ref{sec:stables}). 
Unfortunately we expect the reader to be immediately put off by the 
long list of obscure technical hypotheses this theorem needs. So she/he is  strongly encouraged to skip this section on first reading.

\begin{theorem} 
\label{prop:main} 
Assume that we are in the graded setting and assume in addition the following
\begin{enumerate}
\item \label{uu} The graded version of Conjecture \ref{conj:stable} holds for $(G,S)$.
\item \label{vv} $S$ is a graded polynomial ring.
\item \label{aa} $M^\bullet$ is a complex of finitely generated $\nabla$-free modules
  concentrated in degree $\ge 0$.
\item \label{bb} The $G$-representations $\HHH^i(M^\bullet)$ have good filtrations for all $i$.
\item \label{rr} The reflexive $(G,S)$-module $\HHH^0(M^\bullet)=Z^0(M^\bullet)$ has the property that $\Hom_S(\HHH^0(M^\bullet),S)$ has a good filtration.
\item \label{dd} $\tau_{\ge 1}M^\bullet$ is formal in the derived category of $(G,S)$-modules.
\end{enumerate}
Define $\tilde{\Omega}_s ^{j+1}\HHH^j(M^\bullet)$ as the $(G,S)$-module
obtained from $\Omega_s^{j+1}\HHH^j(M^\bullet)$ by deleting all summands which are $\nabla$-free.
Then we have as graded $(G,S)$-modules
\[
\HHH^0(M^\bullet)\cong T\otimes_k S\oplus \bigoplus_{j\ge 1} \tilde{\Omega}^{j+1}_s \HHH^j(M^{\bullet}),
\]
where $T$ is a graded tilting module.
\end{theorem}

\begin{remark} \label{lem:items}
To get a feel for the conditions in Theorem \ref{prop:main} note that (\ref{uu},\ref{vv}) do not depend on $M^\bullet$
and (\ref{bb},\ref{rr},\ref{dd}) depend only on the image of $M^\bullet$ in the derived category of
$(G,S)$-modules.
\end{remark}

\begin{proof}[Proof of Theorem \ref{prop:main}] Let $R^\bullet_j$ be bounded tilt-free resolutions of
  $\HHH^j(M^\bullet)$ (see \S\ref{sec:projectives}) and let $N^\bullet$ be
  the complex obtained by combining $M^\bullet$ with $R^\bullet_0$, so
  $$
  N^{\bullet}: 0 \r R_0^{\bullet} \r M_0 \r M_1 \r M_2 \r \cdots,
  $$
 with $M_0$ in degree $0$. Thus in the derived category we have $\tau_{\ge 1} M^\bullet\cong
  N^\bullet$ and hence by the resulting formality of $N^\bullet$ we
  obtain maps $\phi_j:R^\bullet_j [-j]\r N^\bullet$ in the derived category
  of graded $(G,S)$-modules.

Proposition \ref{prop:actual} now implies that $\phi_j$ is represented by an actual map of complexes of graded $(G,S)$-modules. Put $R^\bullet=\bigoplus_j R^\bullet_j[-j]$. Then we obtain a morphism of complexes $R^\bullet\r N^\bullet$ which is a quasi-isomorphism 
in degrees $\ge 1$.

So we have a quasi-isomorphism (using that $\sigma_{\ge 0} M^\bullet=\sigma_{\ge 0} N^\bullet$)
\begin{equation}
\label{truncation}
(0\r Z^0(R^\bullet)\r \sigma_{\ge 0} R^\bullet)\r (0\r Z^0(M^\bullet)\r \sigma_{\ge 0} M^\bullet).
\end{equation}
The cone of \eqref{truncation} is an acyclic complex
\[
0\r Z^0(R^\bullet)\r Z^0(M^\bullet)\oplus R^0\r M^0\oplus R^1\r M^1\oplus R^2\r\cdots
\]
The module $B:=\ker(M^0\oplus R^1\r M^1\oplus R^2)$ is projective hence free, forgetting the $G$-structure. This is clear if the cone is right bounded, and follows in general since $S$ has finite global dimension.\footnote{In the bounded case we argue by induction. The kernel of a right-most morphism $C_{m-1}\to C_m\to 0$ ($C$ is the cone) is projective since the short exact sequence splits as $C_m$ is projective. For the infinite case we take $M=\im(C_{m-1}\to C_m)$, $m=\gldim S=n$, and then use Schanuel's lemma.} Consider the exact sequence (using $Z^0(M)=\HHH^0(M)$)
\[
0\r Z^0(R^\bullet)\r \HHH^0(M^\bullet)\oplus R^0\r B\r 0.
\]
Since $Z^0(R^\bullet)$, $\HHH^0(M^\bullet)$,  $R^0$ all have good filtrations we deduce that $B$ has a good filtration. Hence by Proposition \ref{prop:nablafree} $B$ is $\nabla$-free. Choose a surjection $T\otimes_k S\r B$  with $T$ graded tilting, such that the kernel $K$ is $\nabla$-free. We get a pullback diagram
\[
\xymatrix{
&&0\ar[d]&0\ar[d]\\
&&K\ar@{=}[r]\ar[d] & K\ar[d]\\
0\ar[r] &Z^0(R^\bullet)\ar[r]\ar@{=}[d]& X\ar[d]\ar[r]& T\otimes_k S\ar[d]\ar[r]& 0\\
0\ar[r] &Z^0(R^\bullet)\ar[r]& \HHH^0(M^\bullet)\oplus R^0\ar[r]\ar[d]& B\ar[r]\ar[d]& 0\\
&&0&0
}
\]
Since $Z^0(R^\bullet)$ has a good filtration the middle horizontal sequence is split. So we have
\[
X\cong Z^0(R^\bullet)\oplus T\otimes_k S.
\]
The right most vertical sequence is split as $S$-modules, since $B$ is free. Hence so is the middle vertical sequence. Thus it remains exact after dualizing. We obtain an exact sequence
\begin{equation}
\label{eq:dualsequence}
0\r \HHH^0(M^\bullet)^\vee\oplus (R^0)^\vee\r X^\vee \r K^\vee \r 0.
\end{equation}
Now $K^\vee$ is $\Delta$-free, $\HHH^0(M^\bullet)^\vee$ has a good filtration by hypothesis and $(R^0)^\vee$ has a good filtration since it is tilt-free. We conclude
that \eqref{eq:dualsequence} is split, and hence so is the original exact sequence (by dualizing again and using that $\HHH^0(M^\bullet)$ is reflexive). We conclude
\[
X\cong K\oplus  \HHH^0(M^\bullet)\oplus R^0,
\]
and so
\begin{equation}
\label{eq:reaction}
Z^0(R^\bullet)\oplus T\otimes_k S\cong K\oplus  \HHH^0(M^\bullet)\oplus R^0.
\end{equation}
Since $Z^0(R_0^{\bullet})=R^0$ and $Z^0(R_j^{\bullet}[-j])=\Omega_s^{j+1}\HHH^j(M^{\bullet})$ for $j \geq 1$ we obtain by Krull-Schmidt
\[
 \HHH^0(M^\bullet)\cong F\oplus \bigoplus_{j\ge 1} \tilde{\Omega}_s ^{j+1}\HHH^j(M^\bullet)
\]
with $F$ being $\nabla$-free, since $K$ is $\nabla$-free. Because $\HHH^0(M)^\vee$ has a good filtration we deduce that $F^\vee$ also has a good filtration and hence $F$ is tilt-free (using Proposition \ref{prop:nablafree}), so by Proposition \ref{prop:tiltfree}
\[
\HHH^0(M^\bullet)\cong T'\otimes_k S\oplus \bigoplus_{j\ge 1} \tilde{\Omega}^{j+1}_s \HHH^j(M^{\bullet}),
\]
where $T'$ is a graded tilting module.
\end{proof}

\section{Verifying the resolution conjecture in the standard setting}
\label{sec:veri}
In order for us to be able to apply Theorem \ref{prop:main} we need that Conjecture \ref{conj:stable} holds in the standard $\Sl_2$-setting (\S\ref{sec:notation}).
We do this here. The following two lemmas are true for arbitrary reductive groups where the Steinberg representation exist.

\begin{lemma}
\label{lem:stable_good}  
Assume $L$ is a $G$-representation. Let $\lambda$ be in the interior of the dominant cone. 
Then $\nabla(r\lambda+\nu)\otimes_k L\in \Fscr(\nabla)$, $\Delta(r\lambda)\otimes_k L\in \Fscr(\Delta)$ for $\nu\in X(H)$, $r\gg 0$. In particular, $\St_r\otimes_k L$ is tilting for $r\gg 0$.
\end{lemma}

\begin{proof} 
It is sufficient to show that $\nabla(r\lambda+\nu)\otimes_k L$ has a good filtration for $r\gg 0$. The weights of the $B$-representation $(r\lambda+\nu)\otimes_k L$ are all in the interior of the dominant cone if $r$ is big enough, since $\lambda$ itself is. So the induced representation
$$
\nabla((r\lambda+\nu)\otimes_k L)=\nabla(r\lambda+\nu)\otimes_k L
$$
has a good filtration. This is similar to~\cite[Lemma 3.3]{MR2294227} (the method of proof is
  attributed to Donkin). 
\end{proof}

\begin{lemma} 
 \label{lem:quotient_tilting}
Let $L$ be a representation of $G$. Then there exists a surjection $T\r L$ such that $T$ is tilting and projective as $G_1$-representation.
\end{lemma}
\begin{proof} By Lemma \ref{lem:stable_good} $\St_r\otimes_k L$ is tilting for $r\gg 0$. Hence the same is true for $\St_r\otimes_k \St_r\otimes_k L$. Now the Steinberg representations
  are self dual and hence we have a surjection $\St_r\otimes_k \St_r\r
  k$. Tensoring this surjection with $L$ yields what we want, since $\St_r$ is projective as $G_1$-representation.
\end{proof}

The following lemmas are more specific to $G=\SL_2$. 

\begin{lemma} 
\label{lem:stable_DV} 
Let $G=\SL_2$. If $L$ is a $G$-representation then $\Ext^1_G(D^t V,L)$ is
zero for $t\gg 0$. 
\end{lemma}

\begin{proof} This follows from the fact that $\Ext^1_G(D^t V,L)=\Ext^1_G(k,S^t V \otimes_k L)$ and Lemma \ref{lem:stable_good}.
\end{proof}

\begin{lemma} 
\label{high_degree} 
Assume the standard $\SL_2$-setting. 
If $M$ is a finitely generated graded $(G,S)$-module then $M_t$ has a good filtration for $t\gg 0$.
\end{lemma}

\begin{proof} We may construct a resolution 
\[
0\r P\r L_t\otimes_k S\r \cdots L_1\otimes_k S\r L_0\otimes_k S\r M\r 0,
\]
where $L_i$ are graded $G$-representation and where $P$ is free
when forgetting the $G$-action. Then $P$ has a filtration $F$ such that
$F_{i+1}P/F_i P\cong L_i'\otimes_k S$ where the $L'_i$ are also graded
$G$-representations. It is now clearly sufficient to prove the lemma
with $M=L\otimes_k S$ for $L$ a $G$-representation.

We have $S_t\cong \oplus S^{s}V\otimes_k V'$ for $s\ge \lfloor
t/n\rfloor$ where $V'$ has a good filtration. It now suffices to
invoke Lemma \ref{lem:stable_good}.
\end{proof}

\begin{proposition} 
\label{prop:stable_syzygy}
Assume the standard $\SL_2$-setting. 
Then the graded version of Conjecture \ref{conj:stable} holds.
\end{proposition}

\begin{proof}
  It is easy to see that we must find a surjection $\phi: T\otimes_k S\r M$
such that every $G$-morphism $D^t V\r M$ factors through $\phi$.

  By Lemma \ref{lem:quotient_tilting} we may select a surjection
  $\hat{T}\otimes_k S\r M$ where $\hat{T}$ is a graded tilting module. Let
  $K$ be its kernel. From Lemmas~\ref{lem:stable_DV} and~\ref{high_degree} we obtain that there exist only a finite number of $(i,j)$ such that $\Ext^1_G(D^i V,K_j)\neq 0$. In other words there exist only a finite number of linearly independent graded $G$-morphisms
$\psi_{ij}:D^i V(-j)\r M$ that do not factor through $\hat{T}\otimes_k S\r M$ (we choose bases for the groups $\Ext^1_G(D^i V,K_j)\neq 0$).

Let $D^i V\hookrightarrow T(i)$ be the canonical embedding (whose cokernel has a Weyl filtration). Then the $\psi_{ij}$ lift to maps
$\psi'_i:(T(i))(-j)\r M$ (using the fact that $M$ has a good
filtration). It now suffices to define $T=\hat{T}\oplus \bigoplus_{ij} (T(i))(-j)$.
\end{proof}

\section{A formality result for $\Sl_2$} 
\label{calculations}
\subsection{Introduction}
We assume that we are in the standard $\Sl_2$-setting (\S\ref{sec:notation}).
In this section we prove the following formality result.

\begin{theorem}
\label{formalityy}
$\tau_{\geq 1} \RHom_{G_1}(k,S)$ is formal in the derived category of $(G^{(1)},S^p)$-modules.
\end{theorem}

We will identify $S$ with the polynomial ring $k[x_1,y_1,\ldots,x_n,y_n]$ with $x_i$, $y_i$ having weights $+\omega$ and $-\omega$ respectively for $\omega$ the fundamental weight of $G$. Put $S_+=S/(y_1,\ldots,y_n)$ and set\footnote{The $S$-action on $M_j$ follows 
from functoriality  since $\Ind^G_B$ is lax monoidal, so it maps $(B,S)$-modules to $\Ind^G_BS\cong S$-modules.}
\begin{equation}
\label{eq:Mjdef}
M_j=\Ind_B^G((j\omega)\otimes_k S_+)
\end{equation}

For $j\ge 1$ we will exhibit  isomorphisms of graded $(G^{(1)},S^p)$-modules 
\begin{equation}
\label{eq1subsub}
H^j(G_1,S)\cong \bigoplus_{\substack{t\in \Nscr,q_t=0, j \equiv 0 (2) \text{ or}\\ q_t=p-2, j \equiv 1 (2)}} M_{j}^{\Fr}(-d_t)^{\oplus n_t},
\end{equation} 
where $\Nscr$ was introduced in \S\ref{sec:maintechnical}.

\medskip

Before embarking upon the proofs of Theorem \ref{formalityy} and \eqref{eq1subsub}
let us first explain why these results are important to us.
\begin{itemize}
\item
Our grand aim is to decompose
$R=S^G$ into indecomposable $R^p$-modules. According to Corollary \ref{cor:frob} above this
is equivalent to decomposing $S^{G_1}$ as $(G^{(1)},S^p)$-module. To this end we will use
Theorem \ref{prop:main} (with $(G,S)$ replaced by $(G^{(1)},S^p)$). This theorem
has many ingredients but first and foremost we have to find a suitable complex $M^\bullet$
such that $H^0(M^\bullet)=S^{G_1}$. Now we have $H^0(\RHom_{G_1}(k,S))=S^{G_1}$ and this 
suggest that for $M^\bullet$ we should take a  complex 
representing $\RHom_{G_1}(k,S)$. This is only meaningful provided 
the condition \eqref{dd} from Theorem \ref{prop:main} holds (see Remark \ref{lem:items}).
This is precisely Theorem \ref{formalityy}.
\item The conclusion of Theorem \ref{formalityy} gives a decomposition of the $(G^{(1)},S^p)$-module $S^{G_1}=H^0(\RHom_{G_1}(k,S))$ in terms of the stable
syzygies of the cohomology groups $H^j(G_1,S)=H^j(\RHom_{G_1}(k,S))$ for $j\ge 1$. This is the reason why we need \eqref{eq1subsub}.
\end{itemize}

\medskip

To prove formality of $\tau_{\geq 1} \RHom_{G_1}(k,S)$, we will first
work on the level of $B_1$-modules, and then use the Andersen-Jantzen spectral
sequence (see Proposition \ref{AndersenJantzen}) to pass to $G_1$-modules.
We will use Tate cohomology for $B_1$ and $G_1$ for which we refer the reader to Appendix
\ref{sec:appA}.
\subsection{Computing \boldmath $\underline{\RHom}_{B_1}(k,S)$}\label{sec:computingRHom}
Let $P=\Sym(V)$. Then $S=P^{\otimes n}$. We write $P=k[x,y]$ where the weights of
$x,y$ are $+\omega, -\omega$. Then $\mathfrak{u}:=\Lie(U)$ (with $U:=\rad(B)$) acts by the derivation
$F:=y\partial/\partial x$, and $kF\cong \mathfrak{u}$.

For $i\in \NN$ let $L(i)$ be the simple $G$-representation with
highest weight $i\omega$. It will be convenient to consider $L(i)$ for $i=0,\ldots,p-1$
to be embedded in $P$ as the part of degree~$i$. Remember that in this case $L(i)\cong \nabla(i)\cong P_i$.

\begin{lemma} 
\label{B1cohomology0}
If $i\in \{0,\ldots, p-2\}$ then $\underline{\RHom}_{B_1}(k,L(i))$ is formal as object in $D(B^{(1)})$. Moreover we have as $B^{(1)}$-representations:
\begin{equation}
\label{TateB1formula}
\tH^j(B_1,L(i))=
\begin{cases}
0&\text{$i\neq 0,p-2$,}\\
(jp\omega)&\text{$i=0$ and $j \equiv 0 (2)$,}\\
(jp\omega)&\text{$i=p-2$ and $j \equiv 1 (2)$.}
\end{cases}
\end{equation}
\end{lemma}

\begin{proof}
  If $M$ is a $B_1$ representation then
  $\underline{\RHom}_{B_1}(k,M)=\underline{\RHom}_{U_1}(k,M)^{H_1}$  (since $H_1$ is linearly reductive) and $\underline{\RHom}_{U_1}(k,M)$ is computed by the
  complex
\begin{equation}
\label{complex}
\cdots \xrightarrow{F}M\otimes_k (-2(p-1)\omega)\xrightarrow{F^{p-1}} M\xrightarrow{F} M\otimes_k (2\omega)\xrightarrow{F^{p-1}} M\otimes_k (2p\omega)\xrightarrow{F}\cdots 
\end{equation}
with the second $M$ occurring in degree zero. 
Here we used a complete resolution of $k$ as a $\Dist(U_1)$-module, where $\Dist(U_1)=k[F]/(F^p)$, and $F$ has weight $-2\omega$ (see \cite[\S II.12]{jantzen2007representations}). We apply this with $M=L(i)$  for $i=0,\ldots,p-2$. If $p=2$ then $i=0$ and hence $F$ acts as zero so it is clear that~\eqref{complex} is formal. Moreover since $L(0)=k$ we immediately get~\eqref{TateB1formula}.

Assume $p>2$. Then the map $F^{p-1}$ acts as zero so that \eqref{complex} is
 the direct sum of small complexes
 \begin{equation}
( L(i)\otimes_k (2kp\omega)\xrightarrow{F} L(i)\otimes_k (2(1+kp)\omega))[-2k].
\end{equation}
The cohomology of this small complex is equal to $((2kp-i)\omega)$ in degree $2k$ and
$((2(kp+1)+i)\omega)$ in degree $2k+1$. These cannot both be non-zero 
after taking $H_1$-invariants, for otherwise $2kp-i \equiv_p 0 \equiv_p 2(kp+1)+i$ which is impossible since $p>2$. This proves the formality statement for $p>2$ and one checks that \eqref{TateB1formula} also holds.
\end{proof}

\begin{proposition}
\label{mainprop0}
The complex $\underline{\RHom}_{B_1}(k,S)$ is formal as object in the derived category $D(B^{(1)},\Gr(S^p))$ and moreover there are isomorphisms as graded $(B^{(1)},S^p)$-modules
\begin{equation}
\label{B1cohomology}
\begin{aligned}
\widehat{H}^j(B_1,S)&\cong\widehat{H}^j(B_1,M)\otimes_k S^p_+\\
&\cong\bigoplus_{t\in\Nscr} \hat{H}^j(B_1,L(q_t))\otimes_k S^p_+(-d_t)^{\oplus n_t}\\
&\cong\bigoplus_{\substack{t\in\Nscr, q_t=0, j \equiv 0 (2) \text{ or } \\ q_t=p-2, j \equiv 1 (2) }}(j p\omega) \otimes_k S^p_+(-d_t)^{\oplus n_t}
\end{aligned}
\end{equation}
with
$$
M=\bigoplus_{t\in \Nscr} L(q_t)(-d_t)^{\oplus n_t},
$$
where $\Nscr$ is as in \S\ref{sec:maintechnical}.
\end{proposition}

\begin{proof}
An easy computation shows there is an exact sequence of graded $(B,P^p)$-modules\footnote{To see this exact sequence one can make a first quadrant picture of the monomials $x^iy^j$.}
\begin{equation}
\label{finv1}
0\r \bigoplus_{j=0}^{p-1} L(p-1) \otimes_k  ky^j\otimes_k P^p(-j-(p-1))\r P\r \bigoplus_{j=0}^{p-2} L(j)\otimes_k P^p/y^p 
P^p(-j)\r 0.
\end{equation}
Now $L(p-1)$ is projective as $G_1$-representation and hence as $B_1$-representation, since $G_1/B_1$ is affine. Tensoring
\eqref{finv1} $n$ times with itself over $k$ we thus get a filtration
\begin{equation}
0=F_{-1}S\subset F_0S\subset F_1S\subset\cdots \subset F_nS=S
\end{equation} 
on $S$ as $(B,S^p)$-module such that $F_{i}S/F_{i-1}S$ is a
projective $B_1$-representation for $i<n$ and
\[
S/F_{n-1}S=\widetilde{M}\otimes_k S^p_+,
\]
where $\widetilde{M}$ is the graded representation
\[
\widetilde{M}=\bigoplus_{(i_1,\ldots,i_n)\in [p-2]^n} L(i_1)(-i_1) \otimes_k L(i_2)(-i_2) \otimes_k \cdots \otimes_k L(i_n)(-i_n).
\]
Then $M$ is a direct summand of $\widetilde{M}$ and the complement is projective as $B_1$-representation.
Hence we get 
\begin{equation}
\label{eq2}
\underline{\RHom}_{B_1}(k,S)=\underline{\RHom}_{B_1}(k,M\otimes_k S^p_+)
=\underline{\RHom}_{B_1}(k,M)\otimes_k S^p_+
\end{equation}
since $B_1$ acts trivially on $S^p_+$. By Lemma \ref{B1cohomology0} $\underline{\RHom}_{B_1}(k,M)$ is formal in $D(B^{(1)})$. This proves that $\underline{\RHom}_{B_1}(k,S)$ is formal. 

Taking cohomology in \eqref{eq2} we also get the first line of \eqref{B1cohomology}. The second 
line is by definition and the third line follows from Lemma \ref{B1cohomology0}.
\end{proof}

\subsection{Computing \boldmath $\underline{\RHom}_{G_1}(k,S)$.}
\label{sec:standard}
\begin{proposition}
\label{mainprop}
We
have in $D(G^{(1)},\Gr(S^p))$
\begin{equation}
\label{eq3}
\underline{\RHom}_{G_1}(k,S) \cong \bigoplus_{j \in \ZZ} R\Ind_{B^{(1)}}^{G^{(1)}} \widehat{H}^j(B_1,S)[-j].
\end{equation}
Moreover there are isomorphisms of graded $(G^{(1)},S^p)$-modules for $j\ge 0$ 
\begin{equation}
\label{eq1sub}
R\Ind_{B^{(1)}}^{G^{(1)}} \widehat{H}^j(B_1,S)=\Ind_{B^{(1)}}^{G^{(1)}} \widehat{H}^j(B_1,S)=\widehat{H}^j(G_1,S).
\end{equation}
Finally
\begin{equation}
\label{eq1}
\widehat{H}^j(G_1,S)\cong \bigoplus_{\substack{t\in \Nscr, q_t=0, j \equiv 0 (2) \text{ or}\\ q_t=p-2, j \equiv 1 (2)}} M_{j}^{\Fr}(-d_t)^{\oplus n_t},
\end{equation} 
where $\Nscr$ is as in \S\ref{sec:maintechnical}.
\end{proposition}

\begin{proof}
The isomorphism \eqref{eq3} is a direct consequence of Proposition \ref{AndersenJantzen} and the formality statement in Proposition \ref{mainprop0}. We get
\begin{align*}
\hat{H}^j(G_1,S)&\cong \bigoplus_{k\in \ZZ}H^j(R\Ind_{B^{(1)}}^{G^{(1)}} \widehat{H}^k(B_1,S)[-k])\\
&\cong \bigoplus_{k\in \ZZ}R^{j-k} \Ind_{B^{(1)}}^{G^{(1)}} \widehat{H}^k(B_1,S)\\
&\cong  \Ind_{B^{(1)}}^{G^{(1)}} \widehat{H}^j(B_1,S)\oplus R^1 \Ind_{B^{(1)}}^{G^{(1)}} \widehat{H}^{j-1}(B_1,S).
\end{align*}
From \eqref{B1cohomology} we see that for $j \geq -1$
\begin{equation}
R^1\Ind_{B^{(1)}}^{G^{(1)}}(\widehat{H}^{j}(B_1,S))=\bigoplus_{\substack{t\in\Nscr,q_t=0, j \equiv 0 (2) \text{ or}\\ q_t=p-2, j \equiv 1 (2)}} R^1\Ind_{B^{(1)}}^{G^{(1)}}((jp\omega) \otimes_k S^p_+(-d_t)^{\oplus n_t})=0,
\end{equation} 
since the weights of $S^p_+$ are $\ge 0$ and $R^1\Ind_{B}^{G}(l\omega)=0$ for $l \geq -1$. 
This implies \eqref{eq1sub}. Finally \eqref{eq1} follows by invoking \eqref{B1cohomology} again for the
computation of $\widehat{H}^{j}(B_1,S)$.
\end{proof}

\begin{proof}[Proof of Theorem \ref{formalityy} and \eqref{eq1subsub}]
By \eqref{eq:nontate} below we see that the canonical morphism  $\RHom_{G_1}(k,S)\r \underline{\RHom}_{G_1}(k,S)$ in the derived category of $(G^{(1)},S^p)$-mod\-ules becomes a quasi-isomorphism after applying
$\tau_{\ge 1}$.

To prove that $\tau_{\geq 1} \underline{\RHom}_{G_1}(k,S)$ is formal, it suffices to construct morphisms 
$$
\underline{\RHom}_{G_1}(k,S) \to \tH^j(G_1,S)[-j]
$$
for $j \geq 1$ in the derived category of $(G^{(1)},S^p)$-modules inducing isomorphisms on $\HHH^j(-)$. We use 
\begin{equation}
\label{comp}
\underline{\RHom}_{G_1}(k,S) \xrightarrow[\eqref{eq3}]{\text{projection}} R\Ind^{G^{(1)}}_{B^{(1)}}\tH^j(B_1,S)[-j] \xrightarrow[\eqref{eq1sub}]{\cong} \tH^j(G_1,S)[-j].
\end{equation}
Now \eqref{eq1subsub} follows from \eqref{eq1} together with \eqref{eq:nontate} below.
\end{proof}

\section{Some standard modules and their resolutions}
\label{sec:resol}
In this section we exhibit a connection between the modules $K_j$ that appeared in our main decomposition result (Theorem \ref{finaldecomp}) and the $M_j$ that appeared in the computation of the $G_1$-cohomology of $S$ in \S\ref{sec:standard}.

\begin{proposition}\label{prop:resMj}
The $(G,\GL(F),S)$-module $M_j$ has a $G\times \GL(F)$-equivariant graded resolution of the form 
\begin{multline}
\label{biresolution0}
0\r D^{n-j-2}V\otimes_k\wedge^n F\otimes_k S(-n)\r \cdots \r 
D^1 V\otimes_k \wedge^{j+3}F\otimes_k S(-j-3)\r\\ \wedge^{j+2}F \otimes_k S(-j-2)\r \wedge^j F\otimes_k S(-j)
\r
\\
 V\otimes_k \wedge^{j-1} F\otimes_k S(-j+1)\r \cdots \r S^{j-1} V\otimes_k F\otimes_k S(-1) \r S^jV\otimes_k S\r  M_j\r 0
\end{multline}
if $0\le j\le n-2$, 
and if $j\ge n-1$:
\begin{multline}
\label{biresolution1}
0
\r \wedge^j F\otimes_k S(-j)
\r
 V\otimes_k \wedge^{j-1} F\otimes_k S(-j+1)\r\\ \cdots \r S^{j-1} V\otimes_k F\otimes_k S(-1) \r S^jV\otimes_k S\r  M_j\r 0.
\end{multline}
\end{proposition}

\begin{proof}
In characteristic $0$, these resolutions were constructed in~\cite{vspenko2015non} (see also \cite[Section 5]{weyman2003cohomology}).
To make sure that we have the same complexes in characteristic $p>0$, we recall the construction, which in the current situation is particularly easy.

Let $K=(-\omega)\otimes_k F$ be the subspace of $W$ spanned by the negative weight vectors. We have the corresponding $B$-equivariant Koszul resolution of $S_+=\Sym (W/K)$, which remains exact after tensoring with $(j\omega)$:
\begin{multline}
\label{biresolution}
0\r ((j-n)\omega) \otimes_k \wedge^n F\otimes_k S(-n)\r \cdots \r
(-3\omega)\otimes_k \wedge^{j+3}F\otimes_k S(-j-3)\r\\ 
(-2\omega)\otimes_k \wedge^{j+2}F\otimes_k S(-j-2)\r 
(- \omega)\otimes_k \wedge^{j+1}F\otimes_k S(-j-1)\r
\wedge^{j}F \otimes_k S(-j)\r\\
(\omega)\otimes_k\wedge^{j-1} F\otimes_k S(-j+1)\r \cdots \r ((j-1)\omega) \otimes_k F\otimes_k S(-1) \r (j\omega)\otimes_k S\r\\  (j\omega)\otimes_k S_+\r 0.
\end{multline}

Splitting \eqref{biresolution} into short exact sequences 
\begin{equation}\label{eq:sesK}
0\r L_i\r ((j-i)\omega)\otimes_k \wedge^i F\otimes_k S(-i)\r L_{i-1}\r 0
\end{equation}
for $0\leq i\leq n$, one can easily check that $R^1\Ind_{B}^G(L_i)=0$ for $i\leq j$, and that $\Ind_B^G(L_i)=0$ for $i>j$. From \eqref{eq:sesK} with $i=j+1$ we obtain the equality $\Ind_B^G(L_j)=R^1\Ind_B^G(L_{j+1})$. 
Therefore \eqref{biresolution} induces the $G\times \GL(F)$-equivariant resolutions in the statement of the proposition.
\end{proof}

As an immediate consequence of Lemma \ref{lem:mapKj} we obtain the following corollary.
\begin{corollary}
\label{cor:kernel}
The $(G,S)$-module $K^u_j$ equals the kernel of the $j+1$-st non-zero morphism in~\eqref{biresolution0}. 
\end{corollary}

\section{Auxiliary properties of the standard modules}
\label{sec:aux}
In this section we revisit the $(G,S)$-modules $M_j$, $K_j$ introduced in \S\ref{calculations}, \S\ref{sec:maintechnical} (and related in \S\ref{sec:resol}), 
and we prove some useful properties about them which will be relevant below. 

\begin{remark}
Some of these properties can be proved differently using the 
connection with the Grassmannian which will be outlined in \S\ref{sec:sheaf}, combined with \eqref{eq:refl}.
\end{remark}
First we list some properties of $M_j$.
\begin{proposition}\label{prop:Mjproperties}
Let $1\leq j\leq n-2$. 
\begin{enumerate}
\item $M_j$ is Cohen-Macaulay.\label{Mjproperties1}
\item $M_j$ is indecomposable and  \label{Mjproperties2}
\begin{equation}
\End_{S}(M_j,M_j)_t=
\begin{cases}
k&\text{if $t=0$},\\
0&\text{if $t<0$}.
\end{cases}
\end{equation}
\item We have $M_j=\bigoplus_{t\ge 0} S^{j+t}V\otimes_k S^tF(-t)$ as graded $(G,\GL(F))$-represen\-tations. \label{Mjproperties3}
\item The dual of the resolution of $M_j$ given by \eqref{biresolution0} is the corresponding resolution of $M_{n-j-2}$ with the functor $(-\otimes_k \wedge^n F^\dur)(n)$ applied
to it. \label{Mjproperties4}
\end{enumerate}
\end{proposition}
\begin{proposition} \label{prop:Kjproperties}
Let $1\leq i\leq n-3$, $1\leq j\leq n-2$, $l\geq n-2$.
\begin{enumerate}
\item \label{lem:Kjdual}
As $(G,S)$-modules: 
\begin{align*}
K_j&\cong(\wedge^j K_1)^{\vee\vee}(-2j+2),\; K_{n-2}\cong \wedge^n F\otimes_k S, \\
K_i^\vee&\cong K_{n-i-2}(-2)\otimes_k \wedge^n F^\dur.
\end{align*}

\item\label{prop:noniso}
$\pdim K_j=n-j-2$. 

\item\label{lem:Kjgoodf}
$K_j$ has a good filtration.

\item\label{prop:kjindec}
$K_j$ is indecomposable and 
\begin{equation}
{}\End_S(K_j)_t=
\begin{cases}
k & \text{ if } t=0, \\
0 & \text{ if } t<0.
\end{cases}
\end{equation}

\item\label{lem:stablei+1} 
$
\tilde{\Omega}_s^{i+1}M_i=\Omega_s^{i+1}M_i \cong K_i(-i-2),
$ 
$\tilde{\Omega}_s^{l+1}M_l=0$.
\end{enumerate}
\end{proposition}

\begin{remark}
Note that \eqref{lem:stablei+1} is not trivial. Indeed the representations occurring in  \eqref{biresolution0}, \eqref{biresolution1} are typically not tilting
and so it is not clear why they would still yield the correct stable syzygies $\Omega_s^{j+1}M_j$ (see \S\ref{sec:stables}, \S\ref{sec:tiltfree}) and $\tilde{\Omega}_s^{j+1}M_j$ (see 
the statement of Theorem \ref{prop:main} for the definition). In particular this is false for lower syzygies.
\end{remark}
The following corollary is an immediate consequence of Proposition \ref{prop:Kjproperties}\eqref{prop:noniso}.
\begin{corollary}\label{cor:nonfreeKj}
All $K_j$ are non-isomorphic, and they are non-free for $1\le j\le n-3$.  
\end{corollary}

\begin{proposition} \label{prop:cminv}
The $K_j^G$ are indecomposable Cohen-Macaulay $R$-modules  for $1\le j\le n-3$. 
Moreover, $K_j^G\cong (\wedge^j K_1^G)^{\vee\vee}(-2j+2)$ and $K_1^G\cong\Omega^1M(V)(3)$.
\end{proposition}
The Cohen-Macaulayness assertion will ultimately also follow from the fact that the $K^G_j$ all occur as Frobenius summands (see Remark \ref{rem:onedegree}). However for the benefit of the reader we will 
give a direct proof below.

\subsection{Proof of Proposition \ref{prop:Mjproperties}}
\subsubsection{Proof of \ref{prop:Mjproperties}\eqref{Mjproperties1}}
It follows by \eqref{biresolution0} that $\operatorname{pdim} M_j=n-1$, thus the fact that $\operatorname{pdim} M_j+\dim M_j=\dim S$ implies \eqref{Mjproperties1}. 
\subsubsection{Proof of \ref{prop:Mjproperties}\eqref{Mjproperties2}}
By  \eqref{biresolution0}, $M_j$ is generated in degree zero. Hence it cannot have endomorphisms of negative degree.
To prove the second claim of \eqref{Mjproperties2} we use the geometric description of $M_j$. By \cite[\S 5, Theorem (5.1.2)(b)]{weyman2003cohomology} (or \cite[(11.2)]{vspenko2015non}), we have an isomorphism of $S$-modules 
\[M_j\cong\Gamma(X,q_*p^*\Lscr_{G/B}(j\omega)),\]
 where $X=\Spec S$,  $p:Z:=G\times^B \Spec(S_+)\to G/B$ is the projection, $q:Z\to X$ is the action morphism, and $\Lscr_{G/B}(\chi)$ is the line bundle
 on $G/B$ associated to a character~$\chi$. Let $q(Z)=X^u$ be the nullcone for the action of $G$ on $X$. As $q:Z\to X^u$ 
it follows that $M_j$ is torsion free of rank one on  $X^u$. In particular it is indecomposable. Furthermore it also follows that $\End_S(M_j)_0$ cannot contain nilpotent elements. 
Hence $\End_S(M_j)_0$, being finite dimensional, must be equal to~$k$.
\subsubsection{Proof of \ref{prop:Mjproperties}\eqref{Mjproperties3}}
This follows from the definition 
\eqref{eq:Mjdef}
together with the fact that $S_+=\bigoplus_t (t\omega)\otimes_j S^t F(-t)$ as graded $(B,\GL(F))$-modules.
\subsubsection{Proof of \ref{prop:Mjproperties}\eqref{Mjproperties4}}
Since the generators of the projectives in \eqref{biresolution0} live in a single degree these resolutions are unique up to unique isomorphism. 
In particular they are compatible with any group action. Hence  the dual of \eqref{biresolution0} (which is exact since $M_j$ is Cohen-Macaulay by \eqref{Mjproperties1} is the same as the corresponding resolution of~$M_{n-j-2}$ with the functor $(-\otimes_k \wedge^n F^\dur)(n)$ applied to it.

\subsection{Proof of Proposition \ref{prop:Kjproperties}}
\subsubsection{Proof of   \ref{prop:Kjproperties}\eqref{lem:Kjdual}}
Using the coresolution 
\begin{equation}
\label{eq:cores}
0\r K^u_1\r F\otimes_k S(-1)\r V\otimes_k S
\end{equation}
whose cokernel has rank 0 and $\dim$ $n+1$ we find that $K^u_1$ is a reflexive $S$-module of rank $n-2$ and 
\[
(\wedge^{n-2} K^u_1)^{\vee\vee}=\wedge^nF\otimes_k S(-n),
\]
and we obtain in particular the duality 
\begin{equation}
\label{eq:Kduality}
(\wedge^j K^u_1)^{\vee}\cong (\wedge^{n-j-2}K_1^u)^{\vee\vee}(n)\otimes_k \wedge^n F^\dur.
\end{equation}
Applying exterior powers to the coresolution \eqref{eq:cores} we obtain a complex
\begin{multline}
\label{newresolutionsub}
(\wedge^j K^u_1)^{\vee\vee} \r \wedge^j F\otimes_k S(-j)
\r
 V\otimes_k \wedge^{j-1} F\otimes_k S(-j+1).
\end{multline}
This complex is a left exact sequence since its  cohomology is supported in codimension $2n-(n+1)\allowbreak\geq 2$ and $(\wedge^j K_1^u)^{\vee\vee}$ is reflexive.\footnote{Recall that if $M$ is a reflexive module over a normal noetherian domain $S$ 
of dimension $n$ then it is torsion free and $\Ext^1_S(N,M)=0$ for $\dim N\le n-2$.} 

The last map in \eqref{newresolutionsub} and the corresponding map in
\eqref{biresolution0}
must agree up to a scalar by Lemma \ref{lem:mapKj}.
We thus obtain $(\wedge^jK_1^u)^{\vee\vee}\cong K_j^u$ by Corollary \ref{cor:kernel}. The duality now follows from \eqref{eq:Kduality} (or alternatively from Proposition \ref{prop:Mjproperties}\eqref{Mjproperties4}).
\subsubsection{Proof of  \ref{prop:Kjproperties}\eqref{prop:noniso}}
This result is clear from \eqref{biresolution0}.
\subsubsection{Proof of   \ref{prop:Kjproperties}\eqref{lem:Kjgoodf}} 
\begin{lemma}
  \label{divided}
Assume that $M\in \Fscr(\Delta)$ and $N \in \Fscr(\nabla)$. Then
$$
\Ext^i_G(M,D^jV \otimes_k N)=0
$$
for $i \geq \max(j,1)$.
\end{lemma}
\begin{proof}
We have $\Ext^i_G(M,D^jV\otimes_k N)=\Ext^i_G(M\otimes_k N^\vee
,D^jV)$ so we can assume $N=k$. Also $D^jV$ is tilting for $j=0$ or $1$, so in these cases there is nothing to prove. For $j>1$ we have an exact sequence
$$
0\r D^j V\r V\otimes_k D^{j-1}V\r  D^{j-2}V\r 0,
$$
such that the associated long exact sequence for $\Ext_G^\bullet(M,-)$ and induction yields
what we want. 
\end{proof}

\begin{proof}[Proof of \eqref{lem:Kjgoodf}]
We will show that for any $M\in \Fscr(\Delta)$ and $i\geq 1$ we have the vanishing $\Ext^i_G(M,K_j)=0$. To this end it suffices to combine \eqref{biresolution0} 
with Lemma~\ref{divided} and the fact that $S$ has a good filtration.
\end{proof}

\subsubsection{Proof of  \ref{prop:Kjproperties}\eqref{prop:kjindec}} 
\begin{proof}[Proof of  \eqref{prop:kjindec}]
Since $K_i$ is generated in degree zero it is clear that $\End_S(K_i)_t=0$ for $t<0$. As $K_i=\Omega^{i+1} M_i(i+2)$ and $\Ext^t_S(M_i,S)=0$
for $t=1,\ldots,2n-(n+1)-1=\allowbreak n-2$ (since $\dim M_i=n+1$) it follows from \cite{DIITWY} (second display in the proof of Theorem 1) that
$\End_S(K_i)_0=\underline{\End}_S(K_i)_0=\End_S(M_i)_0=k$ by Proposition \ref{prop:Mjproperties} (2). 
\end{proof}

\subsubsection{Proof of  \ref{prop:Kjproperties}\eqref{lem:stablei+1}} 
\label{resolutions}
To prove \eqref{lem:stablei+1} we must construct a tilt-free resolution (\S\ref{sec:tiltfree}) of the $M_j$.
Let us first recall how canonical tilting resolutions for the $S^iV$ can be constructed. Since $\Sym(V)=TV/(\wedge^2V)$ is a Koszul algebra, the vector spaces 
$$
V_\ell=V^{\otimes \ell-1} \otimes_k \wedge^2V \otimes_K V^{\otimes i-\ell-1}, 1\leq \ell\leq i-1,
$$ 
generate a distributive lattice in $V^{\otimes i}$, and give a standard finite complex (consisting of tilting modules) resolving $S^iV$:
\begin{equation}
\label{tiltres}
\cdots \r \bigoplus_{\ell_1<\ell_2<\ell_3} V_{\ell_1}\cap V_{\ell_2}\cap V_{\ell_3}\r \bigoplus_{\ell_1<\ell_2} V_{\ell_1}\cap V_{\ell_2}\r\bigoplus_{\ell_1} V_{\ell_1}\r V^{\otimes i} \r S^i V\r 0.
\end{equation}
The length of \eqref{tiltres} is $\lfloor i/2 \rfloor$. 
\begin{proof}[Proof of \eqref{lem:stablei+1}]
Let us take a resolution~\eqref{biresolution0} of $M_j$. 
We have 
\begin{multline}\label{eq:reswithK}
0\r K_j(-j-2)\r \wedge^j F\otimes_k S(-j)\r V\otimes_k \wedge^{j-1}F \otimes_k S(-j+1) \r\cdots\\\cdots\r S^jV \otimes_k S\r M_j\r 0.
\end{multline}
Note that the first part of the resolution
$$
\xymatrix{
0\ar[r]& K_j(-j-2)\ar[r]& \wedge^j F \otimes_k S(-j)\ar[r]& V\otimes_k \wedge^{j-1}F \otimes_k S(-j+1)\ar[d]\\
&&&S^2V\otimes_k \wedge^{j-2}F \otimes_k S(-j+2)
}
$$
lifts to a diagram
\[
\xymatrix@=0.48em{
0\ar[r]& K_j(-j-2)\ar[r]\ar[d]& \wedge^j F \otimes_k S(-j)\ar[r]\ar[d]& V\otimes_k \wedge^{j-1}F\otimes_k S(-j+1)\ar[d]\\
&0\ar[r]& \wedge^2V \otimes_k  \wedge^{j-2}F\otimes_k S(-j+2) \ar[r]&V\otimes_k V\otimes_k \wedge^{j-2}F\otimes_k S(-j+2)
}
\]
since 
\begin{equation}
\Ext^1_{G,S}(V\otimes_k \wedge^{j-1}F \otimes_k S(-j+1),\wedge^2V \otimes_k  \wedge^{j-2}F\otimes_k S(-j+2))=0,
\end{equation}
because these modules are tilt-free. 
Continuing like this we replace one at a time $S^\ell V \otimes_k \wedge^{j-\ell}V \otimes_k S(-j+l)$ by the tilting resolution \eqref{tiltres}.
The total complex of the resulting double complex gives a tilt-free resolution $T^\bullet$ of $M_j$ up to the $(j+1)$-st syzygy $K_j(-j-2)$ (here we use the comment
on the length of \eqref{tiltres}). Indeed, since $K_j$ has a good filtration, the same holds by induction for all other syzygies in $T^\bullet\r M_j\to 0$. 
The same reasoning works for $j > n-2$ using the resolutions \eqref{biresolution1}, implying $\tilde{\Omega}_s^{j+1}M_j=\Omega_s^{j+1}M_j=0$ in this case.

Since $K_i$ is indecomposable by \eqref{prop:kjindec} and non-free by Corollary \ref{cor:nonfreeKj} for $1\leq i\leq n-3$ we have $\tilde{\Omega}_s^{i+1}M_i=\Omega_s^{i+1}M_i=K_i$.  
Finally, we obtain $\tilde{\Omega}_s^{n-1}M_{n-2}=\Omega_s^{n-1}M_{n-2}=0$ as $K_{n-2}$ is tilt-free by \eqref{lem:Kjdual}.
\end{proof}

\begin{remark}
\label{length}
Depending on $p$ there will often be (much) shorter tilting resolutions of $S^iV$. In general we have a surjective map $T(i)\to S^iV$, and its kernel has a good filtration, consisting of $S^\ell V$ with $\ell<i$, so we can proceed by induction.
For the tilting resolution of $S^iV$ constructed in this way only the $T(j)$ with $j \uparrow i$  appear by the linkage principle \cite[Proposition II.6.20]{jantzen2007representations}.
\end{remark}

\subsection{Proof of Proposition \ref{prop:cminv}}
The tilt-free resolution of $M_j$ constructed in the proof of
\eqref{lem:stablei+1} which yields $K_j^u$ as stable
syzygies involves only modules of the form $V^{\otimes i}\otimes_k S$,
$i\le j\le n-3$.  Since $K_j$ has a good filtration by 
\eqref{lem:Kjgoodf} this tilt-free resolution remains exact after
applying $(-)^G$. Since $M_j^G=0$ (this follows immediately from the definition \eqref{eq:Mjdef} and adjointness) it becomes a coresolution of $K^G_j$
by modules of covariants of the form $M(V^{\otimes i})$ for
$i\le n-3$. It now suffices to apply Proposition \ref{prop:cm}.  The
fact that $K_j^G$ is indecomposable follows from \eqref{prop:kjindec} and \eqref{eq:refl}.
Applying the coresolution for $j=1$ we get in particular $(K_1(-3))^G=\Omega^1M(V)$. 

By \eqref{lem:Kjdual} we have
$K_j\cong (\wedge^j(K_1(-3)))^{\vee\vee}(j+2)=(\wedge^jK_1)^{\vee\vee}(-2j+2)$. This identity remains valid
after taking invariants thanks to \cite[\S4]{vspenko2015non} (or \cite{Brion}).

\section{Proof of the main decomposition result}
\label{sec:mainproof}
Let us recall that we have to prove that there is an isomorphism  of  graded $(G^{(1)},S^p)$-modules
\begin{equation}
\label{finaldecompsub}
S^{G_1} \cong  \bigoplus_{j=1}^{n-3} \bigoplus_{\begin{smallmatrix}t\in\Nscr,q_t=0,j\equiv 0\,(2)\text{ or }\\ q_t=p-2,j\equiv 1\,(2)\end{smallmatrix}} 
K^{\Fr}_j(-p(j+2)-d_t)^{\oplus n_t}\oplus T^{\overline{\Fr}} \otimes_k S^p,  
\end{equation}
where $T$ is a graded tilting $G$-representation which has been explicitly described in~\S\ref{sec:maintechnical}. We will first prove \eqref{finaldecompsub} without considering $T$ and then 
afterwards we will use \eqref{finaldecompsub} to compute $T$.
\subsection{Proof of \eqref{finaldecompsub}}
Let $C(k)$ be a left projective $G_1$-resolution of $k$  consisting of $G$-tilting modules (see Lemma \ref{lem:quotient_tilting}). 
Then $M^\bullet:=\Hom_{G_1}(C(k)_{\ge 0},S)$ computes $\RHom_{G_1}(k,S)$.
We verify the hypotheses of Theorem \ref{prop:main}:
\begin{enumerate}
\item follows from Proposition~\ref{prop:stable_syzygy};
\item clear;
\item follows from Lemma \ref{free} below by the choice of the $C(k)^i$;
\item follows by Theorem \ref{thm:vandenkallen} 
 since $S$ has a good filtration;
\item follows also from Theorem \ref{thm:vandenkallen} by combining it with Proposition~\ref{canonical};
\item follows from Theorem~\ref{formalityy}.
\end{enumerate}

We can apply Theorem~\ref{prop:main} to find the following isomorphism.
\begin{equation}\label{SG1decom}
S^{G_1}\cong T^{\overline{\Fr}}\otimes_k S^p\oplus \bigoplus_{j\ge 1} \tilde{\Omega}^{j+1}_s \HHH^j(G_1,S)^{\Fr},
\end{equation}
with $T$ a graded tilting $G$-module. Now from Proposition~\ref{mainprop} we compute
\begin{equation}
\label{sum}
  \bigoplus_{j\ge 1} \tilde{\Omega}^{j+1}_s \HHH^j(G_1,S)^{\Fr}
=\bigoplus_{j\ge 1} \bigoplus_{\begin{smallmatrix}t\in \Nscr,q_t=0,j\equiv\, 0(2)\text{ or }\\ q_t=p-2,j\equiv 1\,(2)\end{smallmatrix}} 
( \tilde{\Omega}_s^{j+1} M_j)^{\Fr}(-d_t)^{\oplus n_t}.
\end{equation}

By Proposition \ref{prop:Kjproperties}\eqref{lem:stablei+1}, $\tilde{\Omega}^{j+1}_{s} M_j\cong K_j(-j-2)$ for $j< n-2$. 
 Since $\tilde{\Omega}^{j+1}_{s} M_j=0$ for $j\geq n-2$, the decomposition \eqref{finaldecompsub} follows.

We have used the following lemma.
\begin{lemma} 
\label{free}
Let $P$ be a finite dimensional projective $G_1$-representation. Then $(P\otimes_k S)^{G_1}$
is graded free as $S^p$-module. Moreover if $P$ has a good filtration then $(P\otimes_k S)^{G_1}$
is $\nabla$-free (\S\ref{sec:projectives}) as $(G^{(1)},S^p)$-module.
\end{lemma}

\begin{proof} 
The module $P\otimes_k S$ is free as $S^p$-module and moreover the map
$$
P\otimes_k S\r (P\otimes_k S)\otimes_{S^p} k=P\otimes_k (S\otimes_{S^p} k):=M
$$ 
splits as graded $G_1$-modules (as $M$ is also $G_1$-projective). We may use this splitting to construct a map of graded $(G_1,S^p)$-modules
\[
M\otimes_k S^p\r P\otimes_k S
\]
which must be an isomorphism by Nakayama's lemma. Hence $(P\otimes_k S)^{G_1}=\allowbreak M^{G_1}\otimes_k S^p$ which yields the freeness claim. If $P$ has a good filtration
then the same it true for $(P\otimes_k S)^{G_1}$ by Theorem \ref{thm:vandenkallen}. It now suffices to invoke Proposition \ref{prop:nablafree}.
\end{proof}

\subsection{Computing {\boldmath $T$}}\label{subsec:compT}
We will prove the following result. 

\begin{proposition}\label{lem:split}
As graded $G^{(1)}$-modules
\begin{equation}
\label{eq:subtraction}
 (S/S^p_{>0}S)^{G_1} \cong   T^{\overline{\Fr}}\oplus \bigoplus_{j=1}^{n-1}\bigoplus_{\begin{smallmatrix}t\in\Nscr,q_t=0,j\equiv 0\,(2)\text{ or }\\ q_t=p-2,j\equiv 1\,(2)\end{smallmatrix}}\wedge^j F(-jp-d_t)^{\oplus n_t}.
\end{equation}
\end{proposition}
The proof of Proposition \ref{lem:split} will consist of a number of steps.
Tensoring \eqref{finaldecompsub}
with $S^p/S^p_{>0}$ and using the resolution \eqref{biresolution0} defining $K^u_j$  in Section \ref{sec:resol}, we find as $G^{(1)}$-modules:
\begin{equation}
\label{finaldecomp2}
S^{G_1}/S_{>0}^pS^{G_1} \cong  \bigoplus_{j=1}^{n-3} \bigoplus_{\begin{smallmatrix}t\in \Nscr,q_t=0,j\equiv 0\,(2)\text{ or }\\ q_t=p-2,j\equiv 1\,(2)\end{smallmatrix}} 
\wedge^{j+2}F(-(j+2)p-d_t)^{\oplus n_t}\oplus T^{\overline{\Fr}}.  
\end{equation}
We obtain in particular

\begin{corollary}
\label{cor:implicit}
  $T^{\overline{\Fr}}$ is obtained from $S^{G_1}/S^p_{>0}S^{G_1}$ by removing the (explicitly known) first summand in \eqref{finaldecomp2}.
\end{corollary}

The next step is to compare the unknown $S^{G_1}/S^p_{>0}S^{G_1}$  with the more tractable object $(S/S^p_{>0}S)^{G_1}$. 

\begin{remark}
In the analogous situation
in \cite[3.2.3]{MR1444312} $G$ was a torus and then $G_1$  {is} linearly reductive so that we 
actually have $S^{G_1}/S^p_{>0}S^{G_1}=
(S/S^p_{>0}S)^{G_1}$. This will not be the case here.
\end{remark}

\begin{lemma}\label{lem:4seq}{}
We have an exact sequence of graded $G^{(1)}$-modules:
\begin{multline}\label{eq:4seq}
0\r \oplus_{i=1}^\infty \Tor^{S^p}_{i+1}(H^i(G,S),k)\r 
S^{G_1}/S^{G_1}S^p_{>0}\r (S/SS^p_{>0})^{G_1}\r \\\oplus_{i=1}^\infty \Tor^{S^p}_i(H^i(G,S),k)\r 0.
\end{multline}
\end{lemma}

\begin{proof}
Using the fact that $S^p/S_{>0}^p$ is perfect as $S^p$-module, we see that  in ${D}(G^{(1)})$ there is an isomorphism
\begin{equation}\label{eq:iso1}
\RHom_{G_1}(k,S)\Lotimes_{S^p} S^p/S^p_{>0}\cong \RHom_{G_1}(k,S\otimes_{S^p} S^p/S^p_{>0}).
\end{equation}
Note that on the RHS we can indeed write $\otimes$ instead of $\Lotimes$ since $S$ is a free $S^p$-module. Below we write
$k=S^p/S^p_{>0}$ for brevity.

We have a distinguished triangle in ${D}(G^{(1)},\Gr(S^p))$ 
\[
 S^{G_1}\r\RHom_{G_1}(k,S)\r \tau_{\geq 1}\RHom_{G_1}(k,S)\r .
\]
Since $\tau_{\geq 1}\RHom_{G_1}(k,S)$ is formal by Theorem \ref{formalityy}, this distinguished triangle becomes
\begin{equation}\label{eq:seq2for}
 S^{G_1}\r\RHom_{G_1}(k,S)\r \oplus_{i=1}^\infty H^i(G_1,S)[-i]\r .
\end{equation}
Derived tensoring \eqref{eq:seq2for} by $k$ and combining it with \eqref{eq:iso1} 
we thus obtain a distinguished triangle in $D(G^{(1)})$
\begin{equation}\label{eq:precoh}
S^{G_1}\Lotimes_{S^p} k\r\RHom_{G_1}(k,S\otimes_{S^p} k)\r \oplus_{i=1}^\infty H^i(G_1,S)[-i]\Lotimes_{S^p} k\r .
\end{equation}

We obtain \eqref{eq:4seq} from the cohomology long exact sequence
associated to \eqref{eq:precoh}.  In degree $-1$ the middle term has
no cohomology, and the cohomology of the last term equals
$\oplus_{i=1}^\infty \Tor^{S^p}_{i+1}(H^i(G_1,S),k)$. In degree zero
we obtain precisely the last three terms in \eqref{eq:4seq}. It
remains to observe that the first term lives in $\leq 0$ degrees, and
therefore it has no cohomology in degree $1$.
\end{proof}

\begin{lemma}\label{lem:square}
As graded $G$-representations 
\begin{equation}\label{eq:square}
S/S^p_{>0}S\cong \left(\bigoplus_{j=0}^{p-1} L(j)(-j)\oplus \bigoplus_{j=0}^{p-2}L(j)(-(2p-2-j))\right)^{\otimes n}. 
\end{equation}
In particular, $(S/S^p_{>0}S)^{G_1}$ is a tilting $G^{(1)}$-module.
\end{lemma}

\begin{proof}
Note that $S/S^p_{>0}S=(P/P^p_{>0})^{\otimes n}$ (see \S\ref{sec:computingRHom}). 
It is clear that the vector space of polynomials of degree $j$ for $j\leq p-1$ in $P/P_{>0}^p$ is as a $G$-representation isomorphic to $L(j)$. Note that the same holds for the vector space of polynomials of degree $2p-2-j$ in $P/P^p_{>0}$. This gives the isomorphism \eqref{eq:square}.

For the last statement we note that since the tensor product of tilting modules is tilting, the RHS of \eqref{eq:square} is a tilting module. Now we invoke Proposition \ref{prop:G1inv}. 
 \end{proof}

\begin{lemma}\label{lem:tor}
\begin{align}
\bigoplus_{i=1}^\infty \Tor^{S^p}_i(H^i(G,S),k)&\cong \bigoplus_{j=1}^{n}\bigoplus_{\begin{smallmatrix}t\in\Nscr, q_t=0,j\equiv 0\,(2)\text{ or }\\ q_t=p-2,j\equiv 1\,(2)\end{smallmatrix}}\wedge^j F(-jp-d_t)^{\oplus n_t}\\\label{eq:tor2}
\bigoplus_{i=1}^\infty \Tor^{S^p}_{i+1}(H^i(G,S),k)&\cong \bigoplus_{j=1}^{n-2}\bigoplus_{\begin{smallmatrix}t\in\Nscr, q_t=0,j\equiv 0\,(2)\text{ or }\\ q_t=p-2,j\equiv 1\,(2)\end{smallmatrix}}\wedge^{j+2} F(-(j+2)p-d_t)^{\oplus n_t}
\end{align}
\end{lemma}

\begin{proof}
By Proposition \ref{mainprop}
\begin{equation*}
H^j(G_1,S)=\widehat{H}^j(G_1,S)\cong \bigoplus_{\substack{t\in\Nscr, q_t=0,j\equiv 0\,(2)\text{ or }\\ q_t=p-2,j\equiv 1\,(2)}} M_{j}^{\Fr}(-d_t)^{\oplus n_t}
\end{equation*} 
for $j>0$. 
Let $1\leq i<\infty$ and let $l\in \{i,i+1\}$. 
Using the resolutions \eqref{biresolution0}, \eqref{biresolution1} we find
\[\Tor^{S^p}_l(M_i^{\Fr},k)=
\begin{cases}
\wedge^iF(-ip) &\text{if $l=i$,}\\
\wedge^{i+2}F(-(i+2)p) &\text{if $l=i+1$.}\\
\end{cases}
\]
Thus{}, the lemma follows.
\end{proof}

\begin{proof}[Proof of Proposition \ref{lem:split}]
Since all terms in \eqref{eq:4seq} are tilting modules for $G^{(1)}$ by \eqref{finaldecomp2} and Lemmas \ref{lem:square}, \ref{lem:tor}, \eqref{eq:4seq} splits and becomes the desired isomorphism (after canceling  common summands on both sides).
\end{proof}
\subsection{Combinatorial description of {\boldmath $T$}}\label{sec:combdesc}
The sign sequences introduced in \S7 are designed to keep track of  individual $L(i)$ in the expansion of \eqref{eq:square}.
Indeed in  every factor of \eqref{eq:square} the representation $L(i)$ for $0\le i\le p-2$ appears twice, namely in degrees $i$ (sign=$-$) and $2p-2-i$ (sign=$+$). 
On the other hand $L(p-1)$ occurs only once in degree $p-1$ (sign=$0$).
With these conventions \eqref{eq:square} may be rewritten as
\begin{equation}
\label{eq:roleofsigns}
S/S^p_{>0}S=\bigoplus_{t\in\Mscr,\sscr\in \mathfrak{S}_t} T(q_t)(-d^{\sscr}_t)^{\oplus n_t}.
\end{equation}

\begin{remark}
\label{rem:sign}
Put $\mathfrak{S}=\{-,+\}^n$. When $t\in \Nscr\subset \Mscr$ then $\mathfrak{S}_t=\mathfrak{S}$.
\end{remark}

The following proposition completes the proof of Theorem \ref{mainth}.

\begin{proposition}\label{prop:thm} Let $T$ be as in \eqref{eq:subtraction}. Then
we have 
\begin{multline}\label{eq:tiltexp}
T\cong \bigoplus_
{\begin{smallmatrix}t\in \Mscr,q_t=0,\\ \sscr=(-)_{\ell=1}^n\, \text{or}\,\sscr=(+)_{\ell=1}^n\end{smallmatrix}}
k(-d_t^{\sscr})^{\oplus n_t}\oplus \\\bigoplus_{\begin{smallmatrix}t\in \Mscr,q_t-2p+2\equiv 0\,(p),\,q_t\neq p-2,\\\sscr\in \mathfrak{S}_t\end{smallmatrix}} T((q_t-2p+2)/p)(-d_t^{\sscr})^{\oplus n_t}\,.
\end{multline}
\end{proposition}

\begin{proof}
To prove this proposition  we have to subtract certain trivial representations from the $G_1$-invariants of the right-hand side  of \eqref{eq:roleofsigns},
according to \eqref{eq:subtraction}. 

Using Proposition \ref{prop:G1inv} and Remark \ref{rem:sign} the $G_1$-invariants 
of the right-hand side of \eqref{eq:roleofsigns} (separating out the trivial representations)
are given by
\begin{multline}
\label{eq:trivialpart}
\hspace*{-0.5cm} (S/S^p_{>0}S)^{G_1}=\bigoplus_{\begin{smallmatrix} t\in\Mscr,q_t=0\\\sscr\in \mathfrak{S}_t\end{smallmatrix}}
k(-d_t^{\sscr})^{\oplus n_t}\oplus \bigoplus_{\begin{smallmatrix}t\in\Mscr,q_t-2p+2\equiv 0\,(p)\\q_t\neq p-2,\sscr\in \mathfrak{S}_t\end{smallmatrix}} T((q_t-2p+2)/p)^{\Fr}(-d_t^{\sscr})^{\oplus n_t}
\\
=\bigoplus_{\begin{smallmatrix} t\in\Nscr,q_t=0\\\sscr\in \mathfrak{S}\end{smallmatrix}}
k(-d_t^{\sscr})^{\oplus n_t}\oplus \bigoplus_{\begin{smallmatrix}t\in\Mscr,q_t-2p+2\equiv 0\,(p)\\q_t\neq p-2,\sscr\in \mathfrak{S}_t\end{smallmatrix}} T((q_t-2p+2)/p)^{\Fr}(-d_t^{\sscr})^{\oplus n_t}
\end{multline}
and the representation we have to subtract is
\begin{equation}
\label{eq:tosubtract}
\bigoplus_{j=1}^{n-1}\bigoplus_{\begin{smallmatrix}t\in\Nscr,q_t=0,j\equiv 0\,(2)\text{ or }\\ q_t=p-2,j\equiv 1\,(2)\end{smallmatrix}}\wedge^j F(-jp-d_t)^{\oplus n_t}.
\end{equation}
We will in fact subtract \eqref{eq:tosubtract} from the first part of the right-hand side of \eqref{eq:trivialpart}. This amounts to a combinatorial problem which we discuss next.
For $q\in\{0,1,\allowbreak\ldots,p-2\}$ put
\[
\epsilon_j(q)=
\begin{cases}
q&\text{if $j$ is even,}\\
p-2-q&\text{if $j$ is odd.}
\end{cases}
\]
Let $\mathscr{P}$ be the power set of $\{1,\ldots,n\}$.
For $J\in \mathscr{P}$ we define $\phi_J:\ZZ^{n}\r \ZZ^{n}$ via
\[
\phi_J(i_1,\ldots,i_n)_\ell=
\begin{cases}
p-2-i_\ell&\text{if $\ell\in J$,}\\
i_\ell&\text{otherwise,}
\end{cases}
\]
and $\Phi_J:\ZZ^{n+1}\r \ZZ^{n+1}$ via
\[
\Phi_J(q,i_1,\ldots,i_n)=(\epsilon_{|J|}(q),\phi_J(i_1,\ldots,i_n))
\]
According to Lemma \ref{lem:phijdef} below $\Phi_J$ restricts to a multiplicity preserving bijection
$\Nscr\r\Nscr$.

We define the sign sequence $\sscr^J\in \mathfrak{S}$ via
\[
\sscr^J_\ell=
\begin{cases}
+&\text{if $\ell\in J$,}\\ 
-&\text{otherwise.}
\end{cases}
\]
This yields a bijection
\begin{equation}
\label{eq:bij}
\mathscr{P}\r \mathfrak{S}:J\mapsto \sscr^{J}\,.
\end{equation}
Let $\mathscr{P}^{(j)}$, $\mathfrak{S}^{(j)}$ be the subsets of $\mathscr{P}$, $\mathfrak{S}$ which have respectively
$j$ elements and $j$ $+$ signs. For $a\in \{0,p-2\}$ put
\[
\Sscr_a=\{t\in \Nscr\mid q_t=a\}\,.
\]
\begin{lemma}\label{lem:dualitiesdt}
  Let $0\leq j\leq n$,  $a\in \{0,p-2\}$, $a\equiv j\,(2)$. There is a bijection
\begin{gather*}
\Psi_j:\Sscr_a\times \mathscr{P}^{(j)}\to \Sscr_0\times \mathfrak{S}^{(j)}:\;
(t,J)\mapsto (\Phi_J(t),\sscr^J).
\end{gather*}
Moreover, if we put $\Psi_j(t,J)=(\bar{t},\sscr)$ then  
$d_{\bar{t}}^{\sscr}=jp+d_t$, 
$n_{\bar{t}}=n_t$. 
\end{lemma}

\begin{proof} 
The bijection follows from \eqref{eq:bij} and Lemma \ref{lem:phijdef} below.
The last claim follows from the definitions and Lemma \ref{lem:phijdef} again.
\end{proof}
Since $\mathscr{P}^{(j)}$ can be identified with a basis of $\wedge^jF$ 
 we can use the isomorphism $\Psi_j$ from Lemma \ref{lem:dualitiesdt} for $1\leq j\leq n-1$ to obtain 
\begin{equation}
\label{eq:rewrite}
\eqref{eq:tosubtract}
\cong
\bigoplus_
{\begin{smallmatrix} t\in\Nscr,q_t=0,\sscr\in \mathfrak{S}\\\sscr\neq (-)_{\ell=1}^n,\sscr\neq (+)_{\ell=1}^n\end{smallmatrix}}
k(-d_t^{\sscr})^{\oplus n_t}.
\end{equation}
Note that $(-)_{\ell=1}^n$, $(+)_{\ell=1}^n$ equal $\sscr^{\{\}}$, $\sscr^{\{1,\dots,n\}}$ respectively, 
and thus both need to be excluded on the RHS.

Subtracting the right-hand side of \eqref{eq:rewrite} from \eqref{eq:trivialpart} finishes the proof of Proposition \ref{prop:thm}.
\end{proof}
We have used the following lemma which shows that the pair $(\Nscr,n)$ has a lot of internal symmetry.
\begin{lemma} \label{lem:phijdef} The $\Phi_J$ restrict to multiplicity preserving bijections
$
\Phi_J:\Nscr\to \Nscr
$.
\end{lemma}
\begin{proof} If $J_1\cap J_2=\emptyset$ then $\Phi_{J_1\cup J_2}=\Phi_{J_1}\Phi_{J_2}$.
Thus it suffices to consider $J=\{\ell\}$. Put $(i'_1,\ldots,i'_n)=\phi_J(i_1,\ldots,i_n)$. 
Then by  Corollary \ref{cor:invertible} we find
\[
L(i'_{1})\underline{\otimes} \cdots \underline{\otimes} L(i'_{n})=L(i_{1})\underline{\otimes} \cdots \underline{\otimes} L(i_{n})\underline{\otimes} L(p-2).
\]
It follows that a summand $L(q)$ of  $L(i_{1})\underline{\otimes} \cdots \underline{\otimes} L(i_{n})$ 
corresponds to a summand $L(q)\underline{\otimes} L(p-2)=L(p-2-q)=L(\epsilon_1(q))$ (again using Corollary \ref{cor:invertible}) of the same multiplicity of $L(i'_{1})\underline{\otimes} \cdots \underline{\otimes} L(i'_{n})$,
finishing the proof.
\end{proof}

\section{Degrees and types of the indecomposable summands}\label{sec:freesummands}
\subsection{Indecomposable summands of {\boldmath $T$}}\label{subsec:T}
From \eqref{eq:Tdef} we can read off the highest weights $j$ and the degrees $d$ of the indecomposable summands $T(j)(-d)$ in $T$.

\begin{lemma} \label{rem:q_tvse_t}
For all $t\in \Mscr$, $\sscr\in \mathfrak{S}_t$ we have:
\begin{enumerate}
\item  $q_t$, $d_t$, and $d_t^{\sscr}$ all have the same parity,
\item
$
q_t\leq d_t^{\sscr}\leq 2(p-1)n-q_t
$.
\end{enumerate}
\end{lemma}

\begin{proof}
(1) is easy. For (2) we note that
by looking a the highest weight we have $q_t\leq \sum_j i_{tj}\leq \sum_j i^{\sscr}_{tj}$. Moreover the definition of $i^{\sscr}_t$ 
yields $\sum_j i^{\sscr}_{tj}+\sum_j i_{tj} \leq 2(p-1)n$
 and therefore  also
$\sum_j i^{\sscr}_{tj}\leq 2(p-1)n-\sum_j i_{tj}\leq 2(p-1)n-q_t$. This yields the desired result.
\end{proof}

First we state easy corollaries of \eqref{eq:Tdef} combined with Lemma \ref{rem:q_tvse_t} which give the list of possible highest weights and degrees.

\begin{corollary}\label{cor:n-3}
The tilting module $T$ in \eqref{finaldecomp} has weights in $[0,n-3]$.
\end{corollary}

\begin{proof}
By \eqref{eq:Tdef} we can assume that $q_t=2p-2+jp$ for a weight $j>0$ of $T$.    
By Lemma \ref{rem:q_tvse_t}(2) we have $2q_t\leq 2n(p-1)$. 
Thus, $jp\leq (n-2)(p-1)$ and so $j\leq n-3$. 
\end{proof}

\begin{remark}
Since $R$ is a Cohen-Macaulay $R^p$-module, $T\{l\}^{\Fr}$ is also Cohen-Macaulay if $T(l)$ (forgetting the grading) is a direct summand of $T$. Thus, Corollary \ref{cor:n-3} also follows from Proposition \ref{prop:cm}.
\end{remark}

\begin{corollary}\label{lem:n-3}
Let $j\geq 1$. 
If $T(j)(-d)$ is an indecomposable summand of $T$, then $2p-2+jp\leq d\leq (2n-j-2)p-2n+2$.
\end{corollary}

\begin{proof}
It follows from \eqref{eq:Tdef} that 
$j=(q_t-2p+2)/p$, $d=d_t^{\sscr}$ for some $t\in \Mscr$, $\sscr\in \mathfrak{S}_t$.  
Thus, $q_t=2p-2+jp$. By Lemma \ref{rem:q_tvse_t}(2) we then have the inequalities
$2p-2+jp\leq d\leq 2(p-1)n-2p+2-jp$.
\end{proof}
We now show the converse of Corollaries \ref{cor:n-3}, \ref{lem:n-3}.
The following lemma which may be of independent interest will be used afterward to certify the existence
of certain elements $(q,(p-1)^l,b,0^{n-1-l})\in \Mscr$. 
\begin{lemma}\label{lem:tiltindegrees}
Assume $l\ge 1$ and $0\leq b\le p-1$ and put $d=l(p-1)+b$. Then the tilting module 
$L(p-1)^{\otimes l}\otimes_k L(b)$ contains the tilting
modules $T(q)$ as direct summand for $q\equiv d\,(2)$ in the range $[p-1,d]$.
\end{lemma}

\begin{proof}
We first assume that $p=2$. 
In this case we should prove that $T(q)$ is a direct summand of $L(1)^{\otimes d}$. Looking at the highest weights we see that $T(q)$ is a direct summand of $L(1)^{\otimes q}$. It is thus sufficient to show
that $L(1)^{\otimes q}$ is a direct summand of $L(1)^{\otimes d}$ and using the parity restriction on $d,q$
we see that it is sufficient to consider the case
 $d=q+2$. As $q\ge 1$ this case follows from the fact that
$L(1)$ is a direct summand of $L(1)^{\otimes 3}$.

Assume now that $p>2$. Note that in this case $q\equiv d\equiv b\,(2)$. 
By Lemma \ref{rem:tnz}, $L(p-1)\otimes_k L(p-1)$ is the  direct sum of $T(i)$ for even $i\in [p-1,2p-2]$. 
Looking at the highest weights  we thus see that in  $L(p-1)^{\otimes 3}$ occur as summands $T(i)$ (coming from  $T(i-p+1)\otimes_k T(p-1)$) for even $i\in [2p,3p-3]$, and $T(i)$ (coming from  $L(p-1)\otimes_k L(p-1)$) for even $i\in [p-1,2p-2]$. Thus,  $L(p-1)^{\otimes 3}$ is (up to multiplicity) the direct sum of $T(i)$ for even $i\in [p-1,3p-3]$. 
Continuing in the same way we find that $L(p-1)^{\otimes l}$ is (up to multiplicity) the direct sum of $T(i)$ for even $i\in [p-1,l(p-1)]$. Since  $L(b)\otimes_k L(p-1)$ is the direct sum of $T(i)$ for $i\in [p-1,p-1+b]$ with $i\equiv b\,(2)$ by Lemma \ref{rem:tnz}, we obtain that $L(b)\otimes_k L(p-1)^{\otimes l}$ decomposes (up to multiplicity) into the direct sum of $T(q)$ for $q\in [p-1,l(p-1)+b]$ satisfying $q\equiv b\,(2)$ (again also by highest weights comparison).
\end{proof}

For $t\in \Mscr$ define $\sscr_t^{+}$, $\sscr_t^-\in \mathfrak{S}_t$ by 
\[
\sscr^{\pm}_{tj}=
\begin{cases}
\pm&\text{if $i_{tj}\le p-2$,}\\
0&\text{if $i_{tj}= p-1$.}
\end{cases}
\]
We have
\begin{equation}
\label{rem:dualsigns} 
d_t^{\sscr^-_t}+d_t^{\sscr^+_t}=2n(p-1).
\end{equation}

\begin{corollary}\label{cor:freedeg}
Let $0\leq j\leq n-3$, and let $(j+2)p-2\leq d\leq (2n-j-2)p-2n+2$, $d\equiv jp\,(2)$.
Then $T(j)(-d)$ is a direct summand of $T$. 
Moreover, $k(-d)$ is a direct summand of $T$ for even $0\leq d\leq 2n(p-1)$. 
\end{corollary}

\begin{proof}
It will be convenient to write $T=T_1\oplus T_2$ where $T_1$, $T_2$ refer to the two parts in the expression on the right-hand side of \eqref{eq:Tdef}.
Note that $T_1$ consists only of trivial representations but $T_2$ may contain both trivial and non-trivial representations.\footnote{In fact
setting $j=0$ it follows from the first part of the proof of this corollary that $T_2$ \emph{always} contains trivial representations.}

We now use the same notations as in Lemma \ref{lem:tiltindegrees}.
Put $q=2p-2+jp$ and let $2p-2+jp\leq d\leq n(p-1)$, $d\equiv jp \,(2)$.  Fix $0\leq b\leq n-1$, $1\leq l\leq p-1$ such that $d=l(p-1)+b$. Then by Lemma \ref{lem:tiltindegrees}
 $L(p-1)^{\otimes l}\otimes_k L(b)$ contains $T(q)$ as a summand. In other words $t:=(2p-2+jp,(p-1)^{l},b,0^{n-l-1})\in \Mscr$. It thus follows from  \eqref{eq:Tdef}
(applied with $\sscr=\sscr^-_t$) that $T(j)(-d)$ is a summand of $T_2$.

Using \eqref{rem:dualsigns} we see that the $d\geq n(p-1)$ can be obtained by setting $\sscr=\sscr^+_t$ in \eqref{eq:Tdef}.\footnote{One may also use the duality statement in Lemma \ref{lem:duality}.} 
Indeed,  any $n(p-1)\leq d\leq (2n-j-2)p-2n+2$ with $d\equiv jp\,(2)$ is of the form $2n(p-1)-d'$ for $(j+2)p-2\leq d'\leq n(p-1)$, $d'\equiv jp\,(2)$.

Put $j=0$. We see  that the last statement only yields new information in case $d<2p-2$ and $d>2n(p-1)-2p+2$. These two cases are connected by a duality $d\leftrightarrow 2n(p-1)-d$. 
Assume  $d<2p-2$. Using the fact that $L(0)$ is a direct summand of $L(a)\otimes_k L(a)$, $0\leq a\leq p-2$ (see Lemma \ref{rem:tnz}) and the hypothesis that~$d$ is even
we find that $t:=(0,(d/2)^2,0^{n-2})\in \Mscr$, $\sscr^-_t=(-)_{\ell=1}^n$, and $d_t^{\sscr^-_t}=d$. Since we also have $\sscr_t^{+}=(+)_{\ell=1}^n$ and
$d_t^{\sscr^+_t}=2(p-1)n-d_t^{\sscr^-_t}=2(p-1)n-d$, we find that in all cases $L(0)(-d)=k(-d)$ is a summand of $T_1$, finishing the proof.
\end{proof}

\subsection{The degrees of {\boldmath $K_j^{\Fr}$}}\label{subsec:degKj}
We give a more concrete description of the degrees in which  $K_j$  appear in  \eqref{finaldecomp}. 
\begin{lemma}\label{lem:dt}
Notation as in \S\ref{sec:computingRHom}, Theorem \ref{mainth}. 
Let $t\in\Nscr$ and $q_t\in \{0,p-2\}$.

If $p=2$ then $d_t=0$.

If $j$ is odd and $p>2$ then $d_t$ runs through all odd numbers in the interval
\[
\begin{cases}
[p-2,(n-1)(p-2)]&\text{if $n$ is even},\\
[p-2,n(p-2)]&\text{if $n$ is odd}.
\end{cases}
\]

If $j$ is even then $d_t$ runs through all even numbers in the interval
\[
\begin{cases}
[0,n(p-2)]&\text{if $n$ is even},\\
[0,(n-1)(p-2)]&\text{if $n$ is odd}.
\end{cases}
\]
\end{lemma}

\begin{proof}
If $p=2$ there is nothing to prove so assume $p>2$. 

Note that the set of $t\in \Nscr$ satisfying $q_t\in \{0,p-2\}$ equals $\Sscr_0\cup \Sscr_{p-2}$ (see \S\ref{sec:combdesc}). 
Lemma \ref{lem:phijdef} applied with $J=\{1,\dots,n\}$ associates $\bar{t}=\Phi_J(t)\in \Sscr_0\cup \Sscr_{p-2}$ to every $t\in \Sscr_0\cup \Sscr_{p-2}$ such that $d_t+d_{\bar{t}}=n(p-2)$. 

Observe that $d_t$ (resp. $d_{\bar{t}}$) is odd if and only if $j$
(resp. $n-j-2$) is odd. This follows from the fact that in the
statement of \eqref{finaldecomp} $j$ and $q_t\in \{0,p-2\}$ have the
same parity and that moreover from Lemma \ref{rem:q_tvse_t}(1) 
it follows that
$q_t$ and $d_t$ also have the same
parity. 

In addition it follows from Lemma \ref{rem:q_tvse_t}(2)
that $q_t\le d_t$. 
Finally if $d_t$ is odd then $d_t\geq q_t=p-2$. 
Using these observations we obtain in the four listed cases the
indicated intervals are the only possible values of $d_t$.  For example,
assume that $n$ is even and $j$ is odd. Then $d_t$ is odd and thus
$\geq p-2$. This implies that $d_{\bar{t}}=n(p-2)-d_t$ is odd, and
$\leq (n-1)(p-2)$. Therefore for $n$ even and $j$ odd, $d_t$ is an odd
number lying in the interval $[p-2,(n-1)(p-2)]$.

It remains to show that all the values appear.  
Let $c$  belong to the (appropriate) interval (depending on the
parity of $j$, $n$).  
 Note that  we can (non-uniquely) write $c=\ell(p-2)+2r$ for some 
$0\leq r\leq p-2$, $0\leq \ell\leq n-2$. (Note that not all such $\ell,r$ are possible.) 
 As $L(0)$ is a direct
summand of $L(j)\otimes_k L(j)$ for any $0\le j \le p-2$ (see Lemma
\ref{rem:tnz}) it follows that $L(0)$ (resp. $L(p-2)$) is a direct
summand of $L(p-2)^{\otimes \ell}\otimes_k L(r)\otimes_k L(r)$ if $\ell$
is even (resp. odd).   This shows that $c$ equals
some $d_t$.
\end{proof}{}

\subsection{Indecomposable summands of $S^{G_1}$}
We gather the information collected in \S\ref{subsec:T}, \S\ref{subsec:degKj} in the following list.
\begin{corollary}\label{cor:indec} 
Up to multiplicity, the indecomposable summands of $S^{G_1}$ are
\begin{enumerate}
\item 
$S^p(-d)$, $d\in[0,2n(p-1)]$ even,
\item
for $1\leq j\leq n-3$, if $p \geq 1+\lceil j/(n-2-j)\rceil$
\[
T(j)^{\Fr}\otimes_k S^p(-d),\, d\in[p(j+2)-2,p(2n-2-j)-2n+2],\, jp\equiv d \,(2),
\]
\item
for odd $1\leq j\leq n-3$ 
\[\quad\;\,
\begin{cases}
K_j^{\Fr}(-d),\, d\in [p(j+3)-2,p(j+2+n-1)-2(n-1)]\, \text{even}& \text{if $n$ even},\\
K_j^{\Fr}(-d),\, d\in [p(j+3)-2,p(j+2+n)-2n]\ \text{even}& \text{if $n$ odd},
\end{cases}
\]
\item
for even $1\leq j\leq n-3$
\[
\begin{cases}
K_j^{\Fr}(-d),\, d\in [p(j+2),p(j+2+n)-2n]\, \text{even}& \text{if $n$ even},\\
K_j^{\Fr}(-d),\, d\in [p(j+2),p(j+2+n-1)-2(n-1)]\, \text{even}& \text{if $n$ odd}.
\end{cases}
\]
\end{enumerate}
\end{corollary}

\begin{proof}
By Corollary \ref{cor:freedeg}, we get the summands in (1), (2) in the desired degrees. 

Corollary \ref{cor:n-3} excludes $T(j)^{\Fr}\otimes_k S^p(-d)$ for $j >n-3$ as a tilt-free summand.  
Note that if $T(j)^{\Fr}\otimes_k S^p(-d)$ occurs in $T$ then $d$ and $jp$ have the same parity by \eqref{eq:tiltexp} and 
Lemma \ref{rem:q_tvse_t}(1). 
Thus, 
for $T(j)^{\Fr}\otimes_k S^p(-d)$ and $1\leq j \leq n-3$  other possible degrees $d$ that do not fall under (2) are excluded by Corollary \ref{lem:n-3}. 

Finally, (3) and (4) are immediate consequences of Lemma \ref{lem:dt}.
\end{proof}

\begin{remark}
\label{rem:onedegree}
Note that while $K_j^{\Fr}$ appears in at least one degree for every $n,p$, for $T(j)^{\Fr}\otimes_k S^p$  to appear $p\geq 1+\lceil j/(n-2-j)\rceil$ should be satisfied. 
\end{remark}

\subsection{Self-duality of the decomposition}
The decomposition \eqref{finaldecomp} is self dual in the following sense.

\begin{lemma}\label{lem:duality}
Tilting modules $T(j)^{\Fr}\otimes_k S^p(-d)$ and $T(j)^{\Fr}\otimes_k S^p(-2n(p-1)+d)$ appear as summands of $S^{G_1}$ with the same (possibly zero) multiplicity. 
The same holds for modules $K_j^{\Fr}(-d)$ and $K_{n-j-2}^{\Fr}(-2n(p-1)-2p+d)$. 
\end{lemma}

\begin{proof}
The last statement follows by dualizing $S^{G_1}$ (see \eqref{Sdual}) and Proposition \ref{prop:Kjproperties}\eqref{lem:Kjdual}. 

Let us now focus on the summand $T^{\Fr} \otimes_k S^p =\bigoplus_{j,d} (T(j)^{\Fr})^{\oplus a_{j,d}} \otimes_k S^p(-d)$. By dualizing $S^{G_1}$ 
combined with the last statement and the fact that tilting modules for $\SL_2$ are self-dual, we find that the tilt-free summand of $(S^{G_1})^\vee$ equals
$$
\bigoplus_{j,d} (T(j)^{\Fr} )^{\oplus a_{j,d}}\otimes_k S^p(-d +2n(p-1)).
$$
On the other hand, $(T^{\Fr}\otimes_k S^p)^{\vee}=\bigoplus_{j,d} (T(j)^{\Fr})^{\oplus a_{j,d}} \otimes_k S^p(d)$. The lemma follows by comparing the summands.
\end{proof}

\section{Decomposition of the invariant ring 
}\label{sec:15}
Here we state our main theorem which gives the decomposition of $R$ as
a module over $R^p$. Note that $R$ lives in even degrees. After taking the
$2$-Veronese we obtain the decomposition of the
homogeneous coordinate ring of $\Gr(2,n)$.  

The decomposition of $R$ follows immediately from Theorem \ref{mainth}  
by applying $(-)^{G^{(1)}}$. Recall that we denote $T\{j\}=M(T(j))$, and  
note that 
$K\{j\}=K_j^G$ by Proposition \ref{prop:cminv}.

\begin{theorem}\label{thm:mainR}
We have 
{\small
\begin{multline*}\label{eq:decR}
R\cong  \bigoplus_{j=1}^{n-3} \bigoplus_{\begin{smallmatrix}t\in\Nscr,q_t=0,j\equiv 0\,(2)\text{ or }\\ q_t=p-2,j\equiv 1\,(2)\end{smallmatrix}}
{K}\{j\}^{\Fr}(-jp-2p-d_t)^{\oplus n_t}\oplus
\bigoplus_
{\begin{smallmatrix}t\in\Mscr,q_t=0,\\ \sscr=(-)_\ell\, \text{or}\,\sscr=(+)_\ell\end{smallmatrix}}
T\{0\}^{\Fr}(-d_t^{\sscr})^{\oplus n_t}\\\oplus\bigoplus_{\begin{smallmatrix}t\in \Mscr,q_t-2p+2\equiv 0\,(p),\,q_t\neq p-2,\\\sscr\in \mathfrak{S}_t\end{smallmatrix}} T\{(q_t-2p+2)/p\}^{\Fr}(-d_t^{\sscr})^{\oplus n_t}.
\end{multline*}
}
\end{theorem}

\begin{remark}
\label{rem:found}
The indecomposable summands that appear  may be found by applying $(-)^{G^{(1)}}$ to the summands of $S^{G_1}$ listed in Corollary \ref{cor:indec}.
\end{remark}

\section{Pushforwards of the structure sheaf on the Grassmannian}
\subsection{Notation}
\label{ssec:induced}
In this section we introduce some general notation with regard to Grassmannians.
Let $F,V$ be respectively vector spaces of dimension $n$, $l$ with $n\ge l+1$. We let $\GG$ be the Grassmannian of $l$-dimensional quotients of $F$.
We fix a surjection $[F\r V]\in \GG$ and we use it to write $\GG$ as $\GL(F)/P$ 
where $P$ is the parabolic subgroup of $\GL(F)$ which is the stabilizer of the corresponding flag. The parabolic~$P$ has $\GL(V)$ as a Levi factor. 
For $U$ a $\GL(V)$-representation we write $\Lscr_{\GG}(U)$ for the $\GL(F)$-equivariant vector
bundle on~$\GG$ whose fiber in $[P]$ is $U$, considered as a $P$-representation.
Let 
\begin{equation}
\label{eq:tautseq}
0 \to \Rscr \to F\otimes_k \Oscr_{\GG} \to \Qscr \to 0
\end{equation} 
be the tautological exact sequence on $\GG$, where $\Qscr$ is the universal quotient bundle, and $\Rscr$ is the universal  subbundle. We then have $\Qscr=\Lscr_{\GG}(V)$. Also, put $\Oscr(1)=\wedge^l\Qscr$.
\subsection{Homogeneous bundles  and modules of covariants}
\label{sec:homogeneous}
In this section we clarify the relation
between homogeneous bundles on Grassmannians and modules of covariants on the corresponding  homogeneous coordinate rings.

In the next proposition we consider a $\GL(V)$-representation as a graded $\SL(V)$-representation via
the covering  $m:G_m\times \SL(V)\r \GL(V):(\lambda,g)\mapsto \lambda^{-1}g$.
\begin{proposition} \label{prop:diag2}
Let $W=F\otimes_k V^{\dur}$ and put
 $S=\Sym W$, $R=S^{\SL(V)}$ where $R,S$ are both considered as $\NN$-graded rings. 
Then we have a commutative diagram of functors
\[
\xymatrix{
\Rep(\GL(V)) \ar[rr]^{\Lscr_{\GG}(-)} \ar[d]_{m^\ast}&& \coh(\GG) \ar[d]^{\Gamma_*(\GG,-)} \\
\Rep_{\gr}(\SL(V)) \ar[rr]_{M(-)^{[l]}} && \gr(R^{[l]})
}
\]
where 
$
\Gamma_\ast(\GG,\Fscr)=\bigoplus_{m\in \ZZ} 
\Gamma(\GG,\Fscr(m)).
$
\end{proposition}
The proof of Proposition \ref{prop:diag2} is based on the following lemma.
\begin{lemma} \label{lem:indformula}
For $U\in \Rep(\GL(V))$ we have
\[
\Ind_P^{\GL(F)} U=(U\otimes_k \Oscr(\Hom(F,V)))^{\GL(V)}.
\]
\end{lemma}
\begin{proof}
According to \cite[\S I.3.3]{jantzen2007representations} we have
\[
\Ind_P^GU=(U\otimes_k \Oscr(\GL(F)))^P.
\]
Let $N$ be the kernel of $P\r\GL(V)$. So $N$ acts trivially on $U$. Thus
\begin{align*}
(U\otimes_k \Oscr(\GL(F)))^P&=((U\otimes_k \Oscr(\GL(F)))^N)^{\GL(V)}\\
&=(U\otimes_k \Oscr(\GL(F))^N)^{\GL(V)}\\
&=(U\otimes_k \Oscr(\GL(F)/N))^{\GL(V)}\\
&=(U\otimes_k \Oscr(S(F,V)))^{\GL(V)}
\end{align*}
where $S(F,V)$ is the space of linear surjections $F\r V$. Since $\dim F>\dim V$ 
the complement of $S(F,V)$ in $\Hom(F,V)$ has codimension $\ge 2$ and we obtain
$
\Oscr(S(F,V))=\Oscr(\Hom(F,V))
$,
finishing the proof.
\end{proof}
We also need the following easy observation.
\begin{lemma} \label{lem:SLGL}
Let $U$ be a $\GL(V)$-representation considered as graded $\SL(V)$-representation as above. Then we have as graded
vector spaces
\[
U^{\SL(V)}=\bigoplus_{m\in l\ZZ} ((\wedge^l V)^{\otimes m/l} \otimes_k U)^{\GL(V)}.
\]
\end{lemma}
\begin{proof} It suffices to consider the case that the center of $\GL(V)$
acts with weight $m$ on $U$ in which case
the conclusion is easy.
\end{proof}

\begin{proof}[Proof of Proposition \ref{prop:diag2}]
Note that $S=\Oscr(\Hom(F,V))$. 
So for 
 $U$  a $\GL(V)$ representation we obtain
\begin{align*}
M(U)&=(U\otimes_k \Oscr(\Hom(F,V)))^{\SL(V)}\\
&=\bigoplus_{m\in l\ZZ} ((\wedge^l V)^{\otimes m/l}\otimes_k U\otimes_k \Oscr(\Hom(F,V)))^{\GL(V)}&&\text{(Lemma \ref{lem:SLGL})}\\
&=\bigoplus_{m\in l\ZZ} \Ind^{\GL(F)}_P((\wedge^l V)^{\otimes m/l}\otimes_k U)&&\text{(Lemma \ref{lem:indformula})}\\
&=\bigoplus_{m\in l\ZZ} \Gamma(\GG,(\wedge^l \Qscr)^{\otimes m/l}\otimes_{\GG} \Lscr_{\GG}(U))
\end{align*}
Taking the $l$-Veronese on both sides finishes the proof.
\end{proof}

\subsection{Associated sheaves}\label{sec:sheaf} 
We keep the notation introduced in \S\ref{ssec:induced} but we assume again $\dim V=l=2$. In that case we have $V\cong V^\vee$ so that $S=\Sym(F\otimes_k V)$, with its
natural grading.
For $i\in \NN$ we upgrade $T(i)$ to a $\Gl(V)$-representation with highest weight $(i,0)$ and
we let $\Tscr_i=\Lscr_{\GG}(T(i))$. Note that if $p>i$ then we also have $\Tscr_i=S^i\Qscr$.

\begin{corollary}\label{cor:alggeo}
There are isomorphisms
\begin{enumerate}
\item\label{ena} $(M(S^iV)(i))^{[2]}\cong \Gamma_*(\GG,S^i\Qscr)$,
\item \label{dva} $(M(T(i))(i))^{[2]}\cong \Gamma_*(\GG,\Tscr_i)$,
\item\label{tri} $(K_i^G)^{[2]}\cong\Gamma_*(\GG,\wedge^i\Rscr(1))$.
\end{enumerate}
\end{corollary}

\begin{proof}
The isomorphisms \eqref{ena} and \eqref{dva} follow immediately from Proposition \ref{prop:diag2}. 
For \eqref{tri}, the proof follows similar lines as the proof of Proposition \ref{prop:Kjproperties}\eqref{lem:Kjdual}. 
By applying $\Gamma_*$ to the $i$-th exterior power of $\eqref{eq:tautseq}$\footnote{For Schur functors applied to complexes see \cite[\S 2.4]{weyman2003cohomology}.}  we get a left exact sequence 
\begin{equation}\label{eq:gammataut}
0\r \Gamma_*(\GG,\wedge^i\Rscr)\to \wedge^i F\otimes_k R^{[2]}\to \wedge^{i-1}F\otimes_k (M(V)(1))^{[2]}.
\end{equation}
In degree $0$ the rightmost map is determined by its part in degree zero which is
$\wedge^{i}F\r \wedge^{i-1}F\otimes_k M(V)_1=\wedge^{i-1}F\otimes_k F$. Since this map
is $\GL(F)$-equivariant it is unique up to a scalar (see the proof of Lemma \ref{lem:mapKj}).

From \eqref{eq:cores} we get a left exact sequence (after shifting, taking invariants, and passing to the 2-Veronese)
\[
0\r (K_i^G(-2))^{[2]}\r \wedge^i F\otimes_k R^{[2]}\r \wedge^{i-1}F\otimes_k (M(V)(1))^{[2]}
\]
whose rightmost map must be the same as the one in \eqref{eq:gammataut} (up to a scalar). It follows
\begin{equation}
\label{eq:i1}
(K_i^G(-2))^{[2]}= \Gamma_*(\GG,\wedge^i\Rscr),
\end{equation}
which is a variant of \eqref{tri}. 
\end{proof}

\subsection{Decomposition of {\boldmath $\Fr_*\Oscr_\GG$}}
\label{sec:decompg}
We are now ready to give the decomposition of $\Fr_*\Oscr_{\GG}$ into indecomposable summands.
\begin{theorem}\label{thm:grass}
{\tiny
\begin{multline*}\label{eq:decR}
\Fr_*\Oscr_{\GG}\cong  \bigoplus_{j=1}^{n-3} \bigoplus_{\begin{smallmatrix}t\in\Nscr,q_t=0,j\equiv 0\,(2)\text{ or }\\ q_t=p-2,j\equiv 1\,(2),\\
j/2+d_t/(2p)\in \NN\end{smallmatrix}}
\wedge^j\Rscr(-j/2-d_t/(2p))^{\oplus n_t}\oplus
\bigoplus_
{\begin{smallmatrix}t\in\Mscr,q_t=0,\\ \sscr=(-)_\ell\, \text{or}\,\sscr=(+)_\ell,\\d_t^{\sscr}/(2p)\in\NN\end{smallmatrix}}
\Oscr(-d_t^{\sscr}/(2p))^{\oplus n_t}\\\oplus \bigoplus_{\begin{smallmatrix}t\in\Mscr,q_t-2p+2\equiv 0\,(p),\,q_t\neq p-2,\\\sscr\in \mathfrak{S}_t,\\(q_t-2p+2+d_t^{\sscr})/(2p)\in\NN\end{smallmatrix}} \Tscr_{(q_t-2p+2)/p}((-q_t+2p-2-d_t^{\sscr})/(2p))^{\oplus n_t}.
\end{multline*}
}
\end{theorem}

\begin{proof}
Since $\Proj(R^p) \cong \Proj((R^p)^{[2p]})$, we obtain the decomposition of $\Fr_* \Oscr_{\GG}$ by decomposing $R^{[2p]}$ in $\Proj((R^p)^{[2p]})$. The theorem now follows by combining Theorem \ref{thm:mainR} with \S\ref{sec:sheaf}.
\end{proof}

We also list the indecomposable summands.
\begin{corollary}\label{cor:indecgrass}
Up to multiplicity, the indecomposable summands of $\Fr_*\Oscr_\GG$ are 
\begin{enumerate}
\item
$\Oscr(-d)$, $d\in[0, n-\lceil n/p\rceil]$, 
\item
for $1\leq j\leq n-3$, if $p\geq 1+\lceil (j+1)/(n-2-j)\rceil$
\[
\Tscr_j(-d),\, d\in[j+1, (n-1)-\lceil (n-1)/p\rceil],\, 
\]
\item 
for odd $1\leq j\leq n-3$, if $p>2$
\[
\begin{cases}
\wedge^j\Rscr(-d+1),\, d\in [(j+3)/2,(j+1+n)/2-\lceil (n-1)/p\rceil]& \text{if $n$ even},\\
\wedge^j\Rscr(-d+1),\, d\in [(j+3)/2,(j+2+n)/2-\lceil n/p\rceil]& \text{if $n$ odd},
\end{cases}
\]
\item
for even $1\leq j\leq n-3$
\[
\begin{cases}
\wedge^j\Rscr(-d+1),\, d\in [(j+2)/2,(j+2+n)/2-\lceil n/p\rceil]& \text{if $n$ even},\\
\wedge^j\Rscr(-d+1),\, d\in [(j+2)/2,(j+1+n)/2-\lceil (n-1)/p\rceil]& \text{if $n$ odd}.
\end{cases}{}
\]
\end{enumerate}
\end{corollary}

\subsection{(Non)tilting property of $\Fr_*\Oscr_\GG$}
From Corollary \ref{cor:indecgrass} we can easily deduce that $\Fr_*\Oscr_\GG$ is only rarely tilting. The positive answer for $n=4$  is a special case of Langer's result \cite[Theorem 1.1]{langer2008d} (see also \cite[Theorem 4]{achinger2012frobenius}).

\begin{corollary}
\label{cor:langer}
$\Fr_*\Oscr_{\GG}$ is tilting if and only if $n=4$ and $p>3$.
\end{corollary}

The proof will follow after a series of lemmas distinguishing summands of $\Fr_*\Oscr_\GG$ for different $n,p$. They all follow immediately from Corollary \ref{cor:indecgrass}.

\begin{lemma}\label{lem:spsum}
Let $p>2$, $n\geq 4$. Then  $\Qscr(-2)$ 
occurs in $\Fr_*\Oscr_{\GG}$. 
Moreover, if $(n,p)\not \in \{(4,*),(*,2),(5,3),(6,3),(7,3),(8,3)\}$ then $\Rscr(-2)$ occurs in $\Fr_*\Oscr_{\GG}$.
\end{lemma}

\begin{lemma}\label{lem:frob_sum_n=4}
Let $n=4$, $p>3$. The indecomposable summands of $\Fr_*\Oscr_\GG$ 
are $\{\Oscr,\Oscr(-1),\Oscr(-2),\Oscr(-3),\Qscr(-2),\Rscr(-1)\}$. 
\end{lemma}

\begin{lemma}\label{lem:smallnp}
Assume that $(n,p)\in \{(4,2),(4,3),(5,3),(6,3),(7,3),(8,3)\}$.
Up to multiplicity, there are $3,4,7,12,17,18$, respectively, indecomposable summands of $\Fr_*\Oscr_{\GG}$. 
If $p=2$ and $n>4$ is odd (resp. even) then $\Fr_*\Oscr_{\GG}$ has  $(-5+4n+n^2)/8$ (resp. $(-16+2n+n^2)/8$) 
indecomposable summands (up to multiplicity). 
\end{lemma}

\begin{proof}[Proof of Corollary \ref{cor:langer}]
If $(n,p)$ does not belong to the distinguished set in Lemma \ref{lem:spsum} then $\Fr_*\Oscr_{\GG}$ cannot be tilting since there is a tautological extension of $\Qscr(-2)$ by $\Rscr(-2)$.

If $n=4$ and $p>3$ then one can apply \cite[Theorem 4]{achinger2012frobenius}. Alternatively, one can argue directly using Lemma \ref{lem:frob_sum_n=4}, and note that 
$\{\Oscr,\Oscr(-1),\Oscr(-2),\Oscr(-3),\Qscr(-2),\allowbreak\Rscr(-1)\}$ is Kaneda's strong exceptional collection (see \S\ref{sec:kaneda} below). 

Comparing the rank of $K_0(\Gr(2,n))\cong \ZZ^{\binom{n}{2}}$ to the numbers in Lemma \ref{lem:smallnp} the claim follows also for the remaining cases.
\end{proof}

\subsection{Kaneda's exceptional collection}\label{sec:kaneda}
Let 
\begin{multline*}
E=\{\wedge^{k-1}\Rscr(-k+1)\mid 1\leq k\leq n-2\}\cup \{\Oscr(-n+1)\}\\\cup
\{S^i\Qscr(j-n+1)\mid i \geq 0, j\geq 1, i+j\leq n-2\}.
\end{multline*}
 
Kaneda \cite{kanedaML} 
shows that $E$ forms a full strong exceptional collection of $\GG$ whose elements are subquotients of $\Fr_*\Oscr_\GG$ for $p\gg 0$.  

\begin{corollary}\label{cor:kaneda}
Let $p\geq n$. Kaneda's exceptional collection $E$  is a direct summand of $\Fr_*\Oscr_\GG$. 
\end{corollary}

\begin{proof}
Immediate from Corollary \ref{cor:indecgrass}.
\end{proof}

\begin{remark}
Note that Kaneda's collection can be obtained from Kapranov's collection  $\{(S^i\Qscr)(j) \mid i,j\geq 0, i+j\leq n-2\}$ (see \cite{Kapranov}) by doing a series of block mutations 
(and shifting).
\end{remark}
\section{Proof of the NCR property}
We assume that we are in the standard $\Sl_2$-setting (\S\ref{sec:notation}). In this section we will prove the following result
which implies Theorem \ref{thm:ncr} using Corollary \ref{cor:summandlist}.
\begin{theorem} \label{thm:ncr1} The $R$-module
\[
M=\bigoplus_{i=0}^{n-3} T\{i\}\oplus \bigoplus_{i=1}^{n-3} K\{i\}
\]
defines an NCR for $R$.
\end{theorem}

We first prove a positive characteristic generalisation of some results in \cite{vspenko2015non}. We will frequently use the exactness of $(-)^G$ on representations with good filtration
(see Proposition \ref{prop:exact:good})
as well as the fact that the tensor product of representations with good filtration has a good filtration (see Proposition \ref{tensorgoodfiltration}). 

\begin{proposition}
\label{prop:pn-2}
Let $p\geq 0$. Put $T_i=T(i)\otimes_k S$, $N_0=\bigoplus_{i=0}^{n-3} T_i$. 
Then 
$\Lambda_0=\End_{G,S}(N_0)$ has finite global dimension. 
\end{proposition}
\begin{proof}
The case $p=0$ is covered by Theorem 1.6.1 (for suitable $\Delta$) 
in \cite{vspenko2015non}.
We only need to ensure that all steps in the proof remain valid in characteristic $p>0$. 
It turns out that in our specific case, although the result we want is not formally covered by the statement
of \cite[Theorem 1.5.1]{vspenko2015non}, it is sufficient  to put $\Delta=\{0\}$ and to use the arguments for the proof of  \cite[Theorem 1.5.1]{vspenko2015non} (see \cite[\S11.1-11.3]{vspenko2015non}).
For the convenience of the reader we compile here a list of notations (adapted to the current setting) that are used in
loc. cit.\ and that will accordingly be used below:
\begin{multline*}
\Ascr=\mod(G,S),\; X(H)\cong \ZZ,\;\chi\in X(H),\; P_{\chi}=T(\chi)\otimes_k S,\; S_\chi=T(\chi)\otimes_k S/S_{>0},\\ 
\Lscr=[0,n-3]\cap \ZZ\subset X(H),\; P_\Lscr=N_0,\; \Lambda_{\Lscr}=\Lambda_0,\;  \tilde{P}_{\Lscr,\chi}=\Ascr(P_\Lscr,P_\chi).
 \end{multline*}
Moreover since $Y(H)\cong \ZZ$ we will write $\Cscr_{\chi}$ for $\Cscr_{\lambda,\chi}$ with $\lambda\neq 0$ in \cite{vspenko2015non} and we note
that $\Cscr_{\chi}$ is just \eqref{biresolution0},\eqref{biresolution1} (without the $M_j$-part) for $j=\chi$.

  We first give the required modifications to the proof of \cite[Lemma 11.1.1]{vspenko2015non}
which asserts that $\Lambda_\Lscr$ has finite global dimension if $\tilde{P}_{\Lscr,\chi}$ has finite projective dimension as $\Lambda_\Lscr$-module
for all $\chi\in X(H)^+$.

In contrast to the situation in loc.\ cit.
 $\Ascr(P_\Lscr,S_\chi)$ (for $\chi\in \Lscr)$ is now not a simple
 $\Lambda_\Lscr$-module (unless $\chi\leq p-1$ in which case $S_\chi=L(\chi)\otimes_k S/S_{>0}$ is a simple $G$-module).  Therefore we proceed in two steps.  We first
 note that
 $\bar{\Lambda}_\Lscr=\Lambda_\Lscr/(\Lambda_\Lscr)_{>0}\allowbreak\cong \End_{G}(\oplus_{i=0}^{n-3}T(i))$
 has finite global dimension 
by \cite[Remark A.11]{jantzen2007representations} and
 \cite{Ringel}. 
 Therefore it
 suffices to show that the indecomposable projective $\bar{\Lambda}_\Lscr$-modules (which
 are of the form $\bar{P}_\chi=\End_{G}(\oplus_{i=0}^{n-3}T(i),T(\chi))$ for
 $0\leq \chi\leq n-3$) have finite projective dimension as
 $\Lambda_\Lscr$-modules. 
By using Lemma \ref{lem:bounded}  $T(\chi)$ has a bounded tilt-free resolution $\Kscr_\chi$.
Then using Propositions \ref{prop:exact:good},\ref{tensorgoodfiltration} $\Ascr(P_\Lscr,\Kscr_\chi)$ is exact and gives a resolution of $\bar{P}_\chi$. 
The terms 
in $\Ascr(P_\Lscr,\Kscr_\chi)$ are of the form $\tilde{P}_{\Lscr,\mu}$ for $\mu \in X(H)^+$. Hence it is sufficient that $\tilde{P}_{\Lscr,\mu}$
has finite projective dimension for $\mu\in X(H)^+$.

To prove the latter statement
 we
 describe the proof in \S11.2-11.3 in
 loc.\ cit, adapted to our current setting. The $\Lambda_\Lscr$-module $\tilde{P}_{\Lscr,\chi}$ is
 projective for $\chi\leq n-3$. For $\chi\geq n-2$ we use the complex
 \eqref{biresolution0}, \eqref{biresolution1} without the last term
 $\Cscr_\chi$ giving a resolution of $M_\chi$ and we replace it with a tilt-free resolution $\tilde{\Cscr}_\chi$ of $M_\chi$ as in \S\ref{resolutions}
(under the condition $\chi\ge n-2$ all terms in $\Cscr_\chi$ are $\nabla$-free).
It follows from the procedure in \S\ref{resolutions} that
 all terms in $\tilde{\Cscr}_\chi$ consist of $P_{\chi'}$ for
 $\chi'<\chi$ except for the term in degree zero which equals $P_\chi$.
 Therefore, we obtain by induction that 
 $\tilde{P}_{\Lscr,\chi}$ has finite projective dimension.
\end{proof}
Put $N_{i+1}=N_i\oplus K_{i+1}$ for $i=0,\ldots,n-4$. 
Now we will prove by induction on~$i$ that
$
\Lambda_i=\End_{G,S}(N_i)
$
also has finite global dimension. As we have the equality
$
\End_R(M)=\End_{G,S}(N_{n-3})
$
this finishes the proof of Theorem \ref{thm:ncr1}.

\medskip

As usual the starting point is the exact sequence \eqref{biresolution0} for $1\leq j\leq n-3$. 
We break this exact sequence up into two shorter exact sequences
\begin{equation}
\label{eq:seq1}
0\r [\wedge^n F\otimes_k D_{n-j-2}(-n)\r \cdots \r\wedge^{j+2}F\otimes_k D_0(-j-2)]\r K_j(-j-2)\r 0
\end{equation}
\begin{equation}
\label{eq:seq2}
0\r K_j(-j-2)\r [\wedge^j F\otimes_k S_0(-j)\r\cdots \r F\otimes_k S_{j-1}(-1)\r S_j] \r M_j \r 0,
\end{equation}
where $D_j=D^jV\otimes_k S$, $S_j=S^jV\otimes_k S$, 
and we denote the bracketed parts by $\Dscr_j,\Gscr_j$ respectively. 
By Proposition \ref{prop:Mjproperties}\eqref{Mjproperties4}  we have $\Dscr_j=\Gscr_{n-j-2}^\vee\otimes_k \wedge^n F(-n)$. 
Let $\tilde{\Gscr}_j$ be a tilt-free version of  $\Gscr_j$ constructed as in \S\ref{resolutions}.  
Put $\tilde{\Dscr}_j=\tilde{\Gscr}_{n-j-2}^\vee\otimes_k \wedge^n F(-n)$. Splicing together $\tilde{\Gscr}_j$
and $\tilde{\Dscr}_j$
 yields a bounded tilt-free
resolution  $\tilde{\Cscr}_j$ of $M_j$ with the following properties
\begin{enumerate}
\item \label{Cprops1} $Z^j(\tilde{\Cscr}_j)=K_j$;
\item \label{Cprops2} $\tilde{\Cscr}_j$ is a sum of $T(i)\otimes_k S$ for $0\le i\le \max(j,n-j-2)\le n-3$;
\item \label{Cprops3} $\tilde{\Cscr}_j$ has length $n-1$;
\item \label{Cprops4} $\tilde{\Cscr}^{-n+1}_j=V^{\otimes n-j-2}\otimes_k S\otimes_k \wedge^n F(-n)$.
\end{enumerate}
\begin{lemma} \label{lem:fundex} Let $p\geq 0$. Let $L$ be either $T_l$ for arbitrary $l$ or $K_l$ for $l\le j$. Then $\Hom_{G,S}(L,-)$ 
applied to $\tilde{\Dscr}_j$
yields an exact sequence
\begin{multline}
\label{eq:fundex}
0\r \Hom_{G,S}(L,\tilde{\Dscr}_j^{-(n-j-2)})\r \cdots \r\Hom_{G,S}(L,\tilde{\Dscr}_j^0)\r \\
\Hom_{G,S}(L,K_j(-j-2))\r \Sscr(-j-2) \r 0
\end{multline}
where $\Sscr=k$ if $L=K_j$ and $\Sscr=0$ in all other cases.
\end{lemma}
\begin{proof}
Using \eqref{Cprops1}  it is easy to see that
we have to prove
\begin{equation}
\label{eq:havetoprove}
H^i(\Hom_{G,S}(L(-l-2),\tilde{\Cscr}_j))=
\begin{cases}
0&\text{if $i<-j$,}\\
\Sscr&\text{if $i=-j$.}
\end{cases}
\end{equation}
If $L=T_l$ then $\Hom_{G,S}(L,\tilde{\Cscr}_j)$ is acyclic in degrees $<0$ using Propositions \ref{prop:exact:good},\ref{tensorgoodfiltration}
and hence there is nothing to prove. We now assume $L=K_l$. 

Clearly $\Hom_{G,S}(\tilde{\Dscr}_l,M_j)$ has zero cohomology in degrees $<0$. 
Hence the same holds for the total complex associated to the double 
complex $\Hom_{G,S}(\tilde{\Dscr}_l,\tilde{\Cscr}_j)$. We consider the associated $E_1$-spectral sequence (with $\tilde{\Cscr}_j$ horizontally oriented). By the prior discussion we have
$E^{uv}_\infty=0$ for $u+v<0$. To compute the $E_1$-term
 we have to compute the cohomology of 
\[
\Hom_{G,S}(\tilde{\Dscr}_l,T\otimes_k S)=(\tilde{\Dscr}_l^\vee \otimes_S T\otimes_k S)^G=(\tilde{\Gscr}_{n-l-2}\otimes_k T\otimes_k \wedge^n F^*(n))^G[-(n-l-2)]
\]
for $T$ tilting (using \eqref{Cprops2}).
We find, using again Propositions \ref{prop:exact:good}, \ref{tensorgoodfiltration} 
 for the exactness of $(-)^G$,
\begin{equation}
\label{eq:cohomE1columns}
H^t(\Hom_{G,S}(\tilde{\Dscr}_l,T\otimes_k S))=
\begin{cases}
\Hom_{G,S}(K_l(-l-2),T\otimes_k S)&\text{if $t=0$},\\
(M_{n-l-2}\otimes_k T\otimes_k \wedge^n F^*(n))^G&\text{if $t=n-l-2$},\\
=0&\text{otherwise}.
\end{cases}
\end{equation}
It follows  that $E_1^{uv}=0$ when $v\neq 0,n-l-2$ and 
$
E_1^{*0}=\Hom_{G,S}(K_l(-l-2),\tilde{\Cscr}_j)
$ which is the complex occurring on the left-hand side of \eqref{eq:havetoprove} (as $L=K_j$).
Moreover by \eqref{Cprops3} we also find $E^{uv}_1=0$ for $u\not\in [-n+1,0]$.
Finally for $l=j$ we compute:
\begin{align*}
E_1^{-n+1,n-j-2}&=(M_{n-j-2}\otimes_k V^{\otimes n-j-2}\otimes_k\wedge^n F(-n)\otimes \wedge^n F^*(n))^G
&&\text{(\eqref{eq:cohomE1columns}, \eqref{Cprops4})}, \\
&=\bigoplus_t (S^{n-j-2+t}V \otimes_k S^tF(-t)  \otimes_k V^{\otimes n-j-2})^G &&\text{Prop.\ \ref{prop:Mjproperties}\eqref{Mjproperties3}},\\
&=\bigoplus_t S^tF(-t)^{[V^{\otimes n-j-2}:S^{n-j-2+t}V]}&&\text{Prop.\ \ref{basicdeltanabla}(\ref{basicdeltanabla1},\ref{basicdeltanabla2})}, \\
&=k.
\end{align*}
Given that $E_\infty^{uv}=0$ for $u+v<0$
we easily obtain \eqref{eq:havetoprove}.
\end{proof}
\begin{remark} \label{rem:spec}
Note that $K_j$ is indecomposable by Proposition \ref{prop:Kjproperties}\eqref{prop:kjindec}  and hence $\End_{G,S}(K_j)$ is a
graded local ring.
When $L=K_j$
the rightmost map in \eqref{eq:fundex} is the map $\End_{G,S}(K_j,K_j)\r \End_{G,S}(K_j,K_j)/m$ where $m$ is the graded maximal ideal. 
Indeed
the image of $\wedge^{j+2}F\otimes_k \Hom_{G,S}(K_l,\tilde{D}_j^0)$ in $\End_{G,S}(K_j,K_j)$ is a right ideal hence it must be contained in the maximal ideal. Since the codimension is one it must actually be equal 
to the maximal ideal.
\end{remark}
We will use the following result proved in \cite[Theorem 5.4]{APT}. If $A$ is a ring and $e\in A$ is an idempotent then
\[
\gldim A\le \gldim(eAe) + \gldim (A/AeA) + \pdim_A (A/AeA) + 1
\]
We will apply this with $A=\Lambda_{i+1}$ and $e$ the projection of $N_{i+1}$ onto $N_i$. It is convenient to put $\Lambda_{i+1}$ into block matrix form
\[
\begin{pmatrix}
\Lambda_i & \Hom_{G,S}(K_{i+1},N_i)\\
 \Hom_{G,S}(N_i,K_{i+1})&\End_{G,S}(K_{i+1})
\end{pmatrix}
\]
We have $e\Lambda_{i+1}e=\Lambda_i$ and hence by induction we may assume $\gldim e\Lambda_{i+1}e<\infty$.
The ideal $AeA$ is of the form
\[
AeA=\begin{pmatrix}
\Lambda_i & \Hom_{G,S}(K_{i+1},N_i)\\
 \Hom_{G,S}(N_i,K_{i+1})&I
\end{pmatrix}
\]
where $I$ 
 is the set of maps $K_{i+1}\r K_{i+1}$ that factor through $N_i$. It follows from \eqref{eq:fundex} applied with $L=K_j$ and Remark \ref{rem:spec} that $I$ is the graded maximal ideal of $\End_{G,S}(K_{i+1})$. In other words
\[
A/AeA=\begin{pmatrix}
0 & 0\\
0 & k
\end{pmatrix}
\]
It remains to construct a projective resolution for the simple right $A$-module given by $(0\ k)$. It follows from Lemma \ref{lem:fundex} that such a resolution can be obtained from $\Hom(N_{i+1},\tilde{\Dscr}_{i+1})$ (using also \eqref{Cprops2})
 finishing the proof.
 \begin{remark} For the benefit of the reader we will say something
   here about the validity of the NCR construction \cite[Theorem
   1.5.1]{vspenko2015non} in finite characteristic for general
   reductive groups. By inspecting the proof of \cite[Theorem
   1.5.1]{vspenko2015non} we find that it remains valid for $p\ge B$
   where $B$ is  explicitly computable. Indeed in
   \cite[Lemma 11.1.1]{vspenko2015non} we actually only need to prove
   $\pdim_{\Lambda_{\Lscr}} \tilde{P}_{\Lscr,\mu}<\infty$ for $P_\mu$ that
     appear in the Koszul resolution of $(S_\chi)_{\chi\in \Lscr}$, where $S_\chi=\nabla(\chi)\otimes_k SW/SW_{>0}$ (assuming that $\Lscr$ is contained in the fundamental alcove so that $S_\chi$ is a simple $G$-module). Denote the (finite) set of such $\mu$ by $S$. 
     We take some $r\geq 1$ such that 
$S\subset -\bar{\rho}+\Delta+r\Sigma$ and then take $p$ such that $(-\bar{\rho}+\Delta+r\Sigma)\cap X(T)^+$ is contained in the fundamental alcove  such that in addition $\Sym(W)$ has a good filtration, and apply the  argument in \cite[Lemma 11.2.1]{vspenko2015non} (using \cite[Corollary II.5.5]{jantzen2007representations} for Bott's theorem) together with \cite[\S 11.3]{vspenko2015non}. Similar considerations apply
to the other main results of \cite{vspenko2015non}.
\end{remark}

\appendix
\section{Cohen-Macaulay modules of covariants}
\subsection{Introduction}
In Theorem \ref{thm:mainR} (and Remark \ref{rem:found}) we showed that some of the $R$-modules $T\{m\}=M(T(m))$ for $m=0,\ldots,n-3$ occur as Frobenius summands (for large $p$ they all appear).
In particular those that occur  are Cohen-Macaulay. In this appendix we provide some context for this result by comparing it to known results about Cohen-Macaulayness of modules of covariants in 
characteristic zero (e.g.\ \cite{van1991cohen,MR1181176}). We will also explore how these results might extend to finite characteristic. Specifically for our standard $\SL_2$-setting (\S\ref{sec:notation}) we obtain the following precise result valid in arbitrary characteristic.
\begin{proposition}\label{prop:cm}
The $R$-module $T\{m\}$ is a Cohen-Macaulay module if and only if $0\le m \leq n-3$. 
\end{proposition}
Note that the ``if'' part of this result would be false if we replaced $T\{m\}$ by  $M(S^m V)$ as shown in the following example.
\begin{example}
\label{ex:nonsplit}
Let $p=2$, and $n\ge 4$ arbitrary. In that case the short exact sequence of $\Sl_2$-representations 
\[
0\r k\r V\otimes_k V\r S^2 V\r 0
\]
is not split.
Tensoring with $S$ and taking invariants, and using the fact that $S$ has a good filtration
we get an exact sequence of modules of covariants 
\[
0\r R\r M(V\otimes_k V)\r M(S^2V)\r 0
\]
which is also not split by \cite[Lemma 4.1.3]{vspenko2015non}. On the other hand if   $M(S^2V)$ were Cohen-Macaulay then
we would have $\Ext^1_R(M(S^2V),R)=0$ which is a contradiction.
\end{example}
The ``only if'' part of Proposition \ref{prop:cm} follows by a general lifting to characteristic zero argument which is explained in \S\ref{sec:liftcm}.
The ``if''  part on the other hand can be proved for coordinate rings of arbitrary Grassmannians so that is what we will do in this appendix.
However we first provide some more context in the next section.
\subsection{Strongly critical weights}
\begin{definition}
Let $G$ be a reductive group with maximal torus $H$ acting on a representation $W$ with weights $\beta_1,\ldots,\beta_d\in X(H)$. Let $\Sigma=\{\sum_{i=1}^d a_i \beta_i\mid a_i\in ]-1,0]\}\allowbreak\subset X(H)_{\RR}$
and let $\rho$ be half the sum of the positive roots.
The elements of the intersection 
\[
X(H)^+\cap(-2\rho+\Sigma)
\] 
are called \emph{strongly critical (dominant) weights} for $G$.
\end{definition}
The main result about strongly critical weights is the following:
\begin{theorem}[{\protect\cite[Theorem 1.3]{van1991cohen}, \cite[Theorem 4.4.3]{vspenko2015non}}] \label{th:sc}
If the ground field has characteristic zero, $W^{\dur}$ has a $H$-stable point, and $\alpha\in X(H)^+\cap(-2\rho+\Sigma)$ is strongly critical then $M(L({\alpha})^\dur)$  is Cohen-Macaulay.
\end{theorem}
In this appendix we give some evidence for the following reasonable conjecture.
\begin{conjecture} \label{conj:sc} Theorem \ref{th:sc} holds in arbitrary characteristic provided $\Sym(W)$ has a good filtration and we replace $L(\alpha)$ by $T(\alpha)$.
\end{conjecture}
 If this conjecture is true then by Proposition \ref{prop:mainlift} below, whenever $\alpha$ is a strongly critical dominant weight and $(T(\alpha):\nabla(\beta))\neq 0$ then $M(L(\beta)^\ast)$ should
be Cohen-Macaulay in characteristic zero. As we have $\beta\le \alpha$ this indeed follows  from the following observation.
\begin{lemma} Assume that $W^\dur$ has a $H$-stable point and let $\beta\le \alpha$ be dominant weights. If $\alpha$ is strongly critical then so is $\beta$.
\end{lemma}
\begin{proof}
Let $Y(H)$ be the one-parameter subgroups of $H$.
It follows like in the proof of \cite[(11.4)]{vspenko2015non} that there are $\lambda_1,\ldots,\lambda_u\in Y(H)^+$, $c_1,\ldots,c_u\in \RR$ such that 
\begin{equation}
\label{eq:sigmasc}
(-2\rho+\Sigma)\cap X(H)^+=X(H)^+\cap \bigcap_i \{\chi\in X(H)^+\mid \langle \lambda_i,\chi\rangle< c_i\}.
\end{equation}
The right-hand side of \eqref{eq:sigmasc} is clearly invariant under translation by negative roots.
\end{proof}
\begin{remark} \label{rem:cmordering}
Assume the ground field has characteristic zero. If $\beta\le \alpha$ are dominant weights and 
$M(L(\alpha)^\ast)$ is Cohen-Macaulay (but $\alpha$ is not strongly critical) then it is not necessarily true that $M(L(\beta)^\ast)$ is Cohen-Macaulay.
We will give a concrete counterexample in Remark \ref{cm:notvanishing} below.
\end{remark}
\subsection{Main results}
\begin{theorem} \label{thm:sc} Conjecture \ref{conj:sc} holds in the following two situations.
\begin{enumerate}
\item $G=\SL(V)$, $W=F\otimes V^{\dur}$, $\dim F>\dim V$.
\item  $G=\GL(V)$, $W=F_1^\dur\otimes_k V\oplus F_2\otimes_k V^{\dur}$, $\dim F_1>\dim V$, $\dim F_2>\dim V$.
\end{enumerate}
\end{theorem}
Let $B_{u,v}$ be the set of partitions with at most $u$ rows and $v$ columns.
We will prove Theorem \ref{thm:sc} 
using the geometry of suitable Grassmannians.
For a partition $\alpha=[\alpha_1,\ldots,\alpha_k]$ put 
$$
\wedge^{\alpha} \Qscr=\wedge^{\alpha_1} \Qscr \otimes_{\GG} \cdots \otimes_{\GG} \wedge^{\alpha_k} \Qscr.
$$
Also, let $L^{\alpha}(-)$ denote the Schur functor associated to $\alpha$. 
 Recall the following vanishing result from \cite{MR3371493}. 
\begin{proposition}\cite[Proposition 1.4]{MR3371493}
\label{BLVdB-vanishing}
For $\alpha \in B_{l,n-l}$ and $\beta$ an arbitrary partition we have for $\GG=\Gr(l,n)$ and $i>0$ 
$$
\Ext^i_{{\GG}}(\wedge^{\alpha'}\Qscr,L^{\beta} \Qscr)=0,
$$
where $\alpha'$ is the conjugate partition of $\alpha$.
\end{proposition}

\begin{proof}[Proof of Theorem \ref{thm:sc}(1)]
We are now in the setting of \S\ref{sec:homogeneous}. In particular we have $n=\dim F$, $l=\dim V$.
The set of \emph{strongly critical (dominant) weights} for $(\SL(V),W)$ 
equals $B_{l,n-2l+1}$.
which can be easily checked as in the proof of \cite[Proposition 5.2.2.]{vspenko2015non} using the fact that the weights of $W$ are $(-L_i)_{i=1}^n$, each weight occurring with multiplicity $n$. 

By \eqref{tiltingduality} $T(\alpha)^\dur=T(-w_0\alpha)$ and since $W$ is unimodular it is easy to see that $\alpha\mapsto\! {-}w_0\alpha$ permutes the dominant strongly critical weights. 
Hence it suffices to show that $M(T(\alpha))$ is Cohen-Macaulay for $\alpha\in B_{l,n-2l+1}$ and to this end it suffices to show that the associated bundles $\Tscr_\alpha=\Lscr_\GG(T(\alpha))$ (see Proposition \ref{prop:diag2}) for $\alpha\in B_{l,n-2l+1}$ on $\GG=\Gr(l,F)$ satisfy:
\begin{equation}
\label{intermediatevanishing}
\HHH^i(\GG,\Tscr_\alpha(j))=0, 
\end{equation}
for $0 < i < d=l(n-l)$ and $j \in \ZZ$. 
Since by \cite[Lemma (3.4)]{donkin93} 
$T(\alpha)$ is a direct summand of $\wedge^{\alpha'}V$ (defined in analogy to $\wedge^{\alpha'}\Qscr$), $\Tscr_\alpha$ is a direct summand of  $\wedge^{\alpha'}\Qscr$. Thus, 
it is enough to prove the vanishing for $\wedge^{\alpha'}\Qscr$. 

Let us first recall 
\begin{align}\label{eq:t1}
\wedge^{l}\Qscr&\cong \Oscr_{\GG}(1),\\\label{eq:t2}
(\wedge^i \Qscr)^{\vee}&\cong \wedge^{l-i}\Qscr(-1),\\\label{eq:t3}
\omega_{\GG}&\cong \Oscr_{\GG}(-n).
\end{align}
For $j \geq -n+l$, we compute:
\begin{align*}
\HHH^i(\GG,\wedge^{\alpha'}\Qscr(j))&\overset{\mathclap{\eqref{eq:t1}}}{\ \ \ \cong\ \ \ }\Ext_{\GG}^i((\wedge^{\alpha'}\Qscr)^{\vee} \otimes_{\GG} (\wedge^{l} \Qscr)^{\otimes n-l},\Oscr(j+n-l))\\
&\overset{\mathclap{\eqref{eq:t2}}}{\ \ \ \cong\ \ \ \ } \Ext_{\GG}^i(\wedge^{\bar{\alpha}'}\Qscr,L^{\beta}\Qscr)\\
&\overset{\mathclap{\text{Prop. \ref{BLVdB-vanishing}}}}{\ \ \ \cong\ \ \ } 0 
\end{align*}
for   
partition $\bar{\alpha}=[n-l-\alpha_l,\ldots,n-l-\alpha_1]\in B_{l,n-l}$ and $\beta=[(j+n-l)^l]$. 
 Similarly, for $j<-n+l$ we have by Serre duality (using \eqref{eq:t3})
\begin{align*}
\HHH^i(\GG,\wedge^{\alpha'}\Qscr(j))&
\cong\Ext_{\GG}^{d-i}(\wedge^{{\alpha}'}\Qscr,\Oscr_{\GG}(-n-j))^{\vee}\\
&\cong \Ext_{\GG}^{d-i}(\wedge^{{\alpha}'}\Qscr \otimes_{\GG} (\wedge^{l}\Qscr)^{\otimes_k n-l-\alpha_1},\Oscr(-j-l-\alpha_1))^{\vee}\\
&\cong \Ext^{d-i}_{\GG}(\wedge^{\gamma'}\Qscr,L^{\delta}\Qscr)^{\vee}\\
&\cong 0
\end{align*}
where $\gamma=[n-l,\alpha_2+n-l-\alpha_1,\dots,\alpha_l+n-l-\alpha_1] \in B_{l,n-l}$ and $\delta=[(-j-l-\alpha_1)^l]$ (noting that $-j-l-\alpha_1>n-2l-\alpha_1\geq -1$ as $\alpha\in B_{l,n-2l+1}$).
\end{proof}
\begin{proof}[Proof of Proposition \ref{prop:cm}]
One direction follows from Theorem \ref{thm:sc}(1).
To prove the converse we use Proposition \ref{prop:mainlift} below which shows that the Cohen-Macaulay property lifts to characteristic zero. Since $G$, $V$ are defined over $\ZZ$ they are defined over any commutative ring and in particular over the Witt vectors $A=W(k)$.
This is a complete discrete valuation ring of unequal characteristic with residue field $k$. For $K$ we take the algebraic closure of the quotient field of $A$. If $T\{m\}$ is Cohen-Macaulay
then Proposition \ref{prop:mainlift} implies that $M(S^mV)$ is Cohen-Macaulay in characteristic zero. Hence $m\le n-3$ by \cite{MR1181176}.
\end{proof}
\begin{remark}
\label{cm:notvanishing}
  Note that for $l>2$, in the setting of Theorem \ref{thm:sc}(1), the set of strongly-critical weights differs from the set of Cohen-Macaulay weights in characteristic
  $0$ (see e.g. \cite[Theorem 1.3]{van1991cohen} or the classification
  of homogeneous ACM vector bundles on Grassmannians \cite{ACMgrass}).
  One might be tempted to use Proposition \ref{BLVdB-vanishing} to
  produce bigger sets of Cohen-Macaulay modules of covariants also in
  characteristic $p$.  However, even if $M(L(\alpha))$ is
  Cohen-Macaulay in characteristic $0$, $\wedge^{\alpha'}V$ may
  contain non-Cohen-Macaulay summands in characteristic $0$ and
  consequently by Proposition \ref{prop:mainlift} also in characteristic $p$. A
  concrete example is given by $\alpha=[2n-2l,n-l]$ in which case
 $\wedge^{\alpha'}V$ contains $L(\beta)$ as a direct summand for $\beta=[2n-2l,n-l-1,1]$. It follows from \cite[Proposition 3.9]{ACMgrass} that $M(L(\beta))$ is not Cohen-Macaulay.

This example can also be regarded as illustrating Remark \ref{rem:cmordering}. Indeed, while $-w_0\beta\le -w_0\alpha$ (as $-w_0(-)$ is order preserving) and 
$M(L(-w_0\alpha)^\ast)=M(L(\alpha))$ is Cohen-Macaulay, $M(L(-w_0\beta)^\ast)=M(L(\beta))$ is not.
\end{remark}
\begin{proof}[Proof of Theorem \ref{thm:sc}(2)] Since we are mainly recycling arguments from \cite{buchweitz2011non} we will only sketch the proof. 
Let $m=\dim F_1$, $n=\dim F_2$, $\dim V=l$. 
A small computation (similar as in the proof of   
\cite[Proposition 5.2.2.]{vspenko2015non}) shows that the strongly critical dominant weights are given by the set $B_{l,[-m+l,n-l]}=\{\alpha\mid n-l\geq\alpha_1\geq \alpha_2\geq\allowbreak\dots\geq\alpha_l\geq -m+l\}$. Below we show that if $\alpha\in B_{l,[-m+l,n-l]}$ then $M(T(\alpha)^\dur)$ is Cohen-Macaulay. 

Let $\alpha\in B_{l,[-m+l,n-l]}$. 
Since $\alpha$ is not a partition, we need to adapt the definition of $\wedge^{\alpha'}V$. 
Set $\beta=(\alpha_1-\alpha_l,\ldots,\alpha_{l-1}-\alpha_l,0)$ and 
 put $\wedge^{\alpha'} V:=(\wedge^l V)^{\otimes \alpha_l} \otimes_k \wedge^{\beta'} V$. 
 By \cite[Lemma (3.4)]{donkin93}, $\wedge^{\alpha'} V$ is a tilting module. 
The highest weight of $\wedge^{\alpha'} V$ is equal to $\alpha$ and in particular  $\wedge^{\alpha'} V$ contains $T(\alpha)$
as a summand. It is therefore enough to prove that $M((\wedge^{\alpha'}V)^*)$ is Cohen-Macaulay.
 
To show that $M((\wedge^{\alpha'}V)^*)$ for $\alpha\in B_{l,[-m+l,n-l]}$ is Cohen-Macaulay we employ the same method as in \cite[\S 3]{buchweitz2011non}. 
Let $\GG=\Gr(l,F_1^\ast)$ be the Grassmannian of $l$-dimensional quotients of 
$F_1^\ast$ and $\Zscr=\Sym_{\GG}(F_2\otimes \Qscr)$ where $\Qscr$ is the the universal quotient bundle on $\GG$ (in loc. cit. $F_1,F_2$ are denoted by $F,G$, respectively).  Then combining \cite[Lemma 5.3.2]{vspenko2015non} (replacing $S^\alpha V$ by $\wedge^{\alpha'}V$ and dropping the unnecessary assumption $\dim F_1=\dim F_2$) with \cite[Proposition 3.5]{buchweitz2011non} (also making similar modifications to the setting and hypotheses) we find
\[
M((\wedge^{\alpha'} V)^\ast)=\Gamma(\GG, \wedge^{\alpha'} \Qscr\otimes_{\GG}\Zscr).
\]
To prove that $M((\wedge^{\alpha'} V)^\ast)$ is Cohen-Macaulay we must then show (see \S 3 and in particular the proof of Proposition 3.4. in loc. cit. for details) for $i>0$
\[
H^i(\GG, \wedge^{\alpha'} \Qscr\otimes_{\GG}\Zscr)=0,\;
\Ext^i_{\GG}(\wedge^{\alpha'} \Qscr,\Zscr\otimes_{\GG}(\wedge^l\Qscr)^{\otimes n-m})=0.
\]
The verification 
follows similar steps as the proof of Theorem \ref{thm:sc}(1), therefore we omit it.
\end{proof}
\subsection{Lifting Cohen-Macaulayness to characteristic zero.}
\label{sec:liftcm}
In this section $A$ is a discrete valuation ring with uniformizing element
$\pi$. Let $k=A/(\pi)$ and let $K$ be an arbitrary field extension
of the quotient field of $A$. We prove the following result.
\begin{proposition}
\label{prop:mainlift}
Assume that $A$ is complete, $\bar{k}=k$, $\operatorname{char} K=0$, $\bar{K}=K$. 
Let~$G$ be a  reductive group defined over $A$ and let $H$ be a split maximal torus\footnote{By \cite[Corollary 3.2.7]{conrad} the existence of a maximal torus is \'etale local over $A$. Since $A$ is complete with algebraically closed
residue field it is strictly Henselian (local for the \'etale topology) and hence~$G$ has a maximal torus $H$ defined over $A$. Since the property of being split is also \'etale local,
the same reasoning yields that
$H$ is split.} in $G$. 
 Let~$W$ be a  $(G,A)$-module which is free of finite rank as $A$-module.
Put $S=\Sym_A(W)$ and assume that $S_k$ has a good filtration.\footnote{We decorate notations with $k$, $K$ to indicate over which field they are defined.}

Let $\alpha\in X(H_k)^+=X(H_K)^+$ be such that $M(T_k(\alpha))$ is Cohen-Macaulay.
Then for every composition factor $\nabla_k(\beta)$ of $T_k(\alpha)$ we have
that $M(\nabla_K(\beta))$ is Cohen-Macaulay.
\end{proposition}
To prove this we use some lemmas.
\begin{lemma} 
\label{lem:cm}
Let $R=R_0\oplus R_1\oplus\cdots$ be a commutative finitely generated 
graded $A$-algebra such that $R_0=A$.
Let $M$ be a finitely generated graded $R$-module which is flat (or equivalently torsion free) over $A$.
If $M_k$ is Cohen-Macaulay then so is $M_K$.
\end{lemma}
\begin{proof} Assume $\dim M_k=n$. For example using Hilbert functions it follows that $\dim M_K=n$.
Let $(x_1,\ldots,x_n)\in R_k^n$ be a homogeneous regular sequence for $M_k$ in strictly positive degree and let 
 $(\tilde{x}_1,\ldots,\tilde{x}_n)$ be an arbitrary lift to $R$. Then $(\pi,\tilde{x}_1,\ldots,\tilde{x}_n)$
is a regular sequence for $M$ and hence so is $(\tilde{x}_1,\ldots,\tilde{x}_n)$. 

Since $K/A$ is flat $(\tilde{x}_1,\ldots,\tilde{x}_n)$ remains a homogeneous regular sequence in $M_K$ of length $n$ in strictly positive degree. Hence we are done.
\end{proof}

\begin{lemma}
\label{lem:basechange}
Let $G$ be a reductive group over $A$ and let $M$ be a  $(G,A)$-module which is finitely generated
free as $A$-module. Assume $\bar{k}=k$ and $\operatorname{char} K=0$. Assume that $M_k$ has a
good filtration as $G_k$-representation. Then $(M^{G})_k=(M_k)^{G_k}$, $(M^{G})_K=(M_K)^{G_K}$.
\end{lemma}

\begin{proof} As in the proof of \cite[Theorem B.3]{crawley2004absolutely} (see equations (B.1)(B.2)).
\end{proof}

\begin{lemma}
\label{lem:liftingcm}
Let $G$ be a reductive group over $A$ and let
$S=S_0\oplus S_1\oplus \cdots$ be a graded $A$-algebra equipped with a
rational $G$-action such that~$S_0=A$. Let $M$ be a graded $(G,S)$-module which is
finitely generated as $S$-module. Assume $\bar{k}=k$, $\operatorname{char}K=0$. Assume $M_k$
admits a good filtration as $G_k$-module. If $(M_k)^{G_k}$ is Cohen-Macaulay then $(M_K)^{G_K}$ is
Cohen-Macaulay.
\end{lemma}

\begin{proof}
This is combination of Lemmas \ref{lem:cm} and \ref{lem:basechange}.
\end{proof}

\begin{proof}[Proof of Proposition \ref{prop:mainlift}]  Let
 $T_A(\alpha)$ be a lift of $T_k(\alpha)$ to a $(G,A)$-module (this is possible since the obstructions
to lifting are in 
$\Ext^2_{G_k}(T_k(\alpha),T_k(\alpha))$ which is zero). Put $M=T_A(\alpha)\otimes_A S$.
Since $S_k$ has a good filtration, the same is true for
$T_k(\alpha)\otimes_k S_k=M_k$. By hypothesis $(M_k )^{G_k}=M(T_k(\alpha))$ is Cohen-Macaulay.
Hence by Lemma \ref{lem:liftingcm} $(M_K)^{G_K}=M(T_A(\alpha)_K)$ is Cohen-Macaulay.

Since $\bar{K}=K$ and $K$ has characteristic zero $T_A(\alpha)_K=\bigoplus_{\beta\in X(H_K)^+} \nabla_K(\beta)^{\oplus u_\beta}$. 
Since the $u_\beta$ may be computed using characters (see e.g. \cite[\S I.10.9]{jantzen2007representations}) we have
$u_\beta=(T_k(\alpha):\nabla_k(\beta))$. Hence whenever $(T_k(\alpha):\nabla_k(\beta))\neq 0$,
$M(\nabla_K(\beta))$ is Cohen-Macaulay.
\end{proof}

\section{The Andersen-Jantzen spectral sequence for Tate cohomology}
\label{sec:appA}
The spectral sequence \cite[eq. (2)]{JA} relates the cohomology for the Frobenius kernels~$B_1$ and $G_1$.
In this section we give a version of this spectral sequence valid for Tate cohomology (Proposition \ref{AndersenJantzen} below).  
We will state the result in terms of derived functors instead of spectral sequences. The main result of this appendix
(Proposition \ref{AndersenJantzen} below) is used in \S\ref{calculations} and it will also be used in \cite{FFRT2}.
\subsection{Preliminaries on Tate cohomology}
Let $A$ be a finite dimensional self injective algebra over an algebraically closed field $k$. Then the injective and projective $A$-modules coincide. Let
$\underline{\Mod}(A)$ be the stable category of  (not
necessarily finite dimensional) $A$-modules. This is a compactly
generated  triangulated category with shift functor $M[1]=\Omega^{-1}M$ whose compact objects are the finite dimensional $A$-modules. We denote
the $\Hom$-spaces in $\underline{\Mod}(A)$ by $\underline{\Hom}_A$. We also write for $i\in \ZZ$:
\[
\underline{\Ext}^i_A(M,N)=\underline{\Hom}_A(M,N[i])=\underline{\Hom}_A(M,\Omega^{-i} N). 
\]
It is easy to see that for $i\ge 1$ we have 
\begin{equation}
\label{eq:nontate}
\underline{\Ext}^i_A(M,N)=\Ext^i_A(M,N).
\end{equation}

For $M\in \Mod(A)$ let $C(M)$ be complete resolution of $M$: an acyclic complex consisting of
injective $A$-modules which satisfies $Z^0C(M)=M$. 
 Such a resolution can be constructed by splicing together an injective and a projective
resolution of~$M$.
Below we will put 
\begin{multline}
\label{stableRHom}
\underline{\RHom}_A(M,N):=\Hom_A(M,C(N))\cong\Hom_A(C(M)[-1],N)\\
\cong\Hom_A(C(M),C(N))\in D(k),
\end{multline}
where the last $\Hom$ is the product total complex of the double complex 
$$
\Hom_A(C(M)_{-i},C(N)_j).
$$
The bifunctor
\[
\underline{\RHom}_A(-,-):\underline{\Mod}(A)^\circ\times \underline{\Mod}(A)\r D(k) 
\]
 is exact in both arguments and  satisfies 
\[
\underline{\Ext}^i_A(M,N)=\HHH^i(\underline{\RHom}_A(M,N)).
\]
Now let $F$ be a finite group scheme, and denote by $\Oscr(F)$ its coordinate ring, which is a finite dimensional commutative Hopf algebra. Then
$\Mod(F)=\Mod(\Oscr(F)^\vee)$. Since $\Oscr(F)^\vee$ is also a finite
dimensional Hopf algebra, it is self injective by~\cite[Theorem 2.1.3]{MR1243637}. Hence we may define
$\underline{\Mod}(F):=\underline{\Mod}(\Oscr(F)^\vee)$. The Tate cohomology of $M$ in $\underline{\Mod}(F)$
is defined as
\[
\widehat{H}^\bullet(F,M):=\underline{\Ext}^\bullet_F(k,M).
\]

\subsection{The Andersen-Jantzen spectral sequence}
\label{andersenjantzen}
In this section let $G$ be as in \S\ref{sec:prelim} but assume in addition
that the Steinberg representation $\St_1$ exists (see Proposition \ref{steinberg}).

 \begin{lemma} 
\label{functorialresolution} 
Let $M$ be a $G$-module. There exists maps $P(M)\twoheadrightarrow M\hookrightarrow I(M)$ of $G$-modules,
natural in $M$, such that $P(M)$, $I(M)$  are projective (or equivalently injective) as $G_1$-modules. Moreover if $M$ is finite dimensional
  then $P(M)$ and $I(M)$ may be assumed to be finite dimensional as well.
\end{lemma}

\begin{proof}
  It suffices to construct $P(k)$ since we may put $P(M)=P(k)\otimes_k
  M$ and  $I(M)=I(k)\otimes_k
  M$ with $I(k)=P(k)^\dur$. Since the Steinberg representation is self-dual (see Proposition \ref{steinberg}) there is a
  $G$-invariant bilinear form $ \St_1\otimes_k\St_1\r k $, and we put
  $P(k)=\St_1\otimes_k\St_1$. Since $\St_1$ is projective 
  as $G_1$-representation, so is $P(k)$. 
\end{proof}

\begin{corollary} \label{cor:complete}
Let $M$ be a $G$-module. There exists a complete $G_1$-resolution $C(M)$ of $M$ consisting of $G$-modules.
Moreover if $M$ is finite dimensional  then $C(M)^i$ may be assumed to be finite dimensional as well.
\end{corollary}

We need that $\underline{\RHom}_{G_1}(M,N)$ is well defined in the derived category of $G^{(1)}$-modules. To this end we use the following lemma.

\begin{lemma}\label{lem:C12}
Let $M$ be a $G$-module. 
Let $C_1(M)$ and $C_2(M)$ be two complete resolutions of $M$ as $G_1$-modules, consisting of $G$-modules. Then $C_1(M)$ and $C_2(M)$ are connected by a zigzag in
the category of acyclic complexes $C$ of $G$-modules satisfying $Z^0(C)=M$.
\end{lemma}

\begin{proof}
Let us write $P_i=C_i(M)_{\leq -1}$, $Q_i=C_i(M)_{\geq 0}$, $i=1,2$. 
First note that the left resolutions $P_1$, $P_2$ of $M$  can be connected by zig-zags in the usual way. 
Indeed, we have 
\[
\xymatrix{
P_1^{-2}\ar@{->>}[r]\ar@{.>}[d]&K_1\ar@{^{(}->}[r]\ar@{.>}[d]&P_1^{-1}\ar[r]\ar[d]& M\ar[r]\ar@{=}[d]& 0\\
P_1^{-2}\oplus P_2^{-2}\oplus P^{-2}\ar@{->>}[r]^-d&K\ar@{^{(}->}[r]&P_1^{-1}\oplus P_2^{-1}\ar[r]& M\ar[r]\ar@{=}[d]& 0\\
P_2^{-2}\ar@{->>}[r]\ar@{.>}[u]&K_2\ar@{^{(}->}[r]\ar@{.>}[u]&P_2^{-1}\ar[r]\ar[u]& M\ar[r]& 0\\
}
\]
where $P^{-2}$ is a $G$-module which is a projective as $G_1$-module such that $d$ is surjective and $G$-equivariant (e.g. $P^{-2}=P(K)$).  Note that the dotted maps are $G$-equivariant since the kernels are functorial. 

Dually, the same holds for $Q_1,Q_2$ and  we can then construct a zig-zag of complete resolutions 
\[
(P_1,Q_1)\rightarrow (P_2,Q_1)\leftarrow (P_2,Q_2).\qedhere
\]
\end{proof}

\begin{remark} By Lemma \ref{lem:C12} there is an identification of $\Hom_{G_1}(M,C_1(N))$
and $\Hom_{G_1}(M,C_2(N))$ in $D(G^{(1)})$. 
We omit the routine but somewhat tedious verification that this identification 
is independent of the chosen zig-zag connecting $C_1(N)$ and $C_2(N)$.
\end{remark}

We have
\begin{align*}
\underline{\RHom}_{G_1}(M,N)&=\Hom_{G_1}(M,C(N))\\
\underline{\RHom}_{B_1}(M,N)&=\Hom_{B_1}(M,C(N))
\end{align*}
and in particular $\underline{\RHom}_{G_1}(M,N)$ and $\underline{\RHom}_{B_1}(M,N)$ may be regarded as objects in $D(G^{(1)})$ and
$D(B^{(1)})$ respectively.
\begin{proposition}
\label{AndersenJantzen}
Let $M^\bullet$ be a bounded complex of $B$-modules. There are quasi-isomorphisms
of complexes of $G^{(1)}$-modules
\begin{align}
\RInd^{G^{(1)}}_{B^{(1)}} \RHom_{B_1}(k,M^\bullet)&\cong \RHom_{G_1}(k,\RInd^G_B M^\bullet)\label{identity1}\\
\RInd^{G^{(1)}}_{B^{(1)}} \underline{\RHom}_{B_1}(k,M^\bullet)&\cong \underline{\RHom}_{G_1}(k,\RInd^G_B M^\bullet)
\label{identity2}
\end{align}
\end{proposition}
\begin{proof}
 \eqref{identity1} follows by considering
pushforwards for the following
commutative diagram of stacks
\[
\xymatrix{
\bullet/B\ar[r]^F\ar[d] & \bullet/B^{(1)}\ar[d]\\
\bullet/G\ar[r]_F & \bullet/G^{(1)}
}
\]
taking into account $B/B_1=B^{(1)}$, $G/G_1=G^{(1)}$.

Let $V$ be a $G$-representation and let $M$ be a $B$-module. \eqref{identity1} has the following variant. 
\begin{equation}
\label{identity1_variant}
\RInd^{G^{(1)}}_{B^{(1)}} \RHom_{B_1}(V,M)\cong \RHom_{G_1}(V,\RInd^G_B M)
\end{equation}
To see this note that
\[
\RHom_{B_1}(V,M)=\RHom_{B_1}(k,\Hom(V,M))
\]
and
\begin{align*}
\RHom_{G_1}(V,\RInd^G_B M)&=\RHom_{G_1}(k,\Hom(V,\RInd^G_B M))\\
&=\RHom_{G_1}(k,\RInd^G_B \Hom(V,M))
\end{align*}
where the last equality follows from the fact that as $G$-representations we have
\[
\Hom(V,\RInd^G_B M)=\RInd^G_B\Hom(V,M)
\]
(given that $V$ is a $G$-representation). So \eqref{identity1_variant} follows from \eqref{identity1} applied with $M$ replaced by $\Hom(V,M)$.

\medskip

If $M$ is acyclic for $\Ind^G_B$ and $V$ is projective as $G_1$-representation (and hence also projective as $B_1$-representation) 
then we 
deduce from \eqref{identity1_variant}
\[
\RInd^{G^{(1)}}_{B^{(1)}} \Hom_{B_1}(V,M)\cong \Hom_{G_1}(V,\Ind^G_B M)
\]
In other words $\Hom_{B_1}(V,M)$ is acyclic for $\Ind^{G^{(1)}}_{B^{(1)}}$.

Now we consider \eqref{identity2}. Then we have to show 
\[
\RInd^{G^{(1)}}_{B^{(1)}} \Hom_{B_1}(C(k),M^\bullet)\cong \Hom_{G_1}(C(k),\RInd^G_B M^\bullet)
\]
Since $\Ind^G_B$ has bounded cohomological dimension we may assume that $M^\bullet$ is a bounded complex 
 of $B$-modules which are acyclic for $\Ind^G_B$. Then we have to show
\[
\RInd^{G^{(1)}}_{B^{(1)}} \Hom_{B_1}(C(k),M^\bullet)\cong \Hom_{G_1}(C(k),\Ind^G_B M^\bullet)
\]
By the above discussion the components of $ \Hom_{B_1}(C(k),M^\bullet)$ are acyclic for the left exact functor $\Ind^{G^{(1)}}_{B^{(1)}}$. 
So we must show 
\[
\Ind^{G^{(1)}}_{B^{(1)}} \Hom_{B_1}(C(k),M^\bullet)\cong \Hom_{G_1}(C(k),\Ind^G_B M^\bullet)
\]
We claim that this identity holds on the level of complexes. To see this we may replace
$C(k)$ and $M^\bullet$ by single representations $V$ and $M$. It then suffices to apply $H^0(-)$ to
\eqref{identity1_variant}.
\end{proof}
\providecommand{\bysame}{\leavevmode\hbox to3em{\hrulefill}\thinspace}
\providecommand{\MR}{\relax\ifhmode\unskip\space\fi MR }
\providecommand{\MRhref}[2]{%
  \href{http://www.ams.org/mathscinet-getitem?mr=#1}{#2}
}
\providecommand{\href}[2]{#2}

\end{document}